\documentclass[a4paper,10pt,reqno]{amsart}
\usepackage[utf8]{inputenc}
\usepackage{amsmath,amsfonts,amssymb,mathrsfs,latexsym,bbm,tikz}
\usepackage[all,cmtip]{xy}
\usepackage[top=1.5cm, bottom=2cm, left=2.5cm, right=2.5cm]{geometry}
\newtheorem{definition}{Definition}[section]
\newtheorem{theorem}[definition]{Theorem}
\newtheorem{lemma}[definition]{Lemma}
\newtheorem{coro}[definition]{Corollary}
\newtheorem{proposition}[definition]{Proposition}
\newtheorem{remark}[definition]{Remark}
\begin{document}








\begin{abstract}
We study the Milnor-Witt motives which are a finite direct sum of \(\mathbb{Z}(q)[p]\) and \(\mathbb{Z}/\eta(q)[p]\). We show that for MW-motives of this type, we can determine an MW-motivic cohomology class in terms of a motivic cohomology class and a Witt cohomology class. We define the motivic Bockstein cohomology and show that it corresponds to subgroups of Witt cohomology, if the MW-motive splits as above. As an application, we give the splitting formula of Milnor-Witt motives of Grassmannian bundles and complete flag bundles. This in particular shows that the integral cohomology of real complete flags has only \(2\)-torsions.
\end{abstract}

\title{Split Milnor-Witt Motives and its Applications to Fiber Bundles}
\author{Nanjun Yang}
\address{Nanjun Yang\\Yau Math. Sci. Center\\Jing Zhai\\Tsinghua University\\Hai Dian District\\Beijing China}
\email{ynj.t.g@126.com}
\thanks{The author would like to thank Baptiste Calm\`es and Jean Fasel for careful reading of the early version of this paper and to thank Baohua Fu for helpful discussions.}


\maketitle

\tableofcontents
\section{Introduction}
Throughout, our base field \(F\) is infinite perfect with \(char(F)\neq 2\), which is the assumption used in \cite{BCDFO}. Let \(S\) be the base scheme being smooth and separated over \(F\) and \(R\) be a commutative ring with identity, namely the coefficient ring, if no confusion arises.

The category of Milnor-Witt motives (abbr. MW-motives), defined by B. Calm\`es, F. D\'eglise and J. Fasel, is the Chow-Witt refinement of the classical category of motives defined by Voevodsky. In \cite{Y1}, we have shown that the Milnor-Witt motive of \(\mathbb{P}^n\) \emph{splits}, i.e., is a direct sum of \(\mathbb{Z}(i)[2i]\) and \(\mathbb{Z}/\eta(i)[2i]\), where \(\eta\) is the Hopf element (see \cite[Definition 4.1]{Y1}), deducing the projective bundle theorem from this computation. The aim of this paper is to push this split pattern further and show that the Grassmannians and complete flags behave in the same way, as well as corresponding fiber bundles.

More precisely, we say that an MW-motive quasi-splits if it is a finite direct sum of \(\mathbb{Z}(q)[p]\) and \(\mathbb{Z}/\eta(q)[p]\) where \(p,q\in\mathbb{Z}\). We say that it splits if \(p=2q\) for all summands. Denote by \(H_M^{*,*}\) (resp. \(H_{MW}^{*,*}\)) the (resp. MW-) motivic cohomologies (\cite{BCDFO}), whereas \(H_{MW}^{p,q}[\eta^{-1}]\) is identified with the Witt cohomology \(H^{p-q}(-,\textbf{W})\) as shown in Proposition \ref{localization}. Define
\[H_{\eta}^{p,q}=Im(H_{MW}^{p,q}\longrightarrow H_{MW}^{p,q}[\eta^{-1}])\]
and \(\iota\) to be quotient map of \(\tau[1]\), where \(\tau\) is the motivic Bott element (\cite[p.22]{B}).

In \cite[Definition 10.4.3]{P}, Powell defined the total Steenrod power
\[\mathcal{P}_2:K_M(\mathbb{Z}/2,n)\wedge B_{gm}(\mathbb{Z}/2)\longrightarrow K_M(\mathbb{Z}/2,2n)\]
where \(K_M(\mathbb{Z}/2,n)\) is the Eilenberg Maclane space of \(H^{2n,n}_M(-,\mathbb{Z}/2)\). Denote by \(Sq^1:H^{*,*}_M(-,\mathbb{Z}/2)\longrightarrow H^{*+1,*}_M(-,\mathbb{Z}/2)\) the mod-2 Bockstein map. Then for any simplicial pointed sheaf \(\mathcal{F}_{\bullet}\), we define
\[\theta:H^{2n,n}_M(\mathcal{F}_{\bullet},\mathbb{Z}/2)\longrightarrow H^{4n,2n}_M(\mathcal{F}_{\bullet}\wedge B_{gm}(\mathbb{Z}/2),\mathbb{Z}/2)\]
where \(\theta=\mathcal{P}_2\circ(-\wedge B_{gm}(\mathbb{Z}/2))\). Then \cite[Theorem 6.10]{V} gives us the decomposition
\[H^{*,*}_M(\mathcal{F}_{\bullet}\times B_{gm}(\mathbb{Z}/2),\mathbb{Z}/2)=H^{*,*}_M(\mathcal{F}_{\bullet},\mathbb{Z}/2)[a,b]/(a^2=\tau b+\rho a)\]
where \(\rho=Sq^1(\tau)=-1\in F^*/2=H^{1,1}_M(F,\mathbb{Z}/2)\), \(a\in H^{1,1}_M(F,\mathbb{Z}/2)\) and \(b=Sq^1(a)\in H^{2,1}_M(F,\mathbb{Z}/2)\). The map \(\theta\) and the bistability then give the construction of Steenrod operations
\[Sq^{2i}:H^{p,q}_M(\mathcal{F}_{\bullet},\mathbb{Z}/2)\longrightarrow H^{p+2i,q+i}_M(\mathcal{F}_{\bullet},\mathbb{Z}/2)\]
\[Sq^{2i+1}=Sq^1\circ Sq^{2i}:H^{p,q}_M(\mathcal{F}_{\bullet},\mathbb{Z}/2)\longrightarrow H^{p+2i+1,q+i}_M(\mathcal{F}_{\bullet},\mathbb{Z}/2).\]
Suppose \(X\in Sm/F\) and \(\mathscr{L}\in Pic(X)\). We define \(Sq^i_{\mathscr{L}}\) to be operations
\[Sq^{2i}:H^{p+2,q+1}_M(Th(\mathscr{L}),\mathbb{Z}/2)\longrightarrow H^{p+2i+2,q+i+1}_M(Th(\mathscr{L}),\mathbb{Z}/2)\]
\[Sq^{2i+1}:H^{p+2,q+1}_M(Th(\mathscr{L}),\mathbb{Z}/2)\longrightarrow H^{p+2i+3,q+i+1}_M(Th(\mathscr{L}),\mathbb{Z}/2).\]
or simply \(Sq^i\) if \(\mathscr{L}=O_X\). The \(Sq^1_{\mathscr{L}}\) is independent of \(\mathscr{L}\) and the \(Sq^2_{\mathscr{L}}\) depends on the class of \(\mathscr{L}\) in \(Pic(X)/2\) by Proposition \ref{id} and \cite[Theorem 6.2]{Y}.

We define
\[Ker_{\tau}(Sq^2)_{p,q}(\mathcal{F}_{\bullet})=\{u\in H^{p,q}_M(\mathcal{F}_{\bullet},\mathbb{Z}/2)|Sq^2(u)\equiv 0 (\textrm{mod }\tau)\}\]
\[Im_{\tau}(Sq^2)_{p,q}(\mathcal{F}_{\bullet})=\{u\in H^{p,q}_M(\mathcal{F}_{\bullet},\mathbb{Z}/2)|u\in Im(Sq^2) (\textrm{mod }\tau)\}.\]
By Lemma \ref{tau1}, for any \(\mathscr{L}\in Pic(X)\) and \(p,q\in\mathbb{Z}\), we define the motivic Bockstein cohomology
\[E^{p,q}(X,\mathscr{L})=\frac{Ker_{\tau}(Sq^2)_{p+2,q+1}}{Im_{\tau}(Sq^2)_{p+2,q+1}}(Th(\mathscr{L})).\]
Denote by \(\textbf{I}(F)\) the kernel of the rank morphism \(\textbf{W}(F)\longrightarrow\mathbb{Z}/2\) of the Witt group of \(F\). The central theorem of quasi-split MW-motives is the following (see Theorem \ref{sq}):
\begin{theorem}\label{sqintro}
Suppose that \(X\in Sm/F\), \(p,q\in\mathbb{Z}\), \(\mathscr{L}\in Pic(X)\) and that \(Th(\mathscr{L})\) quasi-splits in \(\widetilde{DM}(pt,\mathbb{Z})\).
\begin{enumerate}
\item There is a natural map
\[\Delta:H^{p,q}_{\eta}(X,\mathscr{L})\longrightarrow E^{p,q}(X,\mathscr{L})\]
which induces an isomorphism of \(\mathbb{Z}/2\)-vector spaces
\[\delta:H^{p,q}_{\eta}(X,\mathscr{L})/H^{p+1,q+1}_{\eta}(X,\mathscr{L})\cong E^{p,q}(X,\mathscr{L}).\]
\item We have a Cartesian square
\[
	\xymatrix
	{
		H^{p,q}_{MW}(X,\mathbb{Z},\mathscr{L})\ar[r]\ar[d]	&Ker(\iota\circ Sq_{\mathscr{L}}^2\circ\pi)_{p,q}\ar[d]\\
		H^{p,q}_{\eta}(X,\mathscr{L})\ar[r]^{\Delta}			&E^{p,q}(X,\mathscr{L}).
	}
\]
In particular, if \(Th(\mathscr{L})\) splits, there is an isomorphism
\[H^p(X,\textbf{W}(\mathscr{L}))\otimes_{\textbf{W}(F)}\mathbb{Z}/2=E^{2p,p}(X,\mathscr{L})\]
and a decomposition
\[\widetilde{CH}^{*}(X,\mathscr{L})\cong\textbf{I}(F)\cdot H^*(X,\textbf{W}(\mathscr{L}))\oplus Ker(Sq^2_{\mathscr{L}}\circ\pi)_*.\]
\end{enumerate}
\end{theorem}
The first statement is the algebraic geometric explanation of the fact (see \cite[Theorem 10.3]{Mc}) that for any topological space \(X\) whose singular cohomologies are finitely generated and have only \(2\)-torsions, we have
\[\frac{Ker(Sq^1)_{{*}}}{Im(Sq^1)_{{*}}}=H^{{*}}(X,\mathbb{Z})_{free}\otimes\mathbb{Z}/2.\]
The second statement shows that giving an MW-motivic cohomology class of quasi-split MW-motives is equivalent to giving both a motivic cohomology class and Witt cohomology class, such that they lead to the same motivic Bockstein cohomology. This gives a much further calculation of MW-motivic cohomologies than that of \cite{HW}. (see Remark \ref{better})

Define \(\widetilde{DM}^{eff}(S,R)\) (resp. \(\widetilde{DM}(S,R)\)) to be the category of effective (resp. stable) MW-motives over \(S\) with coefficients in \(R\) (see Section \ref{Split Milnor-Witt motives}, \cite[\S 2]{Y1} or \cite[\S 3]{BCDFO}) and \(((i)):=(i)[2i]\). Moreover, define \(\widetilde{DM}_{\eta}=\widetilde{DM}(pt,\mathbb{Z})[\eta^{-1}]\) (So \(\mathbb{Z}(1)[1]=\mathbb{Z}\)), whose (tensor) unit represents the cohomology of Witt sheaves. Denote by \(\mathbb{Z}(X)\) the motive of \(X\in Sm/S\) in previous categories and by \(Th(\mathscr{E})\) the Thom space of a vector bundle \(E\). In order to explicitly decompose a split MW-motive, it suffices to write down its decompositions in both \(DM\) and \(\widetilde{DM}_{\eta}\) (see Lemma \ref{uniqueness}).

A Young diagram \(T\) (a series of rows of boxes with decreasing length, see the beginning of Section \ref{MW-Motivic Decomposition of Grassmannians}) is called untwisted (resp. twisted) if it is filled by the chessboard pattern such that the first box in the first row is black (resp. white). Define
\[\begin{array}{cc}A(T)=\{T'\geq T|T'\setminus T=\textrm{ a white box}\}&D(T)=\{T'\leq T|T\setminus T'=\textrm{ a white box}\}\end{array}.\]
It is called irredundant (resp. full) if \(D(T)=\emptyset\) (resp. \(A(T)=\emptyset\)). If it is both irredundant and full, it is called even (see Definition \ref{Young}, \ref{even} and \ref{Young1}). Hence every Young diagram \(\Lambda\) has \(|A(\Lambda)|\) positions where a white box could be added. The symbol \(\Lambda_{i_1,\cdots,i_l}\) (\(i_1<\cdots<i_l\)) denotes the diagram obtained from adding a white box added at each of the \(i_1,\cdots,i_l\)-th positions of \(\Lambda\) (see Definition \ref{Young1}).

The stage being set, we can now state the following theorem, which describes the decomposition of the Grassmannians \(Gr(k,n)\) in the category of MW-motives (see Theorem \ref{Grass1}), where \(Gr(k,n)\) parametrizes \(k\)-dimensional subspaces inside an \(n\)-dimensional \(F\)-vector space.
\begin{theorem}
\begin{enumerate}
\item[(a)] We have
\[\mathbb{Z}(Gr(k,n))\cong\bigoplus_{\Lambda\textrm{ even}}\mathbb{Z}((|\Lambda|))\oplus\bigoplus_{\Lambda\textrm{ irred. not full}, i_1>1}\mathbb{Z}/{\eta}((|\Lambda_{i_1,\cdots,i_l}|)).\]
\item[(b)] Suppose that all diagrams are twisted. We have
\[Th(O_{Gr(k,n)}(1))\cong\bigoplus_{\Lambda\textrm{ even}}\mathbb{Z}((|\Lambda|+1))\oplus\bigoplus_{\Lambda\textrm{ irred. not full}, i_1>1}\mathbb{Z}/{\eta}((|\Lambda_{i_1,\cdots,i_l}|+1)).\]
\end{enumerate}
\end{theorem}
We derive from this result the Grassmannian bundle theorem, by the method developed in \cite{Y1} (see Theorem \ref{Grassbdl}). In the statement, the symbol \(Gr(k,\mathscr{E})\) denotes the variety of quotient bundles of \(\mathscr{E}\) of rank \(k\).
\begin{theorem}
Let \(S\in Sm/F\) and let \(X\in Sm/S\) be quasi-projective, \(\mathscr{L}\in Pic(X)\) and \({\mathscr{E}}\) be a vector bundle of rank \(n\) over \(X\). Denote by \(p:Gr(k,\mathscr{E})\longrightarrow X\) the structure map.
\begin{enumerate}
\item We have
\[\mathbb{Z}(Gr(k,\mathscr{E}))/\eta\cong\bigoplus_{\Lambda\textrm{ \((k,n)\)-truncated}}\mathbb{Z}(X)/\eta((|\Lambda|))\]
in \(\widetilde{DM}(S,\mathbb{Z})\).
\item If \(k(n-k)\) is even, we have
\[Th(p^*\mathscr{L})\cong\bigoplus_{\Lambda\textrm{ even}}Th(\mathscr{L})((|\Lambda|))\oplus\bigoplus_{\Lambda\textrm{ irred. not full}, i_1>1}\mathbb{Z}(X)/{\eta}((|\Lambda_{i_1,\cdots,i_l}|+1))\]
in \(\widetilde{DM}(S,\mathbb{Z})\).
\item If both \(k\) and \(n\) are even, we have
\begin{footnotesize}
\[\begin{array}{c}Th(p^*\mathscr{L}\otimes O(1))\cong\bigoplus_{\Lambda=\sigma_{n-k}T}Th(det(\mathscr{E})^{\vee}\otimes\mathscr{L})((|\Lambda|))\oplus\bigoplus_{\Lambda=\sigma_{1^k}T}Th(\mathscr{L})((|\Lambda|))\\\oplus\bigoplus_{i_1>1}\mathbb{Z}(X)/{\eta}((|\Lambda_{i_1,\cdots,i_l}|+1))\end{array}\]
\end{footnotesize}
in \(\widetilde{DM}(S,\mathbb{Z})\), where \(T\) is completely even.
\item If \(n-k\) is odd, we have
\[Th(p^*\mathscr{L}\otimes O(1))\cong\bigoplus_{\Lambda\textrm{ even}}Th(\mathscr{L})((|\Lambda|))\oplus\bigoplus_{\Lambda\textrm{ irred. not full}, i_1>1}\mathbb{Z}(X)/{\eta}((|\Lambda_{i_1,\cdots,i_l}|+1))\]
in \(\widetilde{DM}(S,\mathbb{Z})\).
\item If \(k\) and \(n\) are odd, we have
\[Th(p^*\mathscr{L}\otimes O(1))\cong\bigoplus_{\Lambda\textrm{ even}}Th(\mathscr{L}\otimes det(\mathscr{E})^{\vee})((|\Lambda|))\oplus\bigoplus_{\Lambda\textrm{ irred. not full}, i_1>1}\mathbb{Z}(X)/{\eta}((|\Lambda_{i_1,\cdots,i_l}|+1))\]
in \(\widetilde{DM}(S,\mathbb{Z})\).
\item If \(k\) is odd, \(n\) is even and \(e(\mathscr{E})=0\in\widetilde{CH}^n(X,det(\mathscr{E})^{\vee})\), there is an isomorphism
\begin{footnotesize}
\[\begin{array}{c}Th(p^*\mathscr{L})\cong\bigoplus_{\Lambda=\mathcal{R}\cdot T}Th(\mathscr{L}\otimes det(\mathscr{E})^{\vee})((|\Lambda|))\oplus\bigoplus_{\Lambda=T}Th(\mathscr{L})((|\Lambda|))\\\oplus\bigoplus_{\Lambda\textrm{ irred. not full}, i_1>1}\mathbb{Z}(X)/{\eta}((|\Lambda_{i_1,\cdots,i_l}|+1))\end{array}\]
\end{footnotesize}
in \(\widetilde{DM}(S,\mathbb{Z})\), where \(T\) is completely even and \(\mathcal{R}\) is the longest hook.
\end{enumerate}
\end{theorem}
Let us note that the (6) is the hardest part of the theorem, since the vanishing of Euler class is a global condition.

Last but not least, we show that the complete flag also splits as an MW-motive. Moreover, we have a complete flag bundle theorem if all Pontryagin classes and the Euler class vanish in Witt cohomology (see Theorem \ref{flag}). For any vector bundle \(\mathscr{E}\), \(Fl(\mathscr{E})\) parametrizes complete flags
\[0\subseteq\mathscr{E}_1\subseteq\cdots\subseteq\mathscr{E}_n=\mathscr{E}\]
where \(rk(\mathscr{E}_i)=i\) and successive quotients are line bundles.
\begin{theorem}
For any \(\mathscr{L}\in Pic(Gr(1,\cdots,n))/2\), \(Th(\mathscr{L})\) splits as an MW-motive. Moreover, we have
\begin{enumerate}
\item If \(n\) is odd, define \(deg(a)=4a-1\). We have
\[Th(\mathscr{L})=\begin{cases}\bigoplus_{1\leq t\leq\frac{n-1}{2}}\bigoplus_{1\leq a_1<\cdots<a_t\leq\frac{n-1}{2}}\mathbb{Z}[1+\sum_sdeg(a_s)]&\textrm{if }\mathscr{L}=0\\0&\textrm{else}\end{cases}.\]
in \(\widetilde{DM}_{\eta}\).
\item If \(n\) is even, define
\[deg(a)=\begin{cases}4a-1&1\leq a\leq\frac{n}{2}-1\\n-1&a=\frac{n}{2}\end{cases}.\]
We have
\[Th(\mathscr{L})=\begin{cases}\bigoplus_{1\leq t\leq\frac{n}{2}}\bigoplus_{1\leq a_1<\cdots<a_t\leq\frac{n}{2}}\mathbb{Z}[1+\sum_sdeg(a_s)]&\textrm{if }\mathscr{L}=0\\0&\textrm{else}\end{cases}.\]
in \(\widetilde{DM}_{\eta}\).
\end{enumerate}
So \(Th(\mathscr{L})\) are mutually isomorphic in \(\widetilde{DM}(pt,\mathbb{Z})\) if \(\mathscr{L}\neq 0\). Denote by \(G\) this common object.

Suppose that \(X\in Sm/S\) is quasi-projective, \(\mathscr{M}\in Pic(Fl(\mathscr{E}))/2\) and that \(\mathscr{E}\) is a vector bundle of rank \(n\) on \(X\). Denote by \(p:Fl(\mathscr{E})\longrightarrow X\) the structure map. We have
\[Th(\mathscr{M})\cong\begin{cases}Th(\mathscr{L})\otimes\mathbb{Z}(Gr(1,\cdots,n))&\begin{array}{c}p_i(\mathscr{E}),e(\mathscr{E})=0\in H^{{*}}(X,\textbf{W}(-))\\\textrm{for any }i>0\\\mathscr{M}=p^*\mathscr{L},\mathscr{L}\in Pic(X)/2\end{array}\\\mathbb{Z}(X)\otimes G&\mathscr{M}\notin Pic(X)/2\end{cases}\]
in \(\widetilde{DM}(S,\mathbb{Z})\).
\end{theorem}
We can write down the generators of \(H^{{*}}(Gr(1,\cdots,n),\textbf{W})\) in the aid of Theorem \ref{sqintro} (see Remark \ref{flaggen} and the list below for notations).
\begin{proposition}\label{bintro}
Suppose that \(1\leq a\leq\lfloor\frac{n}{2}\rfloor\). If \(n\) is odd or \(n\) is even and \(a<\frac{n}{2}\), define
\[T_a=h_{2a}(x_1,\cdots,x_{n-2a})h_{2a-1}(x_1,\cdots,x_{n-2a+1})+u(x_1,\cdots,x_{n-2a})\]
where \(u\) satisfies \(Sq^2(u)=h_{2a}(x_1,\cdots,x_{n-2a})^2\). If \(n\) is even, define
\[T_{\frac{n}{2}}=x_1^{n-1}.\]

Then we have
\[E^{{*}}(Fl(F^{\oplus n}))=\wedge[\{T_a\}].\]
\end{proposition}
Here the \(\{x_i\}\) are the first Chern classes of quotient line bundles of a complete flag and \(h_i(x_1,\cdots,x_j)\) is the complete homogenous polynomial of degree \(i\). Our computation is compatible with the Cartan model (see \cite[Theorem 1]{T}) and \cite{CF}. Moreover, it answers the conjecture given in \cite{M}, namely, the \(H^*(Gr(1,\cdots,n)(\mathbb{R}),\mathbb{Z})\) has only \(2\)-torsions. Together with Proposition \ref{bintro}, the \(H^*(Gr(1,\cdots,n)(\mathbb{R}),\mathbb{Z})\) is computed by Remark \ref{cell}.

There has been many works related to the Witt or Chow-Witt groups of Grassmannians and complete flags, for example \cite{BC}, \cite{CF}, \cite{W1} and \cite{W2}. The article \cite{BC} was devoted to Witt groups and defined the notion of even Young diagrams. In \cite{W1}, the Chow-Witt ring of Grassmannians was computed in terms of generators and relations as an algebra. In \cite{CF}, Calm\`es and Fasel pointed out some cases where the twisted Witt group of the complete flag of some linear algebraic group vanishes, but they did not compute the Witt group without twist. An explanation of these results in terms of MW-motives (resp. MW-motivic cohomologies) is completely new, needless to say any characterization of their fiber bundles.

For convenience, we give a list of frequently used notations in this paper:
\[\begin{array}{c|c}
pt			&Spec(F)\\
Sm/{S}	&\textrm{Smooth and separated schemes over \(S\)}\\
R(X)		&\textrm{Motive of \(X\) with coefficients in \(R\)}\\
Th({\mathscr{E}})		&\textrm{Thom space of \({\mathscr{E}}\)}\\
((i))		&(i)[2i]\\
C(f)		&\textrm{Mapping cone of \(f\)}\\
\pi		&\textrm{The reduction modulo \(2\) map \(CH\longrightarrow Ch\)}\\
{h}		&\textrm{The hyperbolic quadratic form}\\
\eta		&\textrm{The Hopf element}\\
\tau		&\textrm{The motivic Bott element}\\
\iota		&\textrm{The quotient map of }\tau[1]\\
A/\eta	&A\otimes C(\eta)\\
|D|		&\textrm{Cardinality of the set \(D\)}\\
|\Lambda|&\textrm{Number of boxes in the Young diagram \(\Lambda\)}\\
\mathcal{R}&\textrm{The orientation class of \(Gr(k,n)\)}\\
{[-,-]}	&Hom_{\mathcal{SH}(F)}(-,-)\\
{[-,-]_{K}}&\begin{array}{c}Hom_{DM_K(pt,\mathbb{Z})}(-,-)\\K_*=K_*^{MW},K_*^M,K_*^W,K_*^{M}/2\end{array}\\
h_i(x_1,\cdots,x_n)	&\textrm{The complete homogeneous polynomial of degree \(i\)}\\
e_i(x_1,\cdots,x_n)	&\textrm{The elementary homogeneous polynomial of degree \(i\)}
\end{array}.\]

\section{Split Milnor-Witt motives}\label{Split Milnor-Witt motives}
In this section, we discuss the splitting pattern in MW-motives. Let us briefly recall the language of four motivic theories established in \cite[\S 2]{Y1} and \cite[\S 3]{BCDFO}. Let \(Sm/S\) be the category of smooth and separated schemes over \(S\) and \(K_*\) be one of the homotopy modules \(K_*^{MW}, K_*^M, K_*^W, K_*^{M}/2\) (\(K_*^W=\textbf{I}^*\)). For every \(X\in Sm/F\), if \(T\subset X\) is a closed set and \(n\in\mathbb{N}\), define
\[C_{RS,T}^n(X;{K}_m;\mathscr{L})=\bigoplus_{y\in X^{(n)}\cap T}{K}_{m-n}(k(y),\Lambda_y^*\otimes_{k(y)}\mathscr{L}_y),\]
where \(X^{(n)}\) means the points of codimension \(n\) in \(X\) and \(\Lambda_y^*\) is the highest exterior power of \(T_y\). Then \(C_{RS,T}^*(X;{K}_m;\mathscr{L})\) form a complex (see \cite{Mo}), which is called the Rost-Schmid complex with support on \(T\). Define (see \cite{BCDFO} and \cite{Y1})
\[{K}CH^n_T(X,\mathscr{L})=H^n(C^*_{RS,T}(X;{K}_n;\mathscr{L})).\]
Thus we see that \({K}CH=\widetilde{CH}, CH, Ch\) when \({K_*}=K^{MW}_*, K^M_*, K^M_*/2\) respectively.

Suppose that \(\mathscr{E}\) is a vector bundle of rank \(n\) on \(X\). It admits the Euler class \(e(\mathscr{E})\in\widetilde{CH}^n(X,det(\mathscr{E})^{\vee})\) (see \cite[D\'efinition 13.2.1]{F1}) and Pontryagin classes (or Borel classes) \(p_i(\mathscr{E})\in\widetilde{CH}^{2i}(X)\) (see \cite[Definition 2.2.6]{DF}). So do the cohomology theories \(CH^{{*}}\), \(Ch^{{*}}\) and \(H^{{*}}(-,\textbf{W}(-))\).

For any \(S\in Sm/F\) and \(X, Y\in Sm/S\), define \(\mathscr{A}_S(X,Y)\) to be the poset of closed subsets in \(X\times_SY\) such that each of its component is finite over \(X\) and of dimension \(dimX\). Suppose that \(R\) is a commutative ring. Let
\[{K}Cor_S(X,Y,R):=\varinjlim_T{K}CH_T^{dimY-dimS}(X\times_SY,\omega_{X\times_SY/X})\otimes_{\mathbb{Z}}R,\]
be the finite correspondences between \(X\) and \(Y\) over \(S\) with coefficients in \(R\), where \(T\in\mathscr{A}_S(X,Y)\) and \(\omega_{X\times_SY/X}\) is the relative canonical bundle. Hence \({K}Cor\) just means \(\widetilde{Cor}\), \(Cor\) and \(WCor\)  (see \cite[\S 3]{BCDFO}) when \({K_*}=K^{MW}_*, K^M_*, K^W_*\) respectively. This produces an additive category \({K}Cor_S\) whose objects are the same as \(Sm/S\) and whose morphisms are defined above. There is a functor \(Sm/S\longrightarrow {K}Cor_S\) sending a morphism to its graph.

We say that an abelian Nisnevich sheaf over \(Sm/S\) is a sheaf with \({K}\)-transfers if it admits extra functoriality from \({K}Cor_S\) to \(Ab\). For any smooth scheme $X$, let $R(X)$ be the representable sheaf with \({K}\)-transfers of \(X\), which is called the motive of \(X\) in \(R\)-coefficients. The Tate twist is denoted by \(R(1):=R(\mathbb{G}_m^{\wedge1})[-1]\).

Denote by \(C_K(S,R)\) (resp. \(D_{{K}}(S,R)\)) the (resp. derived) category of cochain complexes of sheaves with \({K}\)-transfers. Define the category of effective \({K}\)-motives over \(S\) with coefficients in \(R\)
\[DM_{K}^{eff}(S,R)=D_{{K}}(S,R)[(R(X\times\mathbb{A}^1)\longrightarrow R(X))^{-1}].\]
Furthermore, the category \(DM_{K}(S,R)\) of stable \(K\)-motives is the homotopy category of \(\mathbb{G}_m\)-spectra of \(C_K(S,R)\), with respect to stable \(\mathbb{A}^1\)-equivalences. When \(K_*=K^{MW}_*\) (resp \(K_*=K^M_*\)), we will write \(DM_K\) as \(\widetilde{DM}\) (resp. \(DM\)) conventionally. We have full faithfull embedding of categories \(DM^{eff}(pt,\mathbb{Z})\subseteq DM(pt,\mathbb{Z})\) and \(\widetilde{DM}^{eff}(pt,\mathbb{Z})\subseteq \widetilde{DM}(pt,\mathbb{Z})\) by cancellation (see \cite[Proposition 2.5]{Y1}).

There are functorialities of \(DM_{K}^{eff}(S,R)\) (resp. \(DM_K(S,R)\)) with respect to \(K\), \(S\) and \(R\) (see \cite[Proposition 2.1,2.2,2.3]{Y1}).

The stable \(\mathbb{A}^1\)-homotopy category \(\mathcal{SH}(F)\) (see \cite{Mo1}) is defined to be the homotopy category of \(\mathbb{P}^1\)-spectra of Nisnevich simplicial sheaves, localized with respect to morphisms \(X\times\mathbb{A}^1\longrightarrow X\).

Recall that in \cite{B}, T. Bachmann defined the motivic cohomology spectra \(H\widetilde{\mathbb{Z}}, H_{\mu}\mathbb{Z}, H_W\mathbb{Z}, H_{\mu}\mathbb{Z}/2\) as the effective cover of the homotopy modules \(K_*^{MW}, K_*^{M}, K_*^{W}, K_*^M/2\), respectively. The readers may also refer to \cite[\S 4]{Y1}.

We set \([-,-]=Hom_{\mathcal{SH}(F)}(-,-)\) and \([-,-]_K=Hom_{DM_K(pt,\mathbb{Z})}(-,-)\) for convenience (see \cite[\S 2]{Y1}). Suppose that \({\mathscr{E}}\) is a vector bundle on \(X\), define \(Th({\mathscr{E}})=R({\mathscr{E}})/R({\mathscr{E}}^{\times})\).
\begin{definition}
Suppose that \(X\in Sm/F\) and that \(\mathscr{L}\in Pic(X)\). Define
\[H^{p,q}_{MW}(X,\mathbb{Z},\mathscr{L})=[Th(\mathscr{L}),\mathbb{Z}(q+1)[p+2]]_{MW}.\]
\end{definition}
\begin{proposition}\label{coh}
Suppose \(p,q\in\mathbb{Z}\). We have
\[[\mathbb{Z}/\eta,\mathbb{Z}(q)[p]]_{MW}=\begin{cases}[\mathbb{Z}/\eta,\mathbb{Z}(q)[p]]_M&p\neq q\textrm{ or }p=q<0\\2K_q^M(F)\oplus H^{p-2,q-1}_M(F,\mathbb{Z})&p=q\geq 0\end{cases}.\]
\end{proposition}
\begin{proof}
We will freely use the results in \cite[Proposition 5.3, 5.6]{Y1}. Recall that we have
\[H^{p,q}_{MW}(F,\mathbb{Z})=H^{p,q}_M(F,\mathbb{Z})\]
if \(p\neq q\) by \cite[Introduction, (E)]{BCDFO}.

We have an exact sequence
\begin{equation}\label{seq}\xrightarrow{\eta}H^{p-2,q-1}_{MW}(F,\mathbb{Z})\longrightarrow[\mathbb{Z}/\eta,\mathbb{Z}(q)[p]]_{MW}\longrightarrow H^{p,q}_{MW}(F,\mathbb{Z})\xrightarrow{\eta}\end{equation}
and
\[[\mathbb{Z}/\eta,\mathbb{Z}(q)[p]]_M=H^{p,q}_M(F,\mathbb{Z})\oplus H^{p-2,q-1}_M(F,\mathbb{Z}).\]
If \(p\neq q,q+1\), the statement follows from the Five lemma. If \(p=q+1\), the statement follows from the exact sequence
\[K^{MW}_q(F)\xrightarrow{\eta}K^{MW}_{q-1}(F)\longrightarrow K^M_{q-1}(F)\longrightarrow 0.\]
If \(p=q<0\), in \eqref{seq}, the first term is zero and the last arrow is an isomorphism so both sides vanish. If \(p=q\geq 0\), in \eqref{seq}, the first arrow is zero and the kernel of the last arrow is \(2K_q^M(F)\). So we have a commutative diagram with exact rows
\[
	\xymatrix
	{
		0\ar[r]	&H^{p-2,q-1}_{MW}(F,\mathbb{Z})\ar[r]\ar[d]_{\cong}	&[\mathbb{Z}/\eta,\mathbb{Z}(q)[p]]_{MW}\ar[r]\ar[d]	&2K_q^M(F)\ar[r]\ar[d]	&0\\
		0\ar[r]	&H^{p-2,q-1}_M(F,\mathbb{Z})\ar[r]								&[\mathbb{Z}/\eta,\mathbb{Z}(q)[p]]_M\ar[r]				&K_q^M(F)\ar[r]			&0
	}.
\]
The second row splits hence does the first row.
\end{proof}
We have a distinguished triangle
\[H_{\mu}\mathbb{Z}/2[2]\xrightarrow{\tau[1]}H_{\mu}\mathbb{Z}/2\wedge\mathbb{P}^1\xrightarrow{\iota}C(\tau[1])\longrightarrow\cdots[1]\]
where \(\tau\) the motivic Bott element.
\begin{proposition}\label{id}
The composite
\[(u,v):H_{\mu}\mathbb{Z}/2\oplus H_{\mu}\mathbb{Z}/2[2]=H_W\mathbb{Z}/\eta\xrightarrow{\partial}H_W\mathbb{Z}\wedge\mathbb{P}^1\longrightarrow H_{\mu}\mathbb{Z}/2\wedge\mathbb{P}^1\]
is equal to \((Sq^2,\tau[1])\), where \(\partial\) is the boundary map of \(\eta\).
\end{proposition}
\begin{proof}
The statement for \(v\) was proved in \cite[Lemma 20]{B}. On the other hand, the \cite{V} and \cite{MKO} showed that the map \(u\) is of the form
\[c\cdot Sq^1+d\cdot Sq^2,\]
where \(c\in H^{1,1}_M(F,\mathbb{Z}/2)\) and \(d\in H^{0,0}_M(F,\mathbb{Z}/2)\). In \cite[Theorem 4.13]{Y1}, we showed that \(d=1\), so we are going to show \(c=0\).

By \cite[Lemma 9.3.3]{P} and \cite[Theorem 6.10]{V}, there are \(a\in H^{1,1}_M(B_{gm}(\mathbb{Z}/2),\mathbb{Z}/2)\) and \(b\in H^{2,1}_M(B_{gm}(\mathbb{Z}/2),\mathbb{Z}/2)\) such that
\[Sq^1(a)=b, Sq^2(a)=0.\]
So it suffices to show that \(u(a)=0\) since \(b\) is \(H^{1,1}\)-torsion free. By \cite[Proposition 8.2.5]{P}, the \(O(2)^{\times}\) over \(\mathbb{P}^{\infty}\) is a model for \(B_{gm}(\mathbb{Z}/2)\). By \cite[Proposition 4.5]{Y1}, the map \([O(2)^{\times},u(1)[1]]\) is identified with a composite
\[H^0(O(2)^{\times},K_1^{M}/2)\longrightarrow H^1(O(2)^{\times},K_2^W)\longrightarrow H^1(O(2)^{\times},K_2^{M}/2).\]
We claim that
\[H^1(O(2)^{\times},K_2^W)=0\]
so the statement is proved. We have a Gysin triangle
\[\mathbb{Z}(O(2)^{\times})\longrightarrow\mathbb{P}^{\infty}\xrightarrow{e(O(2))}\mathbb{P}^{\infty}(1)[2]\longrightarrow\cdots[1]\]
in \(\widetilde{DM}\). By using the decomposition in \cite[Theorem 5.11]{Y1}, we see that the second arrow is the following map
\begin{equation}\label{map}
	\xymatrixcolsep{1pc}
	\xymatrix
	{
		\mathbb{Z}\oplus		&\mathbb{Z}/\eta((1))\oplus\ar[ld]_h\ar[d]_{2p}	&\mathbb{Z}/\eta((3))\oplus\ar[ld]_{h_1}\ar[d]_{2p}	&\cdots\\
		\mathbb{Z}((1))\oplus	&\mathbb{Z}/\eta((2))\oplus							&\mathbb{Z}/\eta((4))\oplus									&\cdots
	}
\end{equation}
where \(h, h_1, p\) are generators of corresponding hom-groups, with \(h, h_1\) (resp. \(p\)) having multiplicity \(2\) (resp. \(1\)) in \(DM\). Applying \([-,\mathbb{Z}(2)[3]]_{MW}\) on the triangle, we get an exact sequence
\[F^*\xrightarrow{2}2F^*\longrightarrow H^1(O(2)^{\times},K_2^{MW})\longrightarrow2\mathbb{Z}\xrightarrow{2}\mathbb{Z}.\]
The first arrow is surjective and the last arrow is injective so we get
\begin{equation}\label{mw}H^1(O(2)^{\times},K_2^{MW})=0.\end{equation}
We have an exact sequence
\[\xrightarrow{h} H^1(O(2)^{\times},K_2^{MW})\longrightarrow H^1(O(2)^{\times},K_2^W)\longrightarrow CH^2(O(2)^{\times})\xrightarrow{h}\widetilde{CH}^2(O(2)^{\times}).\]
The 
\[h:[\mathbb{Z}/\eta,\mathbb{Z}]_M\longrightarrow[\mathbb{Z}/\eta,\mathbb{Z}]_{MW}\]
is an isomorphism whereas the 
\[h:[\mathbb{Z}/\eta,\mathbb{Z}(1)[2]]_M\longrightarrow[\mathbb{Z}/\eta,\mathbb{Z}(1)[2]]_{MW}\]
is the inclusion \(2\mathbb{Z}\subseteq\mathbb{Z}\). By using the diagram \eqref{map}, we obtain a commutative diagram with exact rows
\[
	\xymatrix
	{
		\mathbb{Z}\ar[r]^{2}\ar[d]_{2}	&\mathbb{Z}\ar[r]\ar[d]_2	&CH^2(O(2)^{\times})\ar[r]\ar[d]_h			&0\\
		2\mathbb{Z}\ar[r]^{2}					&\mathbb{Z}\ar[r]				&\widetilde{CH}^2(O(2)^{\times})\ar[r]	&0
	},
\]
which shows that \(h\) is injective (but not surjective!). Combining with \eqref{mw}, we have proved the claim.
\end{proof}
From the proposition above, we see that the \(Sq^2\) is detected in \(\widetilde{DM}\), since the mapping cone of the hyperbolic map \(\mathbb{Z}_M\longrightarrow\mathbb{Z}_{MW}\) represents \(H_W\mathbb{Z}\) in \(\widetilde{DM}\) (see \cite[Corollary 5.4, Chapter 7]{BCDFO}). Furthermore, for any \(X\in Sm/F\) and \(\mathscr{L}\in Pic(X)\), we have
\[Im(H^{p,q}_{MW}(X,\mathbb{Z},\mathscr{L})\longrightarrow H^{p,q}_M(X,\mathbb{Z}))\subseteq Ker(\iota\circ Sq^2_{\mathscr{L}}\circ\pi)\]
because the composites
\[H\widetilde{\mathbb{Z}}\longrightarrow H_{\mu}\mathbb{Z}/2\longrightarrow H_W\mathbb{Z}/\eta\longrightarrow H_W\mathbb{Z}\wedge\mathbb{P}^1\longrightarrow H_{\mu}\mathbb{Z}/2\wedge\mathbb{P}^1\]
\[H_{\mu}\mathbb{Z}/2[2]\xrightarrow{\tau[1]}H_{\mu}\mathbb{Z}/2\wedge\mathbb{P}^1\longrightarrow C(\tau[1])\]
are zero, where the \(\pi\) is the modulo 2 map and \(\tau,\iota,C(-)\) are mentioned before the proposition above.
\begin{proposition}\label{Gysin}
Suppose that \(X, Y\in Sm/F\), \(Y\) is a closed subset of \(X\) with \(codim_XY=n\) and that \({\mathscr{E}}\) is a vector bundle over \(X\). We have a distinguished triangle
\[Th(det({\mathscr{E}}|_{X\setminus Y}))\longrightarrow Th(det({\mathscr{E}}))\longrightarrow Th(det({\mathscr{E}}|_Y)\otimes det(N_{Y/X}))(n)[2n]\longrightarrow\cdots[1]\]
in \(\widetilde{DM}(S,R)\), where \(N_{Y/X}\) is the normal bundle.
\end{proposition}
\begin{proof}
We have an exact sequence
\[0\longrightarrow Th({\mathscr{E}}|_{X\setminus Y})\longrightarrow Th({\mathscr{E}})\longrightarrow Th(N_{Y/{\mathscr{E}}})\longrightarrow 0\]
by the inclusions \(Y\subseteq X\subseteq {\mathscr{E}}\). Then we apply \cite[Theorem 6.2]{Y}.
\end{proof}

Recall that the Hopf map \(\begin{array}{ccc}\mathbb{A}^2\setminus 0&\longrightarrow&\mathbb{P}^1\\(x,y)&\longmapsto&[x:y]\end{array}\) gives the Hopf element (see \cite[Definition 4.1]{Y1}) \(\eta\in[\mathbb{Z}(1)[1],\mathbb{Z}]_{MW}\) up to a \(\mathbb{P}^1\)-suspension. We write \(A/\eta=A\otimes C(\eta)\) for any MW-motive \(A\).
\begin{definition}
We say that an element \(A\in\widetilde{DM}(pt,\mathbb{Z})\) quasi-splits if it is a finite direct sum of elements like \(\mathbb{Z}(q)[p]\) or \(\mathbb{Z}/{\eta}(q)[p]\) (\(p,q\in\mathbb{Z}\)). We say that \(A\) splits if \(p=2q\) for all summands.
\end{definition}
\begin{definition}
If the \(A\) splits, define its Witt weights \(WW(A)\) to be the set
\[\{i\in\mathbb{Z}|\mathbb{Z}(i)[2i]\textrm{ is a direct summand of }A\}.\]
We have
\[WW(A)=\{i\in\mathbb{Z}|[\mathbb{Z}(i+1)[2i+2],A[1]]_{MW}\neq 0\}\]
by \cite[Proposition 5.4]{Y1} hence it is well defined.
\end{definition}
Hence we have for example
\[WW(\mathbb{Z}(\mathbb{P}^n))=\begin{cases}\{0,n\}&\textrm{if \(n\) is odd}\\\{0\}&\textrm{if \(n\) is even}\end{cases}\]
by \cite[Theorem 5.5]{Y1}.
\begin{proposition}\label{W}
For any \(X\in Sm/F\), \(\mathscr{L}\in Pic(X)\) and \(n,m\in\mathbb{Z}\), we have
\[[Th(\mathscr{L}),H_W\mathbb{Z}(n)[m]]=H^{m-n-1}(X,\textbf{I}^{n-1}(\mathscr{L}))\]
\[[Th(\mathscr{L}),H\widetilde{\mathbb{Z}}(n)[m]]=H^{m-n-1}(X,K_{n-1}^{MW}(\mathscr{L}))\]
if \(m\geq 2n-1\), where \(\textbf{I}^{n-1}(\mathscr{L})=\textbf{I}^{n-1}\otimes_{\mathbb{Z}[O_X^{\times}]}\mathbb{Z}[\mathscr{L}^{\times}]\).
\end{proposition}
\begin{proof}
The statement is an analogue of \cite[Remark 4.2.7]{BCDFO}. The proof of both equations are the same so we do the first one. The right hand side is \(H^{m-n}_X(\mathscr{L},\textbf{I}^n)\) by the push-forward along the zero section of \(\mathscr{L}\). We have the Postnikov spectral sequence (see \cite[Remark 4.3.9]{Mo1})
\[H^p_X(\mathscr{L},\pi_{-q}(H_W\mathbb{Z})_n)\Longrightarrow[Th(\mathscr{L}),H_W\mathbb{Z}(n)[n+p+q]].\]
Then use that
\[H^p_X(\mathscr{L},\pi_{-q}(H_W\mathbb{Z})_n)=0\]
if \(q>0\) or \(p\geq n\) and \(q\neq 0\) (see \cite[Proposition 4.5]{Y1}).
\end{proof}

Define
\[\eta_{MW}^{p,q}(X,\mathscr{L})=[Th(\mathscr{L}),\mathbb{Z}/{\eta}(q+1)[p+2]]_{MW},\]
where we write \(\eta_{MW}^{2n,n}\) as \(\eta_{MW}^n\) conventionally.

For general cohomologies of \(\mathbb{Z}/\eta\), we have the following result:
\begin{proposition}\label{eta}
Suppose that \(X\in Sm/F\), \(\mathscr{L}\in Pic(X)\), \(p,q\in\mathbb{Z}\). There is a commutative diagram
\[
	\xymatrix
	{
		\eta_{MW}^{p,q}(X,\mathscr{L})\ar[r]^{\varphi^{p+2,q+1}}\ar[d]							&H^{p+2,q+1}_M(X)\ar[d]_{\iota\circ\pi}\\
		H^{p,q}_M(X)\ar[r]^-{\iota\circ Sq^2_{\mathscr{L}}\circ\pi}_-{\psi^{p,q}}	&[X,C(\tau[1])(q)[p]]_{M/2}
	}.
\]


If \(Th(\mathscr{L})\) quasi-splits, the diagram is Cartesian.
\end{proposition}
\begin{proof}
By the proof of \cite[Theorem 4.13]{Y1}, there is a distinguished triangle
\[H\widetilde{\mathbb{Z}}/\eta\longrightarrow H_{\mu}\mathbb{Z}/\eta\longrightarrow C(\tau[1])\longrightarrow\cdots[1],\]
which induces an exact sequence
\[\eta_{MW}^{p,q}(X,\mathscr{L})\longrightarrow H^{p+2,q+1}_M(X)\oplus H^{p,q}_M(X)\xrightarrow{u,\psi^{p,q}}[X,C(\tau[1])(q)[p]]_{M/2}.\]
The identification of \(u\) comes from the diagram (2) in the proof of \textit{loc. cit.} and that of \(\psi^{p,q}\) comes from Proposition \ref{id}. This gives the diagram.


For the last statement, it suffices to prove that the first arrow is injective, which is reduced to the injectivity when \(Th(\mathscr{L})=\mathbb{Z}, \mathbb{Z}/\eta\), since it quasi-splits. By the strong duality of \(\mathbb{Z}/\eta\) (see \cite[Proposition 5.8]{Y1}), we may assume \(Th(\mathscr{L})=\mathbb{Z}\). Finally we use Proposition \ref{coh}, (1).
\end{proof}
Recall that in \cite[Theorem 6.17]{V1}, we have
\[H^{*,*}_M(F,\mathbb{Z}/2)=K_*^M(F)/2[\tau]\]
where \(K_n^M(F)/2=H^{n,n}_M(F,\mathbb{Z}/2)\).
\begin{lemma}\label{tau}
\[Sq^2(\tau)=0.\]
\end{lemma}
\begin{proof}
By \cite[Proposition 8.3.6]{P}, we have
\[H^{*,*}_M(B_{gm}(\mathbb{Z}/2))=H^{*,*}_M(F,\mathbb{Z}/2)[a,b]/(a^2=\tau b+\rho a)\]
where \(a\in H^{1,1}_M(F,\mathbb{Z}/2)\), \(b\in H^{2,1}_M(F,\mathbb{Z}/2)\). By the computation in \cite[Lemma 9.3.3]{P} and Cartan formula (see \cite[Proposition 9.7]{V}), we see that
\[\tau b^2=Sq^2(\tau)b+\tau b^2\]
after applying \(Sq^2\) on the relation \(a^2=\tau b+\rho a\). Hence the statement follows.
\end{proof}
Suppose \(\mathcal{F}_{\bullet}\) is a pointed simplicial sheaf. Let us define
\[Ker_{\tau}(Sq^2)_{p,q}(\mathcal{F}_{\bullet})=\{u\in H^{p,q}_M(\mathcal{F}_{\bullet},\mathbb{Z}/2)|Sq^2(u)\equiv 0 (\textrm{mod }\tau)\}\]
\[Im_{\tau}(Sq^2)_{p,q}(\mathcal{F}_{\bullet})=\{u\in H^{p,q}_M(\mathcal{F}_{\bullet},\mathbb{Z}/2)|u\in Im(Sq^2) (\textrm{mod }\tau)\}.\]
\begin{lemma}\label{tau1}
In the context above, we have
\[Im_{\tau}(Sq^2)_{p,q}\subseteq Ker_{\tau}(Sq^2)_{p,q}\]
on \(H^{*,*}_M(\mathcal{F}_{\bullet},\mathbb{Z}/2)\).
\end{lemma}
\begin{proof}
By Adem relations (see \cite[Theorem 10.2]{V} and \cite[Theorem 5.1]{MKO}), we have
\[Sq^2Sq^2=\tau Sq^3Sq^1.\]
Suppose
\[u=Sq^2(v)+\tau\cdot w,\]
then
\[Sq^2(u)=\tau Sq^3Sq^1(v)+Sq^2(\tau\cdot w)=\tau Sq^3Sq^1(v)+\tau Sq^2(w)+\tau\rho Sq^1(w)\]
by Cartan formula and Lemma \ref{tau}. So the statement follows.
\end{proof}
\begin{definition}
Suppose \(X\in Sm/F\), \(p,q\in\mathbb{Z}\) and \(\mathscr{L}\in Pic(X)\), define the motivic Bockstein cohomology
\[E^{p,q}(X,\mathscr{L})=\frac{Ker_{\tau}(Sq^2)_{p+2,q+1}}{Im_{\tau}(Sq^2)_{p+2,q+1}}(Th(\mathscr{L})),\]
where we write \(E^{2n,n}\) as \(E^n=\frac{Ker(Sq^2_{\mathscr{L}})_{2n,n}}{Im(Sq^2_{\mathscr{L}})_{2n,n}}\) conventionally.
\end{definition}
\begin{proposition}\label{P2}
We have an exact sequence
\[0\longrightarrow H_M^{p-2,p-1}(F,\mathbb{Z}/2)\longrightarrow H^{p+2,p+1}_M(\mathbb{P}^2,\mathbb{Z}/2)\xrightarrow{Sq^2} H_M^{p+4,p+2}(\mathbb{P}^2,\mathbb{Z}/2)\]
where the first arrow is the pushforward.
\end{proposition}
\begin{proof}
There is an exact sequence
\[H_M^{p+4,p+2}(\mathbb{P}^2,\mathbb{Z})\xrightarrow{h}H_{MW}^{p+4,p+2}(\mathbb{P}^2,\mathbb{Z})\longrightarrow H_W^{p+4,p+2}(\mathbb{P}^2,\mathbb{Z})\longrightarrow0.\]
The middle term is isomorphic to \(H_{M}^{p+4,p+2}(\mathbb{P}^2,\mathbb{Z})\) by Proposition \ref{coh} hence we may also replace \(h\) by \(2\). This shows that
\[H_W^{p+4,p+2}(\mathbb{P}^2,\mathbb{Z})=H_M^{p+4,p+2}(\mathbb{P}^2,\mathbb{Z}/2),\]
so the question is reduced to the kernel of Bockstein morphism \(\beta\) induced by the boundary map \(H_{\mu}\mathbb{Z}/2\longrightarrow H_W\mathbb{Z}\). On the other hand, there is an exact sequence
\[\xrightarrow{\eta}H_W^{p+2,p+1}(\mathbb{P}^2,\mathbb{Z})\xrightarrow{\rho} H^{p+2,p+1}_M(\mathbb{P}^2,\mathbb{Z}/2)\xrightarrow{\beta} H_W^{p+4,p+2}(\mathbb{P}^2,\mathbb{Z}).\]
The first term fits into an diagram with exact rows
\[\tiny\xymatrixcolsep{1pc}
	\xymatrix
	{
		H_M^{p+2,p+1}(\mathbb{P}^2,\mathbb{Z})\ar[r]^h\ar@{=}[d]&H_{MW}^{p+2,p+1}(\mathbb{P}^2,\mathbb{Z})\ar[r]\ar@{^{(}->}[d]&H_W^{p+2,p+1}(\mathbb{P}^2,\mathbb{Z})\ar[r]\ar[d]_{\rho}&H_M^{p+3,p+1}(\mathbb{P}^2,\mathbb{Z})\ar[r]^h\ar@{=}[d]&H_{MW}^{p+3,p+1}(\mathbb{P}^2,\mathbb{Z})\ar[d]_{\cong}\\
		H_M^{p+2,p+1}(\mathbb{P}^2,\mathbb{Z})\ar[r]^2\ar@{->>}[d]&H_{M}^{p+2,p+1}(\mathbb{P}^2,\mathbb{Z})\ar[r]\ar@{->>}[d]&H_M^{p+2,p+1}(\mathbb{P}^2,\mathbb{Z}/2)\ar[r]\ar[d]_g&H_M^{p+3,p+1}(\mathbb{P}^2,\mathbb{Z})\ar[r]^2\ar[d]_{\cong}&H_M^{p+3,p+1}(\mathbb{P}^2,\mathbb{Z})\ar[d]_{\cong}\\
		H_M^{p-2,p-1}(F,\mathbb{Z})\ar[r]^2&H_M^{p-2,p-1}(F,\mathbb{Z})\ar[r]&H_M^{p-2,p-1}(F,\mathbb{Z}/2)\ar[r]&H_M^{p-1,p-1}(F,\mathbb{Z})\ar[r]^2&H_M^{p-1,p-1}(F,\mathbb{Z})
	},
\]
hence \({\rho}\) is injective and \(g\circ {\rho}\) is surjective by the Five lemma. On the other hand, we claim the composite
\[H_W^{p+2,p+1}(\mathbb{P}^2,\mathbb{Z})\xrightarrow{\rho}H_M^{p+2,p+1}(\mathbb{P}^2,\mathbb{Z}/2)\longrightarrow H_M^{p+2,p+1}(\mathbb{P}^1,\mathbb{Z}/2)\longrightarrow H^{p,p}_M(F,\mathbb{Z}/2)\]
is zero, where the second arrow is induced by the pullback along \(\mathbb{P}^1\subseteq\mathbb{P}^2\). We have a commutative diagram
\[
	\xymatrix
	{
		H_W^{p+2,p+1}(\mathbb{P}^2,\mathbb{Z})\ar[r]^e\ar[d]_{\rho}	&H_W^{p+2,p+1}(\mathbb{P}^1,\mathbb{Z})\ar[d]\\
		H_M^{p+2,p+1}(\mathbb{P}^2,\mathbb{Z}/2)\ar[r]					&H_M^{p+2,p+1}(\mathbb{P}^1,\mathbb{Z}/2)
	}
\]
and it suffices to show that \(e=0\). The distinguished triangle
\[\mathbb{P}^1\longrightarrow\mathbb{P}^2\longrightarrow\mathbbm{1}(2)[4]\xrightarrow{\eta}\cdots[1]\]
in \(\mathcal{SH}(F)\) gives the exact sequence
\[H_W^{p+2,p+1}(\mathbb{P}^2,\mathbb{Z})\xrightarrow{e}H_W^{p+2,p+1}(\mathbb{P}^1,\mathbb{Z})\longrightarrow H_W^{p-1,p-1}(F,\mathbb{Z})\]
where the last arrow is the inclusion \(\textbf{I}^p(F)\subseteq\textbf{I}^{p-1}(F)\). So the claim is proved.
\end{proof}
\begin{proposition}\label{P2'}
We have
\[\begin{array}{cc}E^{p,q}(pt)=\begin{cases}K_q^M(F)/2&\textrm{if }p=q\\0&\textrm{else}\end{cases}&E^{p,q}(\mathbb{P}^2)=\begin{cases}K_q^M(F)/2&\textrm{if }p=q\\0&\textrm{else}\end{cases}.\end{array}\]
\end{proposition}
\begin{proof}
By Lemma \ref{tau}, the \(Sq^2\) acts trivially on \(H^{*,*}_M(F,\mathbb{Z}/2)\). So the computation of \(E^{*,*}(pt)\) is clear.

It then suffices to consider \(A^{p,q}=E^{p,q}(\mathbb{P}^2)/E^{p,q}(F)\).

If \(p\leq q\), the
\[H^{p,q}_M(\mathbb{P}^2,\mathbb{Z}/2)/H^{p,q}_M(F,\mathbb{Z}/2)=H^{p-4,q-2}_M(F,\mathbb{Z}/2)\cdot O(1)^2\oplus H^{p-2,q-1}_M(F,\mathbb{Z}/2)\cdot O(1)\] is a multiple of \(\tau\). Hence \(A^{p,q}=0\) again by Lemma \ref{tau}.

If \(p=q+1\), we have
\[Ker_{\tau}(Sq^2)_{p,q}=Ker(Sq^2)_{p,q}=H^{p-4,p-3}_M(F,\mathbb{Z}/2)=\tau K_{p-4}^M(F)/2\]
by Proposition \ref{P2}. Hence \(A^{p,q}=0\).

If \(p=q+2\), we have \(Im(Sq^2)_{p,q}=H^{p,q}_M(\mathbb{P}^2,\mathbb{Z}/2)\) by Proposition \ref{P2}. Hence \(A^{p,q}=0\).

If \(p>q+2\), we have \(H^{p,q}_M(\mathbb{P}^2,\mathbb{Z}/2)=0\). Hence \(A^{p,q}=0\).
\end{proof}
For any \(x\in\eta_{MW}^{p,q}(X,\mathscr{L})\), we will write \(x=(a,b)\) if \((a,b)\) is the image of \(x\) in \(H^{p,q}_M(X,\mathbb{Z})\oplus H^{p+2,q+1}_M(X,\mathbb{Z})\) (see Proposition \ref{eta}).
\begin{proposition}\label{mult}
Suppose that \(X\in Sm/F\) and \(\mathscr{L},\mathscr{M}\in Pic(X)\). The ring structure of \(\mathbb{Z}/\eta\) gives a product
\[\begin{array}{ccccc}\eta_{MW}^{p_1,q_1}(X,\mathscr{L})&\times&\eta_{MW}^{p_2,q_2}(X,\mathscr{M})&\longrightarrow&\eta_{MW}^{p_1+p_2,q_1+q_2}(X,\mathscr{L}\otimes\mathscr{M})\\u&,&v&\longmapsto&u\cdot v\end{array}\]
with the identity being \((1,0)\in\eta_{MW}^{0,0}(pt)\). If \(u=(a,b), v=(c,d)\) and \(Th(\mathscr{L}), Th(\mathscr{M})\) quasi-split, we have
\[(a,b)\cdot(c,d)=(ac,bc+ad).\]
\end{proposition}
\begin{proof}
In \(DM(pt,\mathbb{Z})\), the \(\mathbb{Z}/\eta=\mathbb{Z}\oplus\mathbb{Z}(1)[2]\) admits a commutative ring structure given by the map
\[\begin{array}{ccc}m:\mathbb{Z}\oplus\mathbb{Z}(1)[2]\oplus\mathbb{Z}(1)[2]\oplus\mathbb{Z}(2)[4]&\to&\mathbb{Z}\oplus\mathbb{Z}(1)[2]\\(a,b,c,d)&\mapsto&(a,b+c)\end{array}.\]
It satisfies the last formula in the statement. Regarding \(\mathbb{Z}/\eta(1)[2]\) as \((\mathbb{P}^2)^{\wedge1}\), the \(m\) is the class
\[(O(1)\wedge O(1),O(1)\wedge O(1)^2+O(1)^2\wedge O(1))\in CH^2(\mathbb{P}^2\wedge\mathbb{P}^2)\oplus CH^3(\mathbb{P}^2\wedge\mathbb{P}^2),\]
which lies in \(\eta_{MW}^2(\mathbb{P}^2\wedge\mathbb{P}^2)\) by \cite[Theorem 4.13]{Y1}. So it lifts to a map \((\mathbb{Z}/\eta)^2\longrightarrow\mathbb{Z}/\eta\) in \(\widetilde{DM}(pt,\mathbb{Z})\), which is defined to be the ring structure. Then the axioms for the ring structures follow from the fact that the map
\[[(\mathbb{Z}/\eta)^a,(\mathbb{Z}/\eta)^b]_{MW}\longrightarrow[(\mathbb{Z}/\eta)^a,(\mathbb{Z}/\eta)^b]_{M}\]
is injective for \(a,b>0\), by \cite[Proposition 5.6]{Y1}.
\end{proof}
\begin{definition}
For any \(A,B\in\widetilde{DM}(pt,\mathbb{Z})\), we define
\[[A,B]_{\eta}=\varinjlim_{n\to+\infty}[A(n)[n],B]_{MW}\]
where the limit is taken with respect to the Hopf map \(\eta\). Define \(\widetilde{DM}_{\eta}\) to be the category with same objects as \(\widetilde{DM}(pt,\mathbb{Z})\) and morphisms as defined above. There is a canonical functor
\[L:\widetilde{DM}(pt,\mathbb{Z})\longrightarrow\widetilde{DM}_{\eta}.\]
\end{definition}
We will write \(L(\mathbb{Z}(X))\) as \(\mathbb{Z}(X)\).
\begin{proposition}\label{localization}
The category \(\widetilde{DM}_{\eta}\) is the Verdier localization of the category \(\widetilde{DM}(pt,\mathbb{Z})\) with respect to the morphisms \(\mathbb{Z}(X)\otimes\eta\) for every \(X\in Sm/F\). Thus it is a triangulated category and the functor \(L\) is monoidal and exact.
\end{proposition}
\begin{proof}
Suppose \(\Sigma\) is the smallest thick trianglated subcategory (with arbitrary coproducts) of \(\widetilde{DM}(pt,\mathbb{Z})\) containing \(\mathbb{Z}(X)/\eta\) for every \(X\in Sm/F\). If
\[X\xrightarrow{f}Y\longrightarrow Z\longrightarrow X[1]\]
is a distinguished triangle in \(\widetilde{DM}(pt,\mathbb{Z})\) with \(Z\in\Sigma\), we claim that \(L(f)\) is an isomorphism. It suffices to show that \([T,f]_{\eta}\) is an isomorphism for any \(T\). The functor \([T,-]_{\eta}\) also induces a long exact sequence with respect to the triangle. So it suffices to show that \([T,\mathbb{Z}(X)/\eta]_{\eta}=0\) for any \(T\). It is easy to check that the morphism \(\mathbb{Z}(X)\otimes\eta\) has the inverse \(\epsilon Id_{\mathbb{Z}(X)(1)[1]}\) (see \cite[Lemma 2.0.7, \S 4]{BCDFO} for the sign \(\epsilon\)) in \(\widetilde{DM}_{\eta}\), so \(L(\mathbb{Z}(X)\otimes\eta)\) is an isomorphism. Hence we have proved the claim.

Then it is easy to show that the pair \((L, \widetilde{DM}_{\eta})\) is the universal one such that \(L(f)\) is an isomorphism if \(C(f)\in\Sigma\). The \(\widetilde{DM}_{\eta}\) has a natural symmetric monoidal structure such that \(L\) is symmetric monoidal. This concludes the proof. (see \cite[Proposition 4.6.2]{Kra})
\end{proof}
Hence we see that \(L(\mathbb{Z}(1)[1])=L(\mathbb{Z})\) so there is no need to consider the Tate twist in \(\widetilde{DM}_{\eta}\).
\begin{proposition}
For any \(X\in Sm/F\), \(\mathscr{L}\in Pic(X)\) and \(m,n\in\mathbb{N}\), we have
\[[Th(\mathscr{L}),\mathbb{Z}(m)[m+n+1]]_{\eta}=H^n(X,\textbf{W}(\mathscr{L})).\]
\end{proposition}
\begin{proof}
We have
\[[Th(\mathscr{L})(i)[i],\mathbb{Z}(m)[m+n+1]]_{MW}=H^n(X\otimes\mathbb{G}_m^{\wedge i},K_{MW}^{m-1}(\mathscr{L}))\]
by Proposition \ref{W}. If \(i\) is large, by cancellation we may assume \(m<n,i\). Then the result follows by \cite[Lemma 3.1.8,\S 2]{BCDFO}.
\end{proof}
\begin{lemma}\label{uniqueness}
Suppose that \(A\in\widetilde{DM}(pt,\mathbb{Z})\) splits as
\[A=\oplus_{i\in\mathbb{N}}\mathbb{Z}(i)[2i]^{\oplus w_i(A)}\oplus\oplus_{j\in\mathbb{N}}\mathbb{Z}/\eta(j)[2j]^{\oplus t_j(A)}.\]
Then
\[L(A)=\oplus_{i\in\mathbb{N}}\mathbb{Z}[i]^{\oplus w_i(A)}\]
in \(\widetilde{DM}_{\eta}\). Suppose that \(\gamma^*:\widetilde{DM}\longrightarrow DM\) is the functor defined in \cite[3.2.4, \S 3]{BCDFO} and that
\[\gamma^*(A)=\oplus_{i\in\mathbb{N}}\mathbb{Z}(i)[2i]^{\oplus s_i(A)}.\]
We have
\[t_j(A)=\sum_{i\leq j}(-1)^i(s_{j-i}(A)-w_{j-i}(A)).\]
\end{lemma}
\begin{proof}
The first statement follows from \(L(\mathbb{Z}/\eta)=0\) and \(L(\mathbb{Z}(1)[1])=\mathbb{Z}\). The second statement follows from \(\gamma^*\mathbb{Z}/\eta=\mathbb{Z}\oplus\mathbb{Z}(1)[2]\).
\end{proof}
\begin{remark}\label{cell}
Suppose that \(k=\mathbb{R}\) and that \(X\in Sm/F\) is cellular (see \cite[Definition 5.1]{HWXZ}), we have
\[[\Sigma^{\infty}X_+,H_W\mathbb{Z}(i)[2i]]=H^i(X,\textbf{I}^i)=H^i(X(\mathbb{R}),\mathbb{Z})\]
by \cite[Proposition 4.5]{Y1} and \cite[Theorem 5.7]{HWXZ}. On the other hand, we have
\[[\mathbb{Z}/{\eta}(i)[2i],H_W\mathbb{Z}(j)[2j]]=\begin{cases}\mathbb{Z}/2&\text{if \(j-i=1\)}\\0&\textrm{else}\end{cases}\]
\[[\mathbbm{1}(i)[2i],H_W\mathbb{Z}(j)[2j]]=\begin{cases}\mathbb{Z}&\text{if \(j=i\)}\\0&\textrm{else}\end{cases}\]
by \cite[Proposition 5.3, 5.6]{Y1} and the distinguished triangle (see \cite[Lemma 19]{B})
\[H_{\mu}\mathbb{Z}\xrightarrow{h}H\widetilde{\mathbb{Z}}\longrightarrow H_W\mathbb{Z}\longrightarrow\cdots[1].\]

Hence we see that if \(\mathbb{Z}(X)\) also splits as an MW-motive, then we have
\[H^i(X(\mathbb{R}),\mathbb{Z})_{free}=\mathbb{Z}^{\oplus w_i(\mathbb{Z}(X))}\]
\[H^i(X(\mathbb{R}),\mathbb{Z})_{tor}=\mathbb{Z}/2^{\oplus t_{i-1}(\mathbb{Z}(X))}.\]
Whereas
\[CH^i(X)=H^{2i}(X(\mathbb{C}),\mathbb{Z})=\mathbb{Z}^{\oplus w_i(\mathbb{Z}(X))+t_i(\mathbb{Z}(X))+t_{i-1}(\mathbb{Z}(X))}\]
by \cite[Proposition 5.3]{HWXZ}.
\end{remark}
Now we want to give a general method to detect splitting in MW-motives.
\begin{proposition}\label{splitting}
Suppose \(f:X\longrightarrow Y\) is a morphism in \(\widetilde{DM}(pt,\mathbb{Z})\). If both \(X\) and \(C(f)\) split as MW-motives, then \(f\) is injective in \(\widetilde{DM}(pt,\mathbb{Z})\) if and only if \(L(f)\) is injective.
\end{proposition}
\begin{proof}
This is because by \cite[Proposition 5.4]{Y1}, we have an isomorphism
\[L:[C(f),X[1]]_{MW}\longrightarrow[C(f),X[1]]_{\eta}\]
hence the injectivity of \(f\) is equivalent to that of \(L(f)\) under our setting.
\end{proof}
The ring structure of Chow groups makes \(\oplus_{n,\mathscr{L}}E^n(X,\mathscr{L})\) an \(\mathbb{N}\times Pic(X)/2\)-graded ring. Given a morphism \(f:X\longrightarrow Y\), there is a pullback
\[f^*:E^n(Y,\mathscr{L})\longrightarrow E^n(X,f^*\mathscr{L})\]
by the natuality of Steenrod operations. Moreover, if \(f\) is projective with relative dimension \(c\), the pushforward of Chow groups gives a pushforward
\[f_*:E^n(X,\omega_X\otimes f^*\mathscr{L})\longrightarrow E^{n-c}(Y,\omega_Y\otimes\mathscr{L}),\]
which is well-defined by \cite[Theorem 4.5]{Fu}. For every vector bundle \(\mathscr{E}\), we have the Euler class \(e(\mathscr{E})\in E^n(X,det(\mathscr{E})^{\vee})\) and Pontryagin classes \(p_n(\mathscr{E})\in E^{2n}(X)\) inherited by those of Chow groups. 
\begin{proposition}\label{odd}
Let \(X\in Sm/F\) and \(\mathscr{E}\) be a vector bundle over \(X\). Then for every odd \(i\), we have
\[\begin{array}{cc}p_i(\mathscr{E})=0\in H^{2i}(X,\textbf{W})&p_i(\mathscr{E})=0\in E^{2i}(X)\end{array}.\]
\end{proposition}
\begin{proof}
Denote by \(p:X\times\mathbb{G}_m\longrightarrow X\) the projection. By \cite[Proposition 2.2.8]{DF}, we have
\[p^*(p_i(\mathscr{E}))\cdot<t>=p^*(p_i(\mathscr{E}))\]
where \(t\) is the parameter in \(\mathbb{G}_m\). But
\[H^{2i}(X\times\mathbb{G}_m,\textbf{W})=H^{2i}(X,\textbf{W})\oplus H^{2i}(X,\textbf{W})\cdot(<t>-<1>)\]
by \cite[Lemma 3.1.8]{BCDFO}. Then we have
\[(p_i(\mathscr{E}),0)=p^*(p_i(\mathscr{E}))=p^*(p_i(\mathscr{E}))\cdot<t>=(p_i(\mathscr{E}),0)\cdot<t>=(p_i(\mathscr{E}),p_i(\mathscr{E})).\]
So \(p_i(\mathscr{E})=0\). The second statement follows from \cite[Remark 7.9]{HW}.
\end{proof}
By Jouanolou's trick, there is a Whitney sum formula (see \cite[2.2.4]{DF}) for even Pontryagin classes of vector bundles (in \(H^*(-,\textbf{W})\) or \(E^*(-)\)) on quasi-projective smooth schemes. Namely, for any exact sequence of bundles
\[0\longrightarrow\mathscr{E}_1\longrightarrow\mathscr{E}_2\longrightarrow\mathscr{E}_3\longrightarrow 0,\]
we have
\[\sum_ip_{2i}(\mathscr{E}_1)t^i\sum_jp_{2j}(\mathscr{E}_3)t^j=\sum_kp_{2k}(\mathscr{E}_2)t^k\]
where \(t\) is an indeterminant.


\begin{theorem}\label{sq}
Suppose that \(X\in Sm/F\), \(p,q\in\mathbb{Z}\), \(\mathscr{L}\in Pic(X)\) and that \(Th(\mathscr{L})\) quasi-splits in \(\widetilde{DM}(pt,\mathbb{Z})\).
\begin{enumerate}
\item There is a natural map
\[\Delta:H^{p,q}_{\eta}(X,\mathscr{L})\longrightarrow E^{p,q}(X,\mathscr{L})\]
which induces an isomorphism of \(\mathbb{Z}/2\)-vector spaces
\[\delta:H^{p,q}_{\eta}(X,\mathscr{L})/H^{p+1,q+1}_{\eta}(X,\mathscr{L})\cong E^{p,q}(X,\mathscr{L}).\]
\item We have a Cartesian square
\[
	\xymatrix
	{
		H^{p,q}_{MW}(X,\mathbb{Z},\mathscr{L})\ar[r]\ar[d]	&Ker(\iota\circ Sq_{\mathscr{L}}^2\circ\pi)_{p,q}\ar[d]\\
		H^{p,q}_{\eta}(X,\mathscr{L})\ar[r]^{\Delta}			&E^{p,q}(X,\mathscr{L}).
	}
\]
In particular, if \(Th(\mathscr{L})\) splits, there is an isomorphism
\[H^p(X,\textbf{W}(\mathscr{L}))\otimes_{\textbf{W}(F)}\mathbb{Z}/2=E^{2p,p}(X,\mathscr{L})\]
and a decomposition
\[\widetilde{CH}^{*}(X,\mathscr{L})\cong\textbf{I}(F)\cdot H^*(X,\textbf{W}(\mathscr{L}))\oplus Ker(Sq^2_{\mathscr{L}}\circ\pi)_*.\]
\end{enumerate}
\end{theorem}
\begin{proof}
\begin{enumerate}
\item Define
\[K^{p,q}=Ker(H^{p,q}_{MW}(X,\mathbb{Z},\mathscr{L})\longrightarrow H^{p,q}_{\eta}(X,\mathscr{L})).\]
We show that \(K^{p,q}\) goes to zero under the composite
\[H^{p,q}_{MW}(X,\mathbb{Z},\mathscr{L})\longrightarrow Ker(\iota\circ Sq_{\mathscr{L}}^2\circ\pi)\longrightarrow E^{p,q}(X,\mathscr{L})\]
hence the \(\Delta\) is defined. For this, it suffices to suppose \(Th(\mathscr{L})=\mathbb{Z}\) and \(p=q\) by Proposition \ref{P2'}. But in this case, we have \(K^{p,q}=2K_q^M(F)\).

For the second statement, it suffices to consider the cases when \(Th(\mathscr{L})=\mathbb{Z},\mathbb{Z}/\eta\).

Suppose \(Th(\mathscr{L})=\mathbb{Z}\). If \(p\neq q\), both sides vanish, otherwise the equation is the isomorphism
\[K_q^M(F)/2=K_q^W(F)/K_{q+1}^W(F).\]

Suppose \(Th(\mathscr{L})=\mathbb{Z}/\eta\). Then both sides vanish by Proposition \ref{P2'}.
\item Here we use the computations in Proposition \ref{coh}. Since \(Th(\mathscr{L})\) quasi-splits, it suffices to prove the case when \(Th(\mathscr{L})=\mathbb{Z}\) and \(Th(\mathscr{L})=\mathbb{Z}/\eta\).

Suppose \(Th(\mathscr{L})=\mathbb{Z}\). If \(p\neq q\), the horizontal arrows are isomorphisms. If \(p=q\), the statements come from the definition of \(K_q^{MW}(F)\).

Suppose \(Th(\mathscr{L})=\mathbb{Z}/\eta\). If \(p\neq q\) or \(p=q<0\), the horizontal arrows are isomorphisms. So does the case when \(p=q\geq 0\) by the computation of \(Ker_{\tau}(Sq^2)(\mathbb{P}^2)_{p+2,q+1}\) in Proposition \ref{P2}.

If \(Th(\mathscr{L})\) splits, we have
\[H^{2q,q}_{\eta}(X,\mathscr{L})=H^q(X,\textbf{W}(\mathscr{L}))\]
by direct computation. Considering the case \(\mathbb{Z}(X)=\mathbb{Z}\) one shows that
\[H^{2q+1,q+1}_{\eta}(X,\mathscr{L})=\textbf{I}(F)\cdot H^q(X,\textbf{W}(\mathscr{L})).\]
So the statement follows since
\[Ker(\iota\circ Sq^2\circ\pi)_{2q,q}=Ker(Sq^2\circ\pi)_{2q,q}\]
and they are free abelian groups.
%
%
\end{enumerate}
\end{proof}
\begin{remark}\label{better}
Our results are significantly stronger than those in \cite{HW} or \cite{W1}.

Firstly, we compute cohomologies of degree \((p,q)\) rather than \((2q,q)\), largely thanks to Proposition \ref{id}.

Secondly, we emphasize that the quasi-split of MW-motives is an intrinsic property, which enables us to determine an MW-motivic cohomology class in terms of motivic cohomology class and Witt cohomology class.

Thirdly, we establish a correspondence between a subgroup of Witt cohomology and motivic Bockstein cohomology. In topological terms, this says that if \(H^*(X,\mathbb{Z})\) is finitely generated and has only \(2\)-torsions, the Bockstein spectral sequence degenerates.

Let us stress again that our advantage is the characterization of split itself, rather than the determination of that.
\end{remark}
\begin{proposition}\label{gen}
Suppose that \(X\in Sm/F\), \(\mathscr{L}\in Pic(X)\) and that \(Th(\mathscr{L})\) splits in \(\widetilde{DM}(pt,\mathbb{Z})\). Suppose \(\lambda\) (resp. \(\mu\)) is a generator of \([\mathbb{Z}/\eta,\mathbb{Z}(1)[2]]_{MW}=\mathbb{Z}\) (resp. \([\mathbb{Z}/\eta,\mathbb{Z}]_{MW}=2\mathbb{Z}\)). We have maps
\[\partial:\eta_{MW}^{n-1}(X,\mathscr{L})\longrightarrow\widetilde{CH}^n(X,\mathscr{L})\]
\[h:\eta_{MW}^{n}(X,\mathscr{L})\longrightarrow\widetilde{CH}^{n}(X,\mathscr{L})\]
induced by \(\lambda\) and \(\mu\), respectively. Denote the natural map \(\widetilde{CH}^n(X,\mathscr{L})\longrightarrow CH^n(X)\) by \(\gamma\).

Given a morphism in \(\widetilde{DM}(pt,\mathbb{Z})\)
\[f=\{s_i,t_j\}:Th(\mathscr{L})\longrightarrow Th(\mathscr{L})\cong\oplus_i\mathbb{Z}(a_i)[2a_i]\oplus\oplus_j\mathbb{Z}/\eta(b_j)[2b_j],\]
the following statements are equivalent:
\begin{enumerate}
\item The \(f\) is an isomorphism;
\item 
For every \(n\in\mathbb{N}\), we have
\[\widetilde{CH}^{n}(X,\mathscr{L})=\bigoplus_{a_k=n+1}\textbf{GW}(F)\cdot s_k\oplus\bigoplus_{b_k=n+1}\mathbb{Z}\cdot h(t_k)\oplus\bigoplus_{b_k=n}\mathbb{Z}\cdot\partial(t_k);\]
\item For every \(n\in\mathbb{N}\), we have
\[Ker(Sq_{\mathscr{L}}^2\circ\pi)_n=\bigoplus_{a_k=n+1}\mathbb{Z}\cdot\gamma(s_k)\oplus\bigoplus_{b_k=n+1}\mathbb{Z}\cdot\gamma(h(t_k))\oplus\bigoplus_{b_k=n}\mathbb{Z}\cdot\gamma(\partial(t_k))\]
\[H^n(X,\textbf{W}(\mathscr{L}))=\bigoplus_{a_k=n+1}\textbf{W}(F)\cdot w(s_k)\]
where \(w\) is the composite
\[\widetilde{CH}^{n}(X,\mathscr{L})\longrightarrow H^n(X,\textbf{I}^n(\mathscr{L}))\longrightarrow H^n(X,\textbf{W}(\mathscr{L})).\]
\end{enumerate}
\end{proposition}
\begin{proof}
We have maps
\[\begin{array}{cc}\begin{array}{cccc}h_1:&\eta_{MW}^{n}(X,\mathscr{L})&\to&\eta_{MW}^n(X,\mathscr{L})\\&(x,y)&\mapsto&(2x,0)\end{array}&
\begin{array}{cccc}h_2:&\eta_{MW}^{n}(X,\mathscr{L})&\to&\eta_{MW}^{n-1}(X,\mathscr{L})\\&(x,y)&\mapsto&(0,2x)\end{array}\\
\begin{array}{cccc}p:&\eta_{MW}^{n}(X,\mathscr{L})&\to&\eta_{MW}^{n+1}(X,\mathscr{L})\\&(x,y)&\mapsto&(y,0)\end{array}&
\begin{array}{cccc}\gamma_1:&\widetilde{CH}^n(X,\mathscr{L})&\to&\eta_{MW}^{n}(X,\mathscr{L})\\&x&\mapsto&(\gamma(x),0)\end{array}\end{array}.\]
Here the maps come from the following. The \(\gamma_1\) (resp. \(h_2\)) is induced by the generator in \([\mathbb{Z},\mathbb{Z}/\eta]_{MW}\) (resp. \([\mathbb{Z}/\eta(1)[2],\mathbb{Z}/\eta]_{MW}\)). Moreover,
\[\begin{array}{cc}p=\gamma_1\circ\partial&h_1=\gamma_1\circ h\end{array}.\]

\((1)\Rightarrow(2)\): If \(f\) is an isomorphism, applying \([-,\mathbb{Z}(n+1)[2n+2]]_{MW}\) we obtain an isomorphism
\[\bigoplus_k[\mathbb{Z}((a_k)),\mathbb{Z}((n+1))]_{MW}\oplus\bigoplus_k[\mathbb{Z}/\eta((b_k)),\mathbb{Z}((n+1))]_{MW}\cong\widetilde{CH}^{n}(X,\mathscr{L}).\]
So the statement follows from \cite[Proposition 5.6]{Y1}.

\((2)\Rightarrow(1)\): Suppose \(\widetilde{CH}^{n}(X,\mathscr{L})\) has the required decomposition for every \(n\). Then \([f,\mathbb{Z}(n)[2n]]_{MW}\) is an isomorphism for every \(n\). So it suffices to prove that \([f,\mathbb{Z}/\eta(n)[2n]]_{MW}\) is an isomorphism for every \(n\), which induces an endomorphism of a free abelian group of finite rank by \textit{loc. cit.}. Hence it suffices to prove that it is surjective.

Denote by \(\pi\) the reduction modulo \(2\) map as before. By Theorem \ref{sq}, we have an exact sequence
\[0\longrightarrow Ker(Sq_{\mathscr{L}}^2\circ\pi)_{n-1}\longrightarrow\eta_{MW}^{n-1}(X,\mathscr{L})\longrightarrow\pi^{-1}(Im(Sq_{\mathscr{L}}^2)_n)\longrightarrow 0\]
and the elements
\[\{\gamma(s_k)\}_{a_k=n}\cup\{\gamma(h(t_k))\}_{b_k=n}\cup\{\gamma(\partial(t_k))\}_{b_k=n-1}\]
generate \(Ker(Sq_{\mathscr{L}}^2\circ\pi)_{n-1}\). Suppose \(t_k=(x_k,y_k)\). We have
\[\begin{array}{cc}\gamma(h(t_k))=2x_k&\gamma(\partial(t_k))=y_k\end{array}.\]
On the other hand, we have
\[\begin{array}{ccc}\gamma_1(s_k)=(\gamma(s_k),0)&\gamma_1(h(t_k))=(2x_k,0)&\gamma_1(\partial(t_k))=(y_k,0)\end{array},\]
which shows that \(Ker(Sq_{\mathscr{L}}^2\circ\pi)_{n-1}\subseteq Im([f,\mathbb{Z}/\eta(n)[2n]]_{MW})\) for every \(n\).

Again by the proof of Theorem \ref{sq}, the \(Im(Sq_{\mathscr{L}}^2)_n\) is freely generated by \(\pi(\{\gamma(\partial(t_k))\}_{b_k=n-1})\) for every \(n\). So suppose we have \((x,y)\in\eta_{MW}^{n-1}(X,\mathscr{L})\). We could find \((x_1,y_1)\in\mathbb{Z}\{t_k\}_{b_k=n-1}\) such that \(\pi(y-y_1)=0\), so \(y-y_1=2y_2\). Then there is a \((u,v)\in\mathbb{Z}\{t_k\}_{b_k=n}\) such that \(\pi(v)=Sq_{\mathscr{L}}^2(y_2)\). Hence
\[(x,y)-(x_1,y_1)-h_2(u,v)=(x-x_1,2y_2-2u)\]
and \(y_2-u\in Ker(Sq_{\mathscr{L}}^2\circ\pi)_n\). Then \((x-x_1,0),(0,2y_2-2u)=h_2(y_2-u,0)\in Im([f,\mathbb{Z}/\eta(n)[2n]]_{MW})\) so we have proved that \((x,y)\in Im([f,\mathbb{Z}/\eta(n)[2n]]_{MW})\).

\((2)\Rightarrow(3)\): The map \(w\) is surjective because it is the composite of surjective maps. In the proof of \cite[Theorem 4.13]{Y1}, we showed that there is a commutative diagram
\[
	\xymatrix
	{
		\eta^{n-1}_{MW}(X,\mathscr{L})\ar[r]^{\partial}\ar[d]	&\widetilde{CH}^{n}(X,\mathscr{L})\ar[d]\\
		Ch^{n-1}(X)\ar[r]^-{\beta}										&H^n(X,\textbf{I}^n(\mathscr{L}))
	}
\]
where the lower horizontal arrow is the Bockstein map. So \(w(\partial(t_k))=0\). Moreover, the composite
\[\mathbb{Z}/\eta(n)[2n]\xrightarrow{h}\mathbb{Z}(n)[2n]\xrightarrow{\eta}\mathbb{Z}(n-1)[2n-1]\]
vanishes in \(\widetilde{DM}(pt,\mathbb{Z})\) by \cite[Proposition 5.4]{Y1} so \(w(h(t_k))=0\). So \(\{w(s_k)\}_{a_k=n+1}\) freely generate \(H^n(X,\textbf{W}(\mathscr{L}))\). On the other hand, by Theorem \ref{sq}, the elements given generate \(Ker(Sq_{\mathscr{L}}^2\circ\pi)_n\) and the kernel of \(\gamma\) is \(\bigoplus_{a_k=n+1}\textbf{I}(F)\cdot s_k\). So they freely generate \(Ker(Sq_{\mathscr{L}}^2\circ\pi)_n\).

\((3)\Rightarrow(2)\): Suppose \(x\in\widetilde{CH}^{n}(X,\mathscr{L})\), we could find a \(y\in\mathbb{Z}\{s_k,h(t_k),\partial(t_k)\}\) such that \(\gamma(x-y)=0\). Then \(x-y\in\sum_k\textbf{I}(F){s_k}\) by Theorem \ref{sq}. Suppose \(\sum e_k\cdot s_k+\sum n_k\cdot\partial(t_k)+\sum m_k\cdot h(t_k)=0\), where \(e_k\in\textbf{GW}(F), n_k, m_k\in\mathbb{Z}\). Then applying \(\gamma\), we find \(n_k=m_k=0\) and \(e_k\in\textbf{I}(F)\) by the conditions. But \(\sum e_k\cdot w(s_k)=0\) implies that \(e_k=0\) as well. This concludes the proof.
\end{proof}
\begin{proposition}\label{etaIm}
In the context above, the following sequence
\[\eta_{MW}^{n-1}(X,\mathscr{L})\xrightarrow{\partial}\widetilde{CH}^n(X,\mathscr{L})\longrightarrow H^n(X,\textbf{W}(\mathscr{L}))\longrightarrow 0\]
is exact. Moreover, we have
\[Im(\partial)=\bigoplus_{b_k=n+1}h\cdot\textbf{GW}(F)\cdot s_k\oplus\bigoplus_{b_k=n+1}\mathbb{Z}\cdot h(t_k)\oplus\bigoplus_{b_k=n}\mathbb{Z}\cdot\partial(t_k).\]
\end{proposition}
\begin{proof}
We have a commutative diagram
\[
	\xymatrix
	{
		\widetilde{CH}^n(X,\mathscr{L})\ar[r]^-{\eta}\ar[d]	&H^{2n+1,n}_{MW}(Th(\mathscr{L}),\mathbb{Z})\ar[d]_w\\
		H^i(Gr(k,n),\textbf{I}^i)\ar[r]^-{\eta'}					&H^n(X,\textbf{W}(\mathscr{L}))
	}
\]
where \(Ker(\eta)=Im(\partial)\) and \(w\) is an isomorphism by Proposition \ref{W} and
\[H^{2n+1,n}_{MW}(Th(\mathscr{L}),\mathbb{Z})=H^n(X,K^{MW}_{n-1}(\mathscr{L}))=H^n(X,\textbf{W}(\mathscr{L})).\]
The statements then follow from \(Ker(\eta')=Im(\beta)\) and Proposition \ref{gen}.
\end{proof}
\section{Geometry of Grassmannians and Complete Flags}
Let us state a collection of results on the geometry of Grassmannians. Suppose \(V\) is a vector space of dimension \(n\) with a flag
\[0\subseteq V_1\subseteq\cdots\subseteq V_n=V\]
where \(dimV_i=i\). Denote by \(\mathscr{U}_{k,n}\) (resp. \(\mathscr{U}_{k,n}^{\perp}\)) the (resp. complement) tautological bundle of \(Gr(k,n)\). For a morphism \(f:X\longrightarrow Y\), we denote by \(\Omega_f\) (resp. \(N_f\)) the relative cotangent bundle (resp. normal bundle), when it makes sense.
\begin{proposition}\label{geo1}
We have closed imbeddings
\[\begin{array}{cc}\begin{array}{cccc}i:&Gr(k,n-1)&\longrightarrow&Gr(k,n)\\&\Lambda&\longmapsto&\Lambda\end{array}&\begin{array}{cccc}j:&Gr(k-1,n-1)&\longrightarrow&Gr(k,n)\\&\Lambda&\longmapsto&\Lambda+e_n\end{array}\end{array}\]
with \(det(N_i)=O_{Gr(k,n-1)}(1)=det(\mathscr{U}_{k,n-1})\), \(det(N_j)=O_{Gr(k-1,n-1)}(1)\) and \(Gr(k,n-1)\cap Gr(k-1,n-1)=\emptyset\). Suppose \(p:V\longrightarrow V_{n-1}\) is a projection. The morphism
\[\begin{array}{cccc}f:&Gr(k,n)\setminus Gr(k-1,n-1)&\longrightarrow&Gr(k,n-1)\\&\Lambda&\longmapsto&p(\Lambda)\end{array}\]
(resp.
\[\begin{array}{cccc}g:&Gr(k,n)\setminus Gr(k,n-1)&\longrightarrow&Gr(k-1,n-1)\\&\Lambda&\longmapsto&\Lambda\cap V_{n-1}\end{array})\]
corresponds to \(\mathscr{U}_{k,n-1}\) (resp. \(\mathscr{U}_{k-1,n-1}^{\perp}\)) which is factored through by \(i\) (resp \(j\)) as the zero section. Moreover, we have
\[\begin{array}{cc}i^*\mathscr{U}_{k,n}=\mathscr{U}_{k,n-1},&i^*\mathscr{U}_{k,n}^{\perp}=\mathscr{U}_{k,n-1}^{\perp}\oplus O_{Gr(k,n-1)},\end{array}\]
\[\begin{array}{cc}j^*\mathscr{U}_{k,n}=\mathscr{U}_{k-1,n-1}\oplus O_{Gr(k-1,n-1)},&j^*\mathscr{U}_{k,n}^{\perp}=\mathscr{U}_{k-1,n-1}^{\perp}.\end{array}\]
\end{proposition}
\begin{proof}
Obvious.
\end{proof}
\begin{proposition}\label{geo2}
Let \(p:V_n\longrightarrow V_{n-2}\) be the projection. Define
\[S=\{\Lambda\in Gr(k,n)|\Lambda\cap ker(p)\neq 0\}.\]
We have imbeddings
\[\begin{array}{cc}\begin{array}{ccc}Gr(k,n-2)&\longrightarrow&Gr(k,n)\\\Lambda&\longmapsto&\Lambda\end{array}&\begin{array}{ccc}Gr(k-2,n-2)&\longrightarrow&S\\\Lambda&\longmapsto&\Lambda+ker(p)\end{array}.\end{array}\]
The map
\[\begin{array}{ccc}Gr(k,n)\setminus S&\longrightarrow&Gr(k,n-2)\\\Lambda&\longmapsto&p(\Lambda)\end{array}\]
(resp.
\[\begin{array}{ccc}S\setminus Gr(k-2,n-2)&\longrightarrow&Gr(1,2)\times Gr(k-1,n-2)\\\Lambda&\longmapsto&(\Lambda\cap ker(p),p(\Lambda))\end{array})\]
corresponds to the vector bundle \(\mathscr{U}_{k,n-2}\oplus\mathscr{U}_{k,n-2}\) (resp \(p_2^*\mathscr{U}_{k-1,n-2}\)) where \(p_2:Gr(1,2)\times Gr(k-1,n-2)\longrightarrow Gr(k-1,n-2)\) is the projection.
\end{proposition}
\begin{proof}
Note that
\[S\setminus Gr(k-2,n-2)=\{V\in Gr(k,n)|dim(V\cap ker(p))=1\}.\]
Suppose that \(\{t_1,t_2\}\) is a basis of \(ker(p)\), \(T=span(t_1)\) and that \(\Lambda\in Gr(k-1,n-2)\). We have an identification
\[\begin{array}{ccc}\Lambda^*&\longrightarrow&\{V\in S\setminus Gr(k-2,n-2)|V\cap ker(p)=T, p(V)=\Lambda\}\\f&\longmapsto&\{x+f(x)t_2|x\in\Lambda\}+T\end{array}\]
which is independent of the choice of \(t_2\). So \(S\setminus Gr(k-2,n-2)=p_2^*\mathscr{U}_{k-1,n-2}\). Similarly, suppose \(\Lambda\in Gr(k,n-2)\). The identification
\[\begin{array}{ccc}\Lambda^*\oplus\Lambda^*&\longrightarrow&\{V\in Gr(k,n)\setminus S|p(V)=\Lambda\}\\(f,g)&\longmapsto&\{z+f(z)t_1+g(z)t_2|z\in\Lambda\}\end{array}\]
shows that \(Gr(k,n)\setminus S=\mathscr{U}_{k,n-2}\oplus\mathscr{U}_{k,n-2}\).
\end{proof}
It turns out that \(E^{{*}}(-,-)\) admits a projective bundle theorem as in \cite{N1}.
\begin{proposition}\label{pbtsq}
Suppose that \(X\in Sm/F\), \(\mathscr{L}\in Pic(X)\) and that \(\mathscr{E}\) is a vector bundle of rank \(n\) on \(X\). Denote by \(p:\mathbb{P}(\mathscr{E})\longrightarrow X\) the structure map.
\begin{enumerate}
\item If \(n\) is odd, we have
\[E^i(\mathbb{P}(\mathscr{E}),p^*\mathscr{L})\cong E^i(X,\mathscr{L})\]
\[E^i(X,\mathscr{L})\xrightarrow[\cong]{e(\Omega_p(1))\cdot p^*}E^{i+n-1}(\mathbb{P}(\mathscr{E}),p^*(\mathscr{L}\otimes det(\mathscr{E})^{\vee})(1))\]
for any \(i\).
\item If \(n\) is even, for any \(i\) we have an exact sequence
\[\xrightarrow{\cdot e(\mathscr{E})}E^i(X,\mathscr{L})\xrightarrow{p^*}E^i(\mathbb{P}(\mathscr{E}),p^*\mathscr{L})\xrightarrow{p_*}E^{i-n+1}(X,\mathscr{L}\otimes det(\mathscr{E})^{\vee})\xrightarrow{\cdot e(\mathscr{E})}.\]
If \(e(\mathscr{E})=0\in E^n(X,det(\mathscr{E})^{\vee})\), the above sequence gives a decomposition
\[E^i(\mathbb{P}(\mathscr{E}),p^*\mathscr{L})\cong E^i(X,\mathscr{L})\oplus E^{i-n+1}(X,\mathscr{L}\otimes det(\mathscr{E})^{\vee})\cdot R(\mathscr{E})\]
for any \(i\). Here
\[R(\mathscr{E})=e(\mathscr{U}_{\mathscr{E}}^{\perp})+p^*(u)\in E^{n-1}(\mathbb{P}(\mathscr{E}),p^*det(\mathscr{E})^{\vee})\]
is the orientation class of \(\mathscr{E}\) where \(u\) is any element such that \(Sq^2_{det(\mathscr{E})^{\vee}}(u)=e(\mathscr{E})\).
\end{enumerate}
\end{proposition}
\begin{proof}
Any element \(x\in Ch^m(\mathbb{P}(\mathscr{E}))\) can be uniquely written as \(\sum_{i=0}^{n-1}a_ic^{n-1-i}\) where \(c=c_1(O_E(1))\) and \(a_i\in Ch^{m-n+i+1}(X)\) by the projective bundle theorem of Chow groups. Suppose \(x\in E^m(\mathbb{P}(\mathscr{E}),p^*\mathscr{L})\) with \(\{a_i\}\) determined as above.
\begin{enumerate}
\item We have
\[Sq^2_{\mathscr{L}}(x)=\sum_{i=0}^{n-1}Sq^2_{\mathscr{L}}(a_i)c^{n-1+i}+\sum_{i=1}^{\frac{n-1}{2}}a_{2i-1}c^{n-2i+1}\]
So \(Sq^2_{\mathscr{L}}(x)=0\) is equivalent to the equations
\[\begin{cases}
Sq^2_{\mathscr{L}}(a_i)=a_{i+1}	&\textrm{if \(i<n-1\) even}\\
Sq^2_{\mathscr{L}}(a_i)=0			&\textrm{else}
\end{cases}.\]
If furthermore \(a_{n-1}=0\), we have
\[Sq^2_{\mathscr{L}}(\sum_{i=0}^{\frac{n-3}{2}}a_{2i}c^{n-2-2i})=x.\]
So it is easy to conclude that the pullback \(E^i(X,\mathscr{L})\longrightarrow E^i(\mathbb{P}(\mathscr{E}),p^*\mathscr{L})\) is an isomorphism. The second statement follows similarly.
\item The \(Sq^2_{\mathscr{L}}(x)=0\) is equivalent to the equations
\[\begin{cases}
Sq^2_{\mathscr{L}}(a_i)+a_0c_{i+1}(\mathscr{E})=0				&\textrm{\(i\) even}\\
Sq^2_{\mathscr{L}}(a_i)+a_0c_{i+1}(\mathscr{E})+a_{i+1}=0	&\textrm{\(i<n-1\) odd}\\
Sq^2_{\mathscr{L}}(a_{n-1})+a_0e(\mathscr{E})=0				&
\end{cases}.\]
Suppose \(a_0=a_{n-1}=0\), we have
\[Sq^2_{\mathscr{L}}(\sum_{i=0}^{\frac{n}{2}-2}a_{2i+1}c^{n-3-2i})=x.\]
Since \(p_*(x)=a_0\), the first statement follows by direct computation. Now suppose \(e(\mathscr{E})=0\). It is easy to check that
\[\begin{array}{cc}Sq^2_{\mathscr{L}}(a_0R(\mathscr{E}))=0&Sq^2_{\mathscr{L}}(a_{n-1}+a_0u)=0\end{array}.\]
So \(x=a_0R(\mathscr{E})+a_{n-1}+a_0u\) in \(E^i(\mathbb{P}(\mathscr{E}),p^*\mathscr{L})\). On the other hand, if
\[a_0R(\mathscr{E})+a_{n-1}+a_0u=0\in E^m(\mathbb{P}(\mathscr{E}),p^*\mathscr{L}),\]
we have
\[a_0=0\in E^{m-n+1}(X,\mathscr{L}\otimes det(\mathscr{E})^{\vee}).\]
Thus \(a_{n-1}=0\in E^m(\mathbb{P}(\mathscr{E}),p^*\mathscr{L})\). So there are \(b, b_0\) such that
\[Sq^2_{\mathscr{L}}(b)+b_0e(\mathscr{E})=a_{n-1}\]
by the calculation of \(Sq^2_{\mathscr{L}}(x)\) before, where \(Sq^2_{\mathscr{L}\otimes det(\mathscr{E})^{\vee}}(b_0)=0\). But again we have
\[b_0e(\mathscr{E})=Sq^2_{\mathscr{L}}(b_0u)\]
so we conclude that \(a_{n-1}=0\in E^m(X,\mathscr{L})\). The second statement then follows.
\end{enumerate}
\end{proof}
\begin{definition}\label{flagdef}
Given any vector bundle \(\mathscr{E}\) on a scheme \(X\), we take the projective bundle \(p_1:Y_1=\mathbb{P}(\mathscr{E})\longrightarrow X\). Then inductively define \(p_i:Y_i=\mathbb{P}(\Omega_{p_{i-1}}(1))\longrightarrow Y_{i-1}\). The
\[Fl(\mathscr{E})=Y_{rk(\mathscr{E})-1}\]
is defined as the complete flag of \(\mathscr{E}\), parametrizing filtrations
\[\mathscr{E}_1\subseteq\cdots\subseteq \mathscr{E}_i\subseteq\cdots\subseteq \mathscr{E}\]
where \(rk(\mathscr{E}_i)=i\) and \(\mathscr{E}_i/\mathscr{E}_{i-1}\) is locally free.
\end{definition}
\begin{proposition}\label{flagthom}
Suppose \(X\in Sm/F\) and \(\mathscr{E}\) is a vector bundle on \(X\). We have structure maps
\[\mathbb{P}(\Omega_p(1))\xrightarrow{q}\mathbb{P}(\mathscr{E})\xrightarrow{p}X.\]
For any \(M\in Pic(X)\), we have an isomorphism
\[Th(q^*(p^*M(1)))=Th((q^*p^*M)(1))\]
in \(\widetilde{DM}(X,\mathbb{Z})\).
\end{proposition}
\begin{proof}
Denote by \(r\) the structure map \(\mathbb{P}(p^*\mathscr{E})\longrightarrow\mathbb{P}(\mathscr{E})\). The Euler sequence
\[0\longrightarrow\Omega_p(1)\longrightarrow p^*\mathscr{E}\longrightarrow O_E(1)\longrightarrow 0\]
implies that \(\mathbb{P}(p^*\mathscr{E})\) has a section \(\mathbb{P}(\mathscr{E})\longrightarrow\mathbb{P}(p^*\mathscr{E})\), whose complement \(C\) admits an \(\mathbb{A}^1\)-bundle \(t:C\longrightarrow\mathbb{P}(\Omega_p(1))\). There is an automorphism \(\sigma\) of \(\mathbb{P}(p^*\mathscr{E})\) over \(X\) which locally comes from the swaping map
\[\begin{array}{ccc}X\times\mathbb{P}^{rk(\mathscr{E})-1}\times\mathbb{P}^{rk(\mathscr{E})-1}&\longrightarrow&X\times\mathbb{P}^{rk(\mathscr{E})-1}\times\mathbb{P}^{rk(\mathscr{E})-1}\\(x,(y_i),(z_i))&\longmapsto&(x,(z_i),(y_i))\end{array}.\]
The \(\sigma\), which satisfies \(\sigma^*(O_{p^*\mathscr{E}}(1))=r^*(O_{\mathscr{E}}(1))\), induces an automorphism of \(C\) because locally \(C\) is defined by the equation
\[(y_i)\neq(z_i).\]
So we have
\[Th((r^*O_{\mathscr{E}}(1))|_C)=Th(O_{p^*\mathscr{E}}(1)|_C)\]
in \(\widetilde{DM}(X,\mathbb{Z})\) via \(\sigma\). Moreover, we have
\[O_{p^*\mathscr{E}}(1)|_C=t^*O_{\Omega_p(1)}(1),\]
which implies that
\[Th(O_{\Omega_p(1)}(1))=Th(q^*O_{\mathscr{E}}(1))\]
in \(\widetilde{DM}(X,\mathbb{Z})\). So the statement follows.
\end{proof}
\begin{proposition}\label{flag2sq}
Suppose that \(X\in Sm/F\), \(\mathscr{L}\in Pic(X)\) and that \(\mathscr{E}\) is a vector bundle of odd rank \(n\) on \(X\). Denote by \(p:\mathbb{P}(\mathscr{E})\longrightarrow X\), \(q:\mathbb{P}(\Omega_p(1))\longrightarrow\mathbb{P}(\mathscr{E})\) the structure maps. If \(p_{n-1}(\mathscr{E})=0\in E^{2n-2}(X)\), we have
\[E^i(\mathbb{P}(\Omega_p(1)),q^*p^*\mathscr{L})\cong E^i(X,\mathscr{L})\oplus E^{i-2n+3}(X,\mathscr{L})\cdot T.\]
Here
\[T=q^*(e(\Omega_p(1)))e(\Omega_q(1))+q^*(u)\in E^{2n-3}(\mathbb{P}(\Omega_p(1)))\]
where \(u\) is any element such that \(Sq^2(u)=e(\Omega_p(1))^2\) in \(Ch^{2n-2}(\mathbb{P}(\mathscr{E}))\).
\end{proposition}
\begin{proof}
The Whitney sum formula for even Pontryagin classes tells us that
\[p^*(p_{n-1}(\mathscr{E}))=p_{n-1}(\Omega_p(1))=e(\Omega_p(1))^2\]
hence they all vanish. The composite
\[E^i(X,\mathscr{L})\xrightarrow[\cong]{e(\Omega_p(1))\cdot p^*}E^{i+n-1}(\mathbb{P}(\mathscr{E}),p^*(det(\mathscr{E})(1)\otimes\mathscr{L}))\xrightarrow{\cdot e(\Omega_p(1))}E^{i+2n-2}(\mathbb{P}(\mathscr{E}),p^*\mathscr{L})\]
is zero by \(e(\Omega_p(1))^2=0\), so the second arrow is zero for any \(i\). Hence we have an exact sequence
\[0\longrightarrow E^i(\mathbb{P}(\mathscr{E}),p^*\mathscr{L})\xrightarrow{q^*}E^i(\mathbb{P}(\Omega_p(1)),q^*p^*\mathscr{L})\xrightarrow{q_*}E^{i-n+2}(\mathbb{P}(\mathscr{E}),p^*(det(\mathscr{E})\otimes\mathscr{L})(1))\longrightarrow 0\]
for any \(i\) by Proposition \ref{pbtsq}. So there is an element \(T\in E^{2n-3}(\mathbb{P}(\Omega_p(1)))\) such that \(q_*(T)=e(\Omega_p(1))\). Then for any \(x\in E^{i-2n+3}(X,\mathscr{L})\), we have
\[q_*(q^*p^*(x)\cdot T)=p^*(x)\cdot e(\Omega_p(1)).\]
Since \(n\) is odd, the sequence above is thus isomorphic the sequence
\[0\longrightarrow E^i(X,\mathscr{L})\xrightarrow{q^*p^*}E^i(\mathbb{P}(\Omega_p(1)),q^*p^*\mathscr{L})\xrightarrow{(e(\Omega_p(1))\cdot p^*)^{-1}\circ q_*}E^{i-2n+3}(X,\mathscr{L})\longrightarrow 0.\]
by \textit{loc. cit.} where the third arrow has a section \(T\cdot q^*p^*\). So we obtain the first statement.

For the second statement, it suffices to check \(q_*(T)=e(\Omega_p(1))\) and \(Sq^2(T)=0\), which are straightforward.
\end{proof}
\begin{proposition}\label{flagsq}
Suppose that \(X\in Sm/F\), \(\mathscr{L}\in Pic(X)\) and that \(\mathscr{E}\) is a vector bundle of rank \(n\) on \(X\). Denote by \(p:Fl(\mathscr{E})\longrightarrow X\) the structure map. Suppose \(p_i(\mathscr{E})=0\in E^{2i}(X)\) for all \(i\).
\begin{enumerate}
\item If \(n\) is odd, we have classes \(T_a\in E^{4a-1}(Fl(\mathscr{E})),1\leq a\leq\frac{n-1}{2}\) such that
\[E^i(Fl(\mathscr{E}),p^*\mathscr{L})=\bigoplus_{1\leq t\leq\frac{n-1}{2}}\bigoplus_{1\leq a_1<\cdots<a_t\leq\frac{n-1}{2}}E^{i-\sum_adeg(T_a)}(X,\mathscr{L})\cdot T_{a_1}\cdots T_{a_t}.\]
\item If \(n\) is even and \(e(\mathscr{E})=0\), we have classes \(T_a\in E^{4a-1}(Fl(\mathscr{E})),1\leq a\leq\frac{n}{2}-1\) and \(T_{\frac{n}{2}}\in E^{n-1}(Fl(\mathscr{E}))\) such that
\[E^i(Fl(\mathscr{E}),p^*\mathscr{L})=\bigoplus_{1\leq t\leq\frac{n}{2}}\bigoplus_{1\leq a_1<\cdots<a_t\leq\frac{n}{2}}E^{i-\sum_adeg(T_a)}(X,\mathscr{L})\cdot T_{a_1}\cdots T_{a_t}.\]
\end{enumerate}
\end{proposition}
\begin{proof}
Follows from Proposition \ref{pbtsq} and Proposition \ref{flag2sq}.
\end{proof}
\begin{remark}\label{flaggen}
In case \(X=pt\), we could write down the \(\{T_a\}\) explicitly. In the context of Definition \ref{flagdef}, denote by \(x_i,1\leq i\leq n-1\) the first Chern class of tautological line bundle on \(Y_i\) and \(x_n\) by that of complement bundle on \(Y_{n-1}\). Then
\[Ch^{{*}}(Fl(F^{\oplus n}))=\frac{\mathbb{Z}/2[x_1,\cdots,x_n]}{\mathcal{S}}\]
where \(\mathcal{S}\) is the ideal generated by symmetric functions, with \(Sq^2(x_i)=x_i^2\) for all \(i\). By inductively using Whitney sum formula, one shows that
\[c_i(\Omega_{p_j}(1))=h_i(x_1,\cdots,x_j)\]
for every \(i,j\).

Suppose \(1\leq a\leq\lfloor\frac{n}{2}\rfloor\). If \(n\) is odd or \(n\) is even and \(a<\frac{n}{2}\), by Proposition \ref{flag2sq}, we have
\[T_a=h_{2a}(x_1,\cdots,x_{n-2a})h_{2a-1}(x_1,\cdots,x_{n-2a+1})+u(x_1,\cdots,x_{n-2a})\]
where \(u\) satisfies \(Sq^2(u)=h_{2a}(x_1,\cdots,x_{n-2a})^2\).

If \(n\) is even, we have
\[T_{\frac{n}{2}}=x_1^{n-1}.\]

It is clear that \(T_a^2=0\in E^{{*}}(Fl(F^{\oplus n}))\) (but not true for general \(X\) and \(\mathscr{E}\)). So we conclude
\[E^{{*}}(Fl(F^{\oplus n}))=\wedge[\{T_a\}].\]
\end{remark}
\section{MW-Motivic Decomposition of Grassmannians}\label{MW-Motivic Decomposition of Grassmannians}
Recall that a Young diagram (see \cite[3.1]{W2}) is a collection of left aligned rows of boxes with (non strictly) decreasing row lengths. We denote by \((\Lambda_1,\cdots,\Lambda_n), \Lambda_1\geq\cdots\geq \Lambda_n>0\) a Young diagram \(\Lambda\) whose \(i\)-th row has \(\Lambda_i\) boxes. For example, the diagram
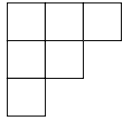
\[\begin{tikzpicture}\filldraw[color=black,fill=white](0,0)rectangle(0.5,-0.5);\filldraw[color=black,fill=white](0,-0.5)rectangle(0.5,-1);\filldraw[color=black,fill=white](0,-1)rectangle(0.5,-1.5);\filldraw[color=black,fill=white](0.5,0)rectangle(1,-0.5);\filldraw[color=black,fill=white](1,0)rectangle(1.5,-0.5);\filldraw[color=black,fill=white](0.5,-0.5)rectangle(1,-1);\end{tikzpicture}\]
is denoted by \((3,2,1)\).

For two Young diagrams \(S,\Lambda\), we say \(S\leq\Lambda\) if \(S_i\leq\Lambda_i\) for all \(i\). If \(S\leq\Lambda\), their difference is the skew Young diagram \(\Lambda\setminus S\).
\begin{definition}\label{Young}
Suppose \(k,n\in\mathbb{N}\) and \(k\leq n\). We say a Young diagram \((a_1,\cdots,a_l)\) is \((k,n)\)-truncated (resp. untruncated) if it is identified with the Schubert cycle \(\sigma_{a_1,\cdots,a_l}\) in \(CH^{\sum a_i}(Gr(k,n))\) (resp. \(Gr(\infty,\infty)\)).
\end{definition}
There is a truncation morphism
\[t_{k,n}:R\{\textrm{untruncated Young diagrams}\}\longrightarrow R\{(k,n)\textrm{-truncated Young diagrams}\}\]
induced by the pullback morphism \(CH^*(Gr(\infty,\infty))\otimes R\longrightarrow CH^*(Gr(k,n))\otimes R\), killing all diagrams exceeding the size \((k,n)\). It is compatible with products and Steenrod operations. For example, we have
\[\sigma_1^2=\sigma_{1,1}+\sigma_2\]
as untruncated diagrams whereas
\[\sigma_1^2=\sigma_2\]
as \((1,3)\)-truncated diagrams. All Young diagrams we will consider are filled by a chessboard pattern (see \cite[Theorem 4.2]{W2}). If the first box in the first row of a diagram is black (resp. white), it is called untwisted (resp. twisted). The following is an example (both are \((3,6)\)-truncated):
\[\begin{array}{cc}\textrm{Untwisted Young diagram }T&\textrm{Twisted Young diagram }T'\\
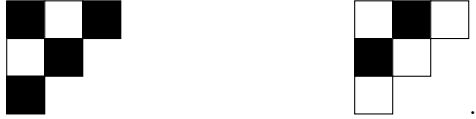
\begin{tikzpicture}\filldraw[color=black,fill=black](0,0)rectangle(0.5,-0.5);\filldraw[color=black,fill=white](0,-0.5)rectangle(0.5,-1);\filldraw[color=black,fill=black](0,-1)rectangle(0.5,-1.5);\filldraw[color=black,fill=white](0.5,0)rectangle(1,-0.5);\filldraw[color=black,fill=black](1,0)rectangle(1.5,-0.5);\filldraw[color=black,fill=black](0.5,-0.5)rectangle(1,-1);\end{tikzpicture}&
\begin{tikzpicture}\filldraw[color=black,fill=white](0,0)rectangle(0.5,-0.5);\filldraw[color=black,fill=black](0,-0.5)rectangle(0.5,-1);\filldraw[color=black,fill=white](0,-1)rectangle(0.5,-1.5);\filldraw[color=black,fill=black](0.5,0)rectangle(1,-0.5);\filldraw[color=black,fill=white](1,0)rectangle(1.5,-0.5);\filldraw[color=black,fill=white](0.5,-0.5)rectangle(1,-1);\end{tikzpicture}.\end{array}\]
\begin{definition}\label{Young}
Suppose \(\Lambda\) is a Young diagram (twisted or not). Define
\[A(\Lambda)=\{\Lambda'\geq\Lambda|\Lambda'\setminus\Lambda=\textrm{a white box}\}\]
\[\bar{A}(\Lambda)=\{\Lambda'\geq\Lambda|\Lambda'\setminus\Lambda=\textrm{white boxes}\}\]
\[D(\Lambda)=\{\Lambda'\leq\Lambda|\Lambda\setminus\Lambda'=\textrm{a white box}\}.\]
Here all diagrams are subjected to the same truncating and coloring rule. We say that \(\Lambda\) is irredundant (resp. full) if \(D(\Lambda)=\emptyset\) (resp. \(A(\Lambda)=\emptyset\)).
\end{definition}
For example, the \(T\) and \(T'\) above satisfy
\[A(T)=\{(3,3,1),(3,2,2)\},\bar{A}(T)=\{(3,2,1),(3,3,1),(3,2,2),(3,3,2)\},D(T)=\emptyset;\]
\[A(T')=\emptyset,\bar{A}(T')=\{(3,2,1)\}, D(T')=\{(2,2,1),(3,1,1),(3,2)\}.\]
So \(T\) is irredundant and \(T'\) is full. Note here that \((4,2,1)\notin A(T)\) because \(T\) is \((3,6)\)-truncated.

Hence we see that the Steenrod square \(Sq^2_{\mathscr{L}}:Ch^{i-1}(Gr(k,n))\longrightarrow Ch^i(Gr(k,n))\) can be written as
\[Sq^2_{\mathscr{L}}(T)=\sum_{\Lambda\in A(T)}\Lambda\]
by \cite[Theorem 4.2]{W2}. In this way, the twisted (resp. untwisted) diagrams correspond to the case \(\mathscr{L}=O(1)\). (resp. \(\mathscr{L}=O_{Gr(k,n)}\))
\begin{lemma}\label{ce1}
Suppose that all diagrams are untwisted.
\begin{enumerate}
\item An untruncated Young diagram \(\Lambda\) is irredundant and full if and only if it is completely even (see \cite[Definition 3.8]{W2}).
\item If \(k(n-k)\) is even, a \((k,n)\)-truncated Young diagram \(\Lambda\) is irredundant and full if and only if it is completely even.
\item If \(k(n-k)\) is odd, a \((k,n)\)-truncated Young diagram \(\Lambda\) is irredundant and full if and only if it is completely even or of the form
\[\sigma_{\cdots,2a_i,2a_i,\cdots}\sigma_{n-k,1^{k-1}}=\sigma_{n-k,\cdots,2a_i+1,2a_i+1,\cdots}.\]
\end{enumerate}
\end{lemma}
\begin{proof}
It suffices to prove the necessity for all those statements.
\begin{enumerate}
\item Suppose \(\lambda_i\) is the last box in the \(i^{th}\) row of \(\Lambda\). Then either \(\lambda_i\) is black with a row of same length adjacently above it or \(\lambda_i\) is white with a row of same length adjacently under it. So we see that \(\Lambda\) is completely even. The converse is clear.
\item If \(n-k\) is even, the first row of \(\Lambda\) must have even length otherwise it is not full, which implies that the second row exists and must have the same length. Then every row started with a black box must have even length, being adjacently above a row with same length. So \(\Lambda\) is completely even.

If \(k\) is even, suppose the \(i^{th}\) row is the first row whose length is even. If no such row exists, the number of rows must be odd otherwise \(\Lambda\) is not irredundant. But we could still add a white box in the first empty row so it is not full. So such \(i\) exists. If \(i\) is even, \(\Lambda\) is not full. If \(i>1\) is odd, \(\Lambda\) is not irredundant. Hence \(i=1\) and \(\Lambda\) is completely even as above. So we have completed the proof.
\item If the length of the first row is even, then \(\Lambda\) is complete even as above. Otherwise its length must be \(n-k\). So the second row must have odd length otherwise \(\Lambda\) is not full. Moreover, there is a row adjacently under it with the same length otherwise \(\Lambda\) is not irredundant. By induction we see \(\Lambda\) must be the form \(\sigma_{n-k,\cdots,2a_i+1,2a_i+1,\cdots}\).
\end{enumerate}
\end{proof}
\begin{lemma}\label{ce2}
Suppose that all diagrams are twisted.
\begin{enumerate}
\item If \(k\) and \(n\) are even, a \((k,n)\)-truncated Young diagram \(\Lambda\) is irredundant and full if and only if it is like \(\sigma_{n-k}T\) or \(\sigma_{1^k}T\) where \(T\) is completely even.
\item If \(k\) and \(n\) are odd, a \((k,n)\)-truncated Young diagram \(\Lambda\) is irredundant and full if and only if it is like \(\sigma_{n-k}T\) where \(T\) is completely even.
\item If \(k\) is even and \(n\) is odd, a \((k,n)\)-truncated Young diagram \(\Lambda\) is irredundant and full if and only if it is like \(\sigma_{1^k}T\) where \(T\) is completely even.
\item If \(k\) is odd and \(n\) is even, there is no \((k,n)\)-truncated Young diagram \(\Lambda\) which is both irredundant and full.
\end{enumerate}
\end{lemma}
\begin{proof}
\begin{enumerate}
\item It suffices to prove the necessity. If the first row has odd length, there must be a row with same length adjacently under it. Inductively we see that \(\Lambda=\sigma_{1^k}T\) where \(T\) is completely even. Otherwise it has length \(n-k\). The second row must have even length otherwise \(\Lambda\) is not full, being adjacently above a row with same length. Inductively we see that \(\Lambda=\sigma_{n-k}T\) where \(T\) is completely even.
\item The same reasoning as in (1).
\item The same reasoning as in (1).
\item The same reasoning as in (1).
\end{enumerate}
\end{proof}
\begin{definition}\label{even}
Slightly abusing the notation, we say that a Young diagram \(\Lambda\) (twisted or not) is even if it is irredundant and full.

Denote the free abelian group generated by the set of all untwisted (resp. twisted) even diagrams by \(N_{k,n}\) (resp. \(M_{k,n}\)), regarded as a subgroup of \(CH^{{*}}(Gr(k,n))\). Its degree \(i\) part is denoted by \(N_{k,n}^i\) (resp. \(M_{k,n}^i\)).
\end{definition}
Our definition of even Young diagram is essentially the same as \cite[Definition 3.8]{W2} by Lemma \ref{ce1} and \ref{ce2}, but it depends on \(k,n\) and whether the diagram is twisted.

\begin{remark}\label{ex1}
Here are the irredundant untwisted Young diagrams in \(Gr(2,4)\) or \(Gr(3,6)\):
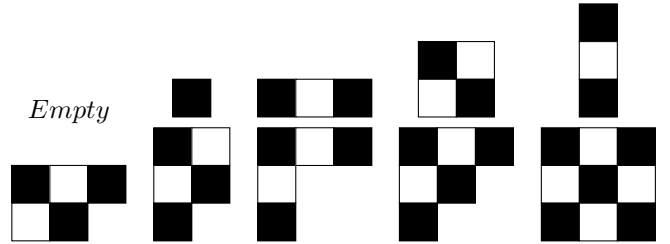
\[\begin{array}{ccccc}
Empty&
\begin{tikzpicture}\filldraw[color=black,fill=black](0,0)rectangle(0.5,-0.5);\end{tikzpicture}&
\begin{tikzpicture}\filldraw[color=black,fill=black](0,0)rectangle(0.5,-0.5);\filldraw[color=black,fill=white](0.5,0)rectangle(1,-0.5);\filldraw[color=black,fill=black](1,0)rectangle(1.5,-0.5);\end{tikzpicture}&
\begin{tikzpicture}\filldraw[color=black,fill=black](0,0)rectangle(0.5,-0.5);\filldraw[color=black,fill=white](0.5,0)rectangle(1,-0.5);\filldraw[color=black,fill=white](0,-0.5)rectangle(0.5,-1);\filldraw[color=black,fill=black](0.5,-0.5)rectangle(1,-1);\end{tikzpicture}&
\begin{tikzpicture}\filldraw[color=black,fill=black](0,0)rectangle(0.5,-0.5);\filldraw[color=black,fill=white](0,-0.5)rectangle(0.5,-1);\filldraw[color=black,fill=black](0,-1)rectangle(0.5,-1.5);\end{tikzpicture}\\
\begin{tikzpicture}\filldraw[color=black,fill=black](0,0)rectangle(0.5,-0.5);\filldraw[color=black,fill=white](0.5,0)rectangle(1,-0.5);\filldraw[color=black,fill=white](0,-0.5)rectangle(0.5,-1);\filldraw[color=black,fill=black](0.5,-0.5)rectangle(1,-1);\filldraw[color=black,fill=black](1,0)rectangle(1.5,-0.5);\end{tikzpicture}&
\begin{tikzpicture}\filldraw[color=black,fill=black](0,0)rectangle(0.5,-0.5);\filldraw[color=black,fill=white](0.5,0)rectangle(1,-0.5);\filldraw[color=black,fill=white](0,-0.5)rectangle(0.5,-1);\filldraw[color=black,fill=black](0.5,-0.5)rectangle(1,-1);\filldraw[color=black,fill=black](0,-1)rectangle(0.5,-1.5);\end{tikzpicture}&
\begin{tikzpicture}\filldraw[color=black,fill=black](0,0)rectangle(0.5,-0.5);\filldraw[color=black,fill=white](0,-0.5)rectangle(0.5,-1);\filldraw[color=black,fill=black](0,-1)rectangle(0.5,-1.5);\filldraw[color=black,fill=white](0.5,0)rectangle(1,-0.5);\filldraw[color=black,fill=black](1,0)rectangle(1.5,-0.5);\end{tikzpicture}&
\begin{tikzpicture}\filldraw[color=black,fill=black](0,0)rectangle(0.5,-0.5);\filldraw[color=black,fill=white](0,-0.5)rectangle(0.5,-1);\filldraw[color=black,fill=black](0,-1)rectangle(0.5,-1.5);\filldraw[color=black,fill=white](0.5,0)rectangle(1,-0.5);\filldraw[color=black,fill=black](1,0)rectangle(1.5,-0.5);\filldraw[color=black,fill=black](0.5,-0.5)rectangle(1,-1);\end{tikzpicture}&
\begin{tikzpicture}\filldraw[color=black,fill=black](0,0)rectangle(0.5,-0.5);\filldraw[color=black,fill=white](0,-0.5)rectangle(0.5,-1);\filldraw[color=black,fill=black](0,-1)rectangle(0.5,-1.5);\filldraw[color=black,fill=white](0.5,0)rectangle(1,-0.5);\filldraw[color=black,fill=black](1,0)rectangle(1.5,-0.5);\filldraw[color=black,fill=black](0.5,-0.5)rectangle(1,-1);\filldraw[color=black,fill=white](1,-0.5)rectangle(1.5,-1);\filldraw[color=black,fill=white](0.5,-1)rectangle(1,-1.5);\filldraw[color=black,fill=black](1,-1)rectangle(1.5,-1.5);\end{tikzpicture}
\end{array}.\]
Here are the generators of \(N_{2,4}\) and \(N_{3,6}\):
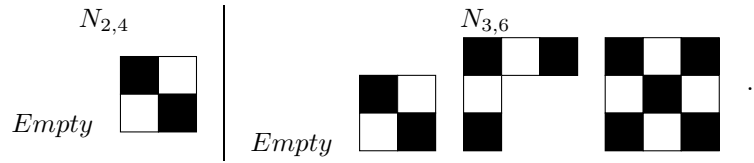
\[\begin{array}{c|c}N_{2,4}&N_{3,6}\\\begin{array}{cc}Empty&\begin{tikzpicture}\filldraw[color=black,fill=black](0,0)rectangle(0.5,-0.5);\filldraw[color=black,fill=white](0.5,0)rectangle(1,-0.5);\filldraw[color=black,fill=white](0,-0.5)rectangle(0.5,-1);\filldraw[color=black,fill=black](0.5,-0.5)rectangle(1,-1);\end{tikzpicture}\end{array}&\begin{array}{cccc}Empty&\begin{tikzpicture}\filldraw[color=black,fill=black](0,0)rectangle(0.5,-0.5);\filldraw[color=black,fill=white](0.5,0)rectangle(1,-0.5);\filldraw[color=black,fill=white](0,-0.5)rectangle(0.5,-1);\filldraw[color=black,fill=black](0.5,-0.5)rectangle(1,-1);\end{tikzpicture}&\begin{tikzpicture}\filldraw[color=black,fill=black](0,0)rectangle(0.5,-0.5);\filldraw[color=black,fill=white](0,-0.5)rectangle(0.5,-1);\filldraw[color=black,fill=black](0,-1)rectangle(0.5,-1.5);\filldraw[color=black,fill=white](0.5,0)rectangle(1,-0.5);\filldraw[color=black,fill=black](1,0)rectangle(1.5,-0.5);\end{tikzpicture}&\begin{tikzpicture}\filldraw[color=black,fill=black](0,0)rectangle(0.5,-0.5);\filldraw[color=black,fill=white](0,-0.5)rectangle(0.5,-1);\filldraw[color=black,fill=black](0,-1)rectangle(0.5,-1.5);\filldraw[color=black,fill=white](0.5,0)rectangle(1,-0.5);\filldraw[color=black,fill=black](1,0)rectangle(1.5,-0.5);\filldraw[color=black,fill=black](0.5,-0.5)rectangle(1,-1);\filldraw[color=black,fill=white](1,-0.5)rectangle(1.5,-1);\filldraw[color=black,fill=white](0.5,-1)rectangle(1,-1.5);\filldraw[color=black,fill=black](1,-1)rectangle(1.5,-1.5);\end{tikzpicture}\end{array}\end{array}.\]
Here are the irredundant twisted Young diagrams in \(Gr(2,4)\) or \(Gr(3,6)\):
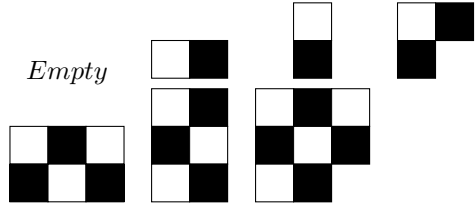
\[\begin{array}{cccc}
Empty&
\begin{tikzpicture}\filldraw[color=black,fill=white](0,0)rectangle(0.5,-0.5);\filldraw[color=black,fill=black](0.5,0)rectangle(1,-0.5);\end{tikzpicture}&
\begin{tikzpicture}\filldraw[color=black,fill=white](0,0)rectangle(0.5,-0.5);\filldraw[color=black,fill=black](0,-0.5)rectangle(0.5,-1);\end{tikzpicture}&
\begin{tikzpicture}\filldraw[color=black,fill=white](0,0)rectangle(0.5,-0.5);\filldraw[color=black,fill=black](0,-0.5)rectangle(0.5,-1);\filldraw[color=black,fill=black](0.5,0)rectangle(1,-0.5);\end{tikzpicture}\\
\begin{tikzpicture}\filldraw[color=black,fill=white](0,0)rectangle(0.5,-0.5);\filldraw[color=black,fill=black](0,-0.5)rectangle(0.5,-1);\filldraw[color=black,fill=black](0.5,0)rectangle(1,-0.5);\filldraw[color=black,fill=white](1,0)rectangle(1.5,-0.5);\filldraw[color=black,fill=white](0.5,-0.5)rectangle(1,-1);\filldraw[color=black,fill=black](1,-0.5)rectangle(1.5,-1);\end{tikzpicture}&
\begin{tikzpicture}\filldraw[color=black,fill=white](0,0)rectangle(0.5,-0.5);\filldraw[color=black,fill=black](0,-0.5)rectangle(0.5,-1);\filldraw[color=black,fill=black](0.5,0)rectangle(1,-0.5);\filldraw[color=black,fill=white](0,-1)rectangle(0.5,-1.5);\filldraw[color=black,fill=white](0.5,-0.5)rectangle(1,-1);\filldraw[color=black,fill=black](0.5,-1)rectangle(1,-1.5);\end{tikzpicture}&
\begin{tikzpicture}\filldraw[color=black,fill=white](0,0)rectangle(0.5,-0.5);\filldraw[color=black,fill=black](0,-0.5)rectangle(0.5,-1);\filldraw[color=black,fill=black](0.5,0)rectangle(1,-0.5);\filldraw[color=black,fill=white](0,-1)rectangle(0.5,-1.5);\filldraw[color=black,fill=white](0.5,-0.5)rectangle(1,-1);\filldraw[color=black,fill=black](0.5,-1)rectangle(1,-1.5);\filldraw[color=black,fill=white](1,0)rectangle(1.5,-0.5);\filldraw[color=black,fill=black](1,-0.5)rectangle(1.5,-1);\end{tikzpicture}&
\end{array}.\]
Here are the generators of \(M_{2,4}\) (\(M_{3,6}=0\)):
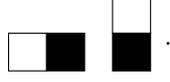
\[\begin{array}{cc}\begin{tikzpicture}\filldraw[color=black,fill=white](0,0)rectangle(0.5,-0.5);\filldraw[color=black,fill=black](0.5,0)rectangle(1,-0.5);\end{tikzpicture}&\begin{tikzpicture}\filldraw[color=black,fill=white](0,0)rectangle(0.5,-0.5);\filldraw[color=black,fill=black](0,-0.5)rectangle(0.5,-1);\end{tikzpicture}\end{array}.\]
\end{remark}
\begin{lemma}\label{decomp}
Suppose \(\mathscr{L}\in Pic(Gr(k,n))\). We have decompositions
\[CH^{{*}}(Gr(k,n))=\bigoplus_{\Lambda\textrm{ irredundant}}\mathbb{Z}\bar{A}(\Lambda);\]
\[Ch^{{*}}(Gr(k,n))=\bigoplus_{\Lambda\textrm{ irredundant}}\mathbb{Z}/2\bar{A}(\Lambda);\]
\[Ker(Sq_{\mathscr{L}}^2)=\bigoplus_{\Lambda\textrm{ irredundant}}\mathbb{Z}/2\bar{A}(\Lambda)\cap Ker(Sq_{\mathscr{L}}^2);\]
\[Im(Sq_{\mathscr{L}}^2)=\bigoplus_{\Lambda\textrm{ irredundant}}\mathbb{Z}/2\bar{A}(\Lambda)\cap Im(Sq_{\mathscr{L}}^2).\]
\end{lemma}
\begin{proof}
Clear from the definition of irredundance.
\end{proof}
\begin{proposition}\label{imsq}
Suppose \(\mathscr{L}\in Pic(Gr(k,n))\), \(\sigma=\sum_i\Lambda_i\in Ker(Sq_{\mathscr{L}}^2)\) where \(\Lambda_i\) are \((k,n)\)-truncated (twisted by \(\mathscr{L}\)) Young diagrams. Then \(\sigma\in Im(Sq_{\mathscr{L}}^2)\) if and only if none of \(\Lambda_i\) is even. In other word, we have
\[Ker(Sq^2)=Im(Sq^2)\oplus N_{k,n}/2N_{k,n}\]
\[Ker(Sq_{O(1)}^2)=Im(Sq_{O(1)}^2)\oplus M_{k,n}/2M_{k,n}.\]
\end{proposition}
\begin{proof}
It suffices to prove the equation in each component \(C=\mathbb{Z}/2\bar{A}(\Lambda)\) generated by an irredundant Young diagram \(\Lambda\) as in Lemma \ref{decomp}. Hence we may assume that \(\Lambda\) is not full and prove that
\[Ker(Sq_{\mathscr{L}}^2)\cap C=Im(Sq_{\mathscr{L}}^2)\cap C.\]
Suppose \(|A(\Lambda)|=n>0\) hence each \(\Lambda_i\) can be seen as a sequence of \(\{0,1\}\) with length \(n\) where \(1\) means `a white box added' and \(0\) otherwise. Suppose \(\sigma=A0+B1\) where \(A\) and \(B\) are of length \(n-1\) and \(A0\) (resp. \(B1\)) means concatenating \(A\) (resp. \(B\)) by \(0\) (resp. \(1\)). Then
\[0=Sq_{\mathscr{L}}^2(A0+B1)=Sq_{\mathscr{L}}^2(A)0+A1+Sq_{\mathscr{L}}^2(B)1.\]
So \(A=Sq_{\mathscr{L}}^2(B)\) thus
\[Sq_{\mathscr{L}}^2(B0)=A0+B1,\]
as desired.
\end{proof}
\begin{definition}\label{Young1}
Suppose \(\Lambda\) is an irredundant Young diagram (maybe twisted). There are \(|A(\Lambda)|\) `positions' in \(\Lambda\) on which a white box could be added. Denote by \(\Lambda_{i_1,\cdots,i_l}, 1\leq i_1<\cdots<i_l\leq|A(\Lambda)|\) (defined by \(\Lambda\) if \(l=0\)) the Young diagram with a white box added on the \(i_1^{th},\cdots,i_l^{th}\) positions of \(\Lambda\).
\end{definition}
For example, the \((3,6)\)-truncated diagram \(T=(3,2,1)\) has two positions available, namely we have
\[\begin{array}{ccc}T_1=(3,3,1)&T_2=(3,2,2)&T_{1,2}=(3,3,2)\end{array}.\]
We could also define
\[\begin{array}{cccc}Sq^2_{\mathscr{L}}:&CH^{{*}}(Gr(k,n))&\longrightarrow&CH^{{*}}(Gr(k,n))\\&\Lambda&\longmapsto&\sum_{T\in A(\Lambda)}T\end{array},\]
which is a lift of the Steenrod square.
\begin{proposition}\label{basis}
Suppose \(\Lambda\) (twisted or not) is irredundant.
\begin{enumerate}
\item If \(\Lambda\) is even, we set \(S_{\Lambda}=\{\Lambda\}\) otherwise
\[S_{\Lambda}=Sq^2_{\mathscr{L}}(\{\Lambda_{i_1,\cdots,i_l}|i_1>1\}).\]
It's a \(\mathbb{Z}/2\)-basis of \(Ker(Sq^2_{\mathscr{L}})\cap\mathbb{Z}/2\bar{A}(\Lambda)\).
\item If \(\Lambda\) is even, we set \(S_{\Lambda}=\{\Lambda\}\) otherwise
\[S_{\Lambda}=Sq^2_{\mathscr{L}}(\{\Lambda_{i_1,\cdots,i_l}|i_1>1\})\cup\{2\Lambda_{i_1,\cdots,i_l}|i_1>1\}.\]
It's a \(\mathbb{Z}\)-basis of \(Ker(Sq^2_{\mathscr{L}}\circ\pi)\cap\mathbb{Z}\bar{A}(\Lambda)\).
\end{enumerate}
\end{proposition}
\begin{proof}
We may assume \(\Lambda\) is not full.
\begin{enumerate}
\item We have
\[Ker(Sq^2_{\mathscr{L}})\cap\mathbb{Z}/2\bar{A}(\Lambda)=Im(Sq^2_{\mathscr{L}})\cap\mathbb{Z}/2\bar{A}(\Lambda)\]
by the proof of Proposition \ref{imsq}. Hence their generators are just \(Sq^2_{\mathscr{L}}(\{\Lambda_{i_1,\cdots,i_l}\})\), which is graded by \(l\). If \(l=0\), \(Sq^2_{\mathscr{L}}(\Lambda)\neq 0\) since \(\Lambda\) is not full. So we suppose \(l>0\). The set \(S\) is linearly indepedent because \(Sq^2_{\mathscr{L}}(\Lambda_{i_1,\cdots,i_l}),i_1>1\) has coefficient \(1\) in \(\Lambda_{1,j_1,\cdots,j_l}\) if and only if \((j_1,\cdots,j_l)=(i_1,\cdots,i_l)\). On the other hand, suppose \(i_2>1\). The equation
\[Sq^2_{\mathscr{L}}(Sq^2_{\mathscr{L}}(\Lambda_{i_2,\cdots,i_l}))=0\]
shows that \(Sq^2_{\mathscr{L}}(\Lambda_{1,i_2,\cdots,i_l})\in\mathbb{Z}/2S\). So we have proved that \(S_{\Lambda}\) is a basis.
\item Denote by \(C=\mathbb{Z}\{\Lambda_{i_1,\cdots,i_l}\},i_1>1\) and \(T=\{Sq^2_{\mathscr{L}}(\{\Lambda_{i_1,\cdots,i_l}|i_1>1\})\}\). So a system of generators is \(\{2\Lambda_{i_1,\cdots,i_l}\}\cup T\). For each \(t\in T\), it could be written as \(\varphi(t)+t'\) where \(\varphi(t)\in\{\Lambda_{1,i_2,\cdots,i_l}\}\) and \(t'\in C\). The \(\varphi\) is injective. If \(\sum_j\alpha_jt_j\in2\cdot C\) where \(t_j\in T\) are distinct, we have \(\sum_j\alpha_j\varphi(t_j)=0\) hence \(\alpha_j=0\) for all \(j\). So \(S_{\Lambda}\) is linearly independent. On the other hand, suppose \(i_2>1\). We have
\[\sum_{j\notin\{i_2,\cdots,i_l\}}\Lambda_{j,i_2,\cdots,i_l}\in T,\]
which shows that \(2\Lambda_{1,i_2,\cdots,i_l}\in\mathbb{Z}S_{\Lambda}\). Hence \(S_{\Lambda}\) is a basis.
\end{enumerate}
\end{proof}
\begin{coro}\label{partial}
There are decompositions
\[Ker(Sq^2_{\mathscr{L}}\circ\pi)=\bigoplus_{\Lambda\textrm{ irredundant}}\mathbb{Z}S_{\Lambda}.\]
\end{coro}
For convenience, we define
\[Ker_t(Sq^2\circ\pi)=Ker(Sq^2\circ\pi)/N_{k,n}\]
\[Ker_t(Sq^2_{O(1)}\circ\pi)=Ker(Sq^2_{O(1)}\circ\pi)/M_{k,n}.\]

Recall the notations in \cite{W1} that we define \(c_i=c_i(\mathscr{U}_{k,n})\), \(c_i^{\perp}=c_i(\mathscr{U}_{k,n}^{\perp})\), \(e_k=e(\mathscr{U}_{k,n})\), \(e_{n-k}^{\perp}=e(\mathscr{U}_{k,n}^{\perp})\) and similarly for \(p_i\) and \(p_i^{\perp}\).
\begin{proposition}\label{ce}
\begin{enumerate}
\item If \(k(n-k)\) is odd, we have isomorphisms
\[\begin{array}{cccc}\gamma:&CH^i(Gr(\lfloor\frac{k}{2}\rfloor,\frac{n}{2}-1))&\longrightarrow&N_{k,n}^{4i}\\&\sigma_{a_1,\cdots,a_l}&\longmapsto&\sigma_{\cdots,2a_j,2a_j,\cdots}\end{array}\]
\[\gamma'=\sigma_{n-k,1^{k-1}}\cdot\gamma:CH^i(Gr(\lfloor\frac{k}{2}\rfloor,\frac{n}{2}-1))\longrightarrow N_{k,n}^{4i+n-1}.\]
Moreover, there is a \(\textbf{W}(F)\)-algebra isomorphism
\[\begin{array}{cccc}\lambda:&CH^{{*}}(Gr(\lfloor\frac{k}{2}\rfloor,\frac{n}{2}-1))\otimes_{\mathbb{Z}}\textbf{W}(F)&\longrightarrow&H^{4{*}}(Gr(k,n),\textbf{W})\\&c_i&\longmapsto&p_{2i}\\&c_i^{\perp}&\longmapsto&p_{2i}^{\perp}\end{array}\]
and a group isomorphism
\[\begin{array}{cccc}\lambda':&CH^i(Gr(\lfloor\frac{k}{2}\rfloor,\frac{n}{2}-1))\otimes_{\mathbb{Z}}\textbf{W}(F)&\longrightarrow&H^{4i+n-1}(Gr(k,n),\textbf{W})\\&\sigma&\longmapsto&\lambda(\sigma)\cdot \mathcal{R}\end{array}.\]
Recall here the \(\mathcal{R}\) is the orientation class (see \cite[Theorem 6.4]{W1}).
\item If \(k(n-k)\) is even, we have an isomorphism
\[\begin{array}{cccc}\gamma:&CH^i(Gr(\lfloor\frac{k}{2}\rfloor,\lfloor\frac{n}{2}\rfloor))&\longrightarrow&N_{k,n}^{4i}\\&\sigma_{\cdots,a_j,\cdots}&\longmapsto&\sigma_{\cdots,2a_j,2a_j,\cdots}\end{array}.\]
Moreover, there is an \(\textbf{W}(F)\)-algebra isomorphism
\[\begin{array}{cccc}\lambda:&CH^{{*}}(Gr(\lfloor\frac{k}{2}\rfloor,\lfloor\frac{n}{2}\rfloor))\otimes_{\mathbb{Z}}\textbf{W}(F)&\longrightarrow&H^{4{*}}(Gr(k,n),\textbf{W})\\&c_i&\longmapsto&p_{2i}\\&c_i^{\perp}&\longmapsto&p_{2i}^{\perp}\end{array}.\]
\item If \(k\) is even and \(n\neq 0\ (mod\ 4)\), we have isomorphisms
\[\gamma_1=\sigma_{1^k}\cdot\gamma:CH^i(Gr(\lfloor\frac{k}{2}\rfloor,\lfloor\frac{n-1}{2}\rfloor))\longrightarrow M_{k,n}^{4i+k}\]
\[\begin{array}{cccc}\lambda_1:&CH^i(Gr(\lfloor\frac{k}{2}\rfloor,\lfloor\frac{n-1}{2}\rfloor))\otimes_{\mathbb{Z}}\textbf{W}(F)&\longrightarrow&H^{4i+k}(Gr(k,n),\textbf{W}(O(1)))\\&\sigma&\longmapsto&\lambda(\sigma)\cdot e_k\end{array}.\]
\item If \(n-k\) is even and \(n\neq 0\ (mod\ 4)\), we have isomorphisms
\[\gamma_2=\sigma_{n-k}\cdot\gamma:CH^i(Gr(\lfloor\frac{k-1}{2}\rfloor,\lfloor\frac{n-1}{2}\rfloor))\longrightarrow M_{k,n}^{4i+n-k}\]
\[\begin{array}{cccc}\lambda_2:&CH^i(Gr(\lfloor\frac{k-1}{2}\rfloor,\lfloor\frac{n-1}{2}\rfloor))\otimes_{\mathbb{Z}}\textbf{W}(F)&\longrightarrow&H^{4i+n-k}(Gr(k,n),\textbf{W}(O(1)))\\&\sigma&\longmapsto&\lambda(\sigma)\cdot e_{n-k}^{\perp}\end{array}.\]
\item If \(k(n-k)\) is even and \(n\equiv 0\ (mod\ 4)\), we have isomorphisms
\[\gamma_3=\gamma_1\oplus\gamma_2:CH^i(Gr(\lfloor\frac{k}{2}\rfloor,\lfloor\frac{n-1}{2}\rfloor))\oplus CH^{i+\frac{2k-n}{4}}(Gr(\lfloor\frac{k-1}{2}\rfloor,\lfloor\frac{n-1}{2}\rfloor))\longrightarrow M_{k,n}^{4i+k}\]
\[\lambda_3=\lambda_1\oplus\lambda_2:(CH^i(Gr(\lfloor\frac{k}{2}\rfloor,\lfloor\frac{n-1}{2}\rfloor))\oplus CH^{i+\frac{2k-n}{4}}(Gr(\lfloor\frac{k-1}{2}\rfloor,\lfloor\frac{n-1}{2}\rfloor)))\otimes_{\mathbb{Z}}\textbf{W}(F)\]\[\longrightarrow H^{4i+k}(Gr(k,n),\textbf{W}(O(1))).\]
\end{enumerate}
In all other cases, \(M_{k,n}^i\) and \(N_{k,n}^i\) vanishes.
\end{proposition}
\begin{proof}
Follows from \cite[Theorem 6.4]{W1}, \cite[Proposition 3.14]{W2}, Lemma \ref{ce1} and \ref{ce2}.
\end{proof}
\begin{theorem}\label{Grass}
\begin{enumerate}
\item If \(k\) is odd and \(n\) is even, \(Th(O_{Gr(k,n)}(1))\) splits in \(\widetilde{DM}(pt,\mathbb{Z})\) and we have
\begin{align*}
Th(O_{Gr(k,n)}(1))=	&Th(O_{Gr(k-2,n-2)}(1))(2n-2k)[4n-4k]\oplus\\
								&Th(O_{Gr(k,n-2)}(1))\oplus\mathbb{Z}(Gr(k-1,n-2))/{\eta}(n-k)[2n-2k].
\end{align*}
Moreover, we have
\[WW(Th(O_{Gr(k,n)}(1)))=\emptyset.\]
\item If \(n-k\) is even, \(Th(O_{Gr(k,n)}(1))\) splits \(\widetilde{DM}(pt,\mathbb{Z})\) and we have
\[Th(O_{Gr(k,n)}(1))=Th(O_{Gr(k,n-1)}(1))\oplus\mathbb{Z}(Gr(k-1,n-1))(n-k+1)[2n-2k+2].\]
Moreover, if \(k\) is even, we have
\[WW(Th(O_{Gr(k,n)}(1)))\subseteq 4\mathbb{Z}+n-k+1\cup(4\mathbb{Z}+k+1)\]
otherwise
\[WW(Th(O_{Gr(k,n)}(1)))\subseteq 4\mathbb{Z}+n-k+1.\]
\item If \(k\) is even and \(n\) is odd, \(Th(O_{Gr(k,n)}(1))\) splits \(\widetilde{DM}(pt,\mathbb{Z})\) and we have
\[Th(O_{Gr(k,n)}(1))=Th(O_{Gr(k-1,n-1)}(1))\oplus\mathbb{Z}(Gr(k,n-1))(k+1)[2k+2].\]
Moreover, we have
\[WW(Th(O_{Gr(k,n)}(1)))\subseteq 4\mathbb{Z}+k+1.\]
\item If \(n-k\) is odd, \(\mathbb{Z}(Gr(k,n))\) splits \(\widetilde{DM}(pt,\mathbb{Z})\) and we have
\[\mathbb{Z}(Gr(k,n))=\mathbb{Z}(Gr(k,n-1))\oplus Th(O_{Gr(k-1,n-1)}(1))(n-k-1)[2n-2k-2].\]
Moreover, if \(k\) is even, we have
\[WW(\mathbb{Z}(Gr(k,n)))\subseteq 4\mathbb{Z}\]
otherwise
\[WW(\mathbb{Z}(Gr(k,n)))\subseteq 4\mathbb{Z}\cup(4\mathbb{Z}+n-1).\]
\item If \(n-k\) is even, \(\mathbb{Z}(Gr(k,n))\) splits \(\widetilde{DM}(pt,\mathbb{Z})\) and we have
\[\mathbb{Z}(Gr(k,n))=\mathbb{Z}(Gr(k-1,n-1))\oplus Th(O_{Gr(k,n-1)}(1))(k-1)[2k-2].\]
Moreover, we have
\[WW(\mathbb{Z}(Gr(k,n)))\subseteq 4\mathbb{Z}.\]
\end{enumerate}
\end{theorem}
\begin{proof}
Define the partial order \(\leq\) on \((k,n)\) by \((k',n')\leq (k,n)\) iff \(k'\leq k\) and \(n'\leq n\). We suppose that all the statements are true for any \((k',n')<(k,n)\). Let us prove the statements for \((k,n)\). The cases for \(k=0,1\) have already known.
\begin{enumerate}
\item Set \(M=Th(O_{Gr(k,n)}(1)|_{Gr(k,n)\setminus Gr(k-2,n-2)})\) and \(p_1:Gr(1,2)\times Gr(k-1,n-2)\longrightarrow Gr(1,2)\). By Proposition \ref{geo2}, we have Gysin triangles (see Proposition \ref{Gysin})
\[Th(O_{Gr(k,n-2)}(1))\longrightarrow M\longrightarrow Th(p_1^*O(1))(n-k-1)[2n-2k-2]\longrightarrow\cdots[1];\]
\[M\longrightarrow Th(O_{Gr(k,n)}(1))\longrightarrow Th(O_{Gr(k-2,n-2)}(1))(2n-2k)[4n-4k]\longrightarrow\cdots[1].\]
Now we know that
\[Th(p_1^*O(1))=\mathbb{Z}(Gr(k-1,n-2))\otimes Th(O_{\mathbb{P}^1}(1))=\mathbb{Z}(Gr(k-1,n-2))/{\eta}(1)[2]\]
so \(WW(Th(p_1^*O(1)))=\emptyset\). Hence the first triangle splits. Moreover, \(WW(Th(O_{Gr(k-2,n-2)}(1)))=\emptyset\) by induction, hence the second triangle splits as well.
\item By the Gysin triangle related to \(Gr(k-1,n-1)\longrightarrow Gr(k,n)\) and Proposition \ref{splitting}.
\item By the Gysin triangle related to \(Gr(k,n-1)\longrightarrow Gr(k,n)\).
\item By the Gysin triangle related to \(Gr(k-1,n-1)\longrightarrow Gr(k,n)\) and Proposition \ref{splitting}.
\item By the Gysin triangle related to \(Gr(k,n-1)\longrightarrow Gr(k,n)\) and Proposition \ref{splitting}.
\end{enumerate}
\end{proof}
\begin{remark}\label{ex3}
Since all Thom spaces of Grassmannians split, we could apply Proposition \ref{sq} and \ref{basis} to give an additive basis of \(\widetilde{CH}^{{*}}(Gr(k,n),\mathscr{L})\). For example, suppose \((k,n)=(2,4)\) or \((k,n)=(3,6)\).
\[\begin{array}{|c|c|c|}\hline
															&\textbf{GW}(F)												&\mathbb{Z}\\\hline
\widetilde{CH}^{{*}}(Gr(2,4))			&1,\sigma_{2^2}												&2\sigma_{1},2\sigma_{2},\sigma_{2,1},\sigma_{2}+\sigma_{1^2}\\\hline
\widetilde{CH}^{{*}}(Gr(3,6))			&\begin{array}{c}1,\sigma_{2^2},\\\sigma_{3,1,1},\sigma_{3^3}\end{array}	&\begin{array}{c}2\sigma_{1},\sigma_{2}+\sigma_{1^2},2\sigma_{2},\sigma_{2,1},2\sigma_{3},\sigma_{3,1},2\sigma_{1^3},\\\sigma_{2,1,1},2\sigma_{3,2},\sigma_{3,3},2\sigma_{2,2,1},\sigma_{2^3},2\sigma_{3,2,1},2\sigma_{3,3,1},\\\sigma_{3,3,1}+\sigma_{3,2,2},\sigma_{3,3,2}\end{array}\\\hline
\widetilde{CH}^{{*}}(Gr(2,4),O(1))	&\sigma_{2},\sigma_{1^2}									&2,\sigma_{1},2\sigma_{2,1},\sigma_{2^2}\\\hline
\widetilde{CH}^{{*}}(Gr(3,6),O(1))	&																	&\begin{array}{c}2,\sigma_{1},2\sigma_{2},\sigma_{3},2\sigma_{1^2},\sigma_{1^3},2\sigma_{2,1},2\sigma_{3,1},\\2\sigma_{2,2},2\sigma_{3,2},\sigma_{3,1}+\sigma_{2,2}+\sigma_{2,1,1},\sigma_{3,2}+\sigma_{3,1,1},\\\sigma_{3,2}+\sigma_{2,2,1},\sigma_{3,2,1},2\sigma_{3,3},\sigma_{3,3,1},2\sigma_{2^3},\sigma_{3,2,2},\\2\sigma_{3,3,2},\sigma_{3^3}\end{array}\\\hline
\end{array}\]
\end{remark}
\begin{proposition}\label{decomp1}
\begin{enumerate}
\item For every \(s\in\widetilde{CH}^i(Gr(k,n))\), it could be written as \(s_1+s_2\) where \(s_1\in\textbf{GW}(F)[p_j,p_j^{\perp}]\otimes\wedge[\mathcal{R}]\) (the orientation class \(\mathcal{R}\) won't appear if \(k(n-k)\) is even) and \(s_2\in Ker_t(Sq^2\circ\pi)_i\).
\item For every \(s\in\widetilde{CH}^i(Gr(k,n),O(1))\), it could be written as \(s_1+s_2\) where \(s_1\in\textbf{GW}(F)[p_j,p_j^{\perp},e_k,e_{n-k}^{\perp}]\) and \(s_2\in Ker_t(Sq^2_{O(1)}\circ\pi)_i\).
\end{enumerate}
\end{proposition}
\begin{proof}
\begin{enumerate}
\item Suppose that \(s\) corresponds to \((u,v)\in H^i(Gr(k,n),\textbf{I}^i)\oplus CH^i(Gr(k,n))\). By \cite[Theorem 6.4]{W1}, \(u\) can be written as \(u_1+u_2\) where
\[u_1=\varphi(p_j,p_j^{\perp})\]
for some polynomial \(\varphi\) with coefficients in \(\textbf{W}(F)\) and \(u_2\in Im(\beta)\). The \(\varphi\) can be lifted to
\[\psi\in \textbf{GW}(F)[p_j,p_j^{\perp}]\otimes\wedge[\mathcal{R}]\subseteq\widetilde{CH}^{{*}}(Gr(k,n)).\]
Hence \(s-\psi\in 2N_{k,n}^i+Ker_t(Sq^2\circ\pi)_i\subseteq\widetilde{CH}^i(Gr(k,n))\) (see Proposition \ref{etaIm}). Then we could find
\[\psi'\in h\cdot\textbf{GW}(F)[p_j,p_j^{\perp}]\otimes\wedge[\mathcal{R}]\]
such that \(s-\psi-\psi'\in Ker_t(Sq^2\circ\pi)_i\subseteq\widetilde{CH}^i(Gr(k,n))\).
\item Same proof as (1).
\end{enumerate}
\end{proof}
\begin{proposition}\label{can}
For every \(j\in\mathbb{N}\), we have a map \(\varphi_j\) such that the following diagram commutes
\[
	\xymatrix
	{
		N_{k,n}^j\ar[r]^-{\varphi_j}\ar[d]	&H^j(Gr(k,n),\textbf{W})\ar[d]\\
		Ker(Sq^2)_j\ar[r]		&E^j(Gr(k,n))
	}.
\]
Thus we obtain an injection
\[u:N_{k,n}^j\otimes_{\mathbb{Z}}\textbf{GW}(F)\longrightarrow\widetilde{CH}^j(Gr(k,n)).\]
When \(k(n-k)\) is odd, the element \(u(\sigma_{n-k,1^{k-1}})\) is denoted by \(\mathcal{R}\), which is the orientation class.
\end{proposition}
\begin{proof}
Let us prove the first statement. The second follows from Theorem \ref{sq}.
\begin{enumerate}
\item Suppose \(j=4i\) and let \(\varphi_j=\lambda\circ\gamma^{-1}\). The \(Im(Sq^2)\) is an ideal of \(Ker(Sq^2)\), which is an subalgebra of \(Ch^*(Gr(k,n))\). Denote by \(\sim\) the equivalent relation on \(Ker(Sq^2)\) modulo by \(Im(Sq^2)\). So it suffices to prove that for every \(n-k\geq a_1\geq\cdots\geq a_k\geq 0\), we have
\[d:=\left|\begin{array}{ccc}\sigma_{2a_1}&\sigma_{2a_1+2}&\cdots\\\sigma_{2a_2-2}&\sigma_{2a_2}&\cdots\\\cdots&&\end{array}\right|^2\sim\gamma(\sigma_{a_1,\cdots,a_k})\]
by the Giambelli formula. Suppose that \(a_t=0\) if \(t>l\). Let us prove by induction on \(l\). If \(l=1\), we have
\[\sigma_{2j}^2=\gamma(\sigma_j)+Sq^2(\sigma_{4j-1}+\sigma_{4j-3,2}+\cdots+\sigma_{2j+1,2j-2}).\]
In general case, the Cramer's rule and induction gives
\[d\sim\sum_{j=1}^l\gamma(\sigma_{a_j-j+1})\gamma(\sigma_{a_1+1,\cdots,a_{j-1}+1,a_{j+1},\cdots,a_l}).\]
Again the Cramer's rule says
\[\sigma_{a_1,\cdots,a_l}=\sum_{j=1}^l\sigma_{a_j-j+1}\sigma_{a_1+1,\cdots,a_{j-1}+1,a_{j+1},\cdots,a_l}\]
so it remains to prove that
\[\gamma(\sigma_a\sigma_{a_1,\cdots,a_l})\sim\gamma(\sigma_a)\gamma(\sigma_{a_1,\cdots,a_l})\]
for arbitrary \(a,a_1,\cdots,a_l\). It suffices to prove the untruncated case at first, then apply the truncation map. But by Proposition \ref{imsq}, it suffices to compare their coefficients on even Young diagrams, which is obvious.
\item Suppose that \(j=4i+n-1\), that \(k(n-k)\) is odd and let \(\varphi_j=\lambda'\circ\gamma'^{-1}\). The statement follows from (1) since we have
\[Sq^2(x)\cdot\sigma_{n-k,1^{k-1}}=Sq^2(x\cdot\sigma_{n-k,1^{k-1}}).\]
\item Otherwise \(N_{k,n}^j\) vanishes.
\end{enumerate}
\end{proof}
\begin{proposition}\label{can1}
For every \(j\in\mathbb{N}\), we have a map \(\psi_j\) such that the following diagram commutes
\[
	\xymatrix
	{
		M_{k,n}^j\ar[r]^-{\psi_j}\ar[d]	&H^j(Gr(k,n),\textbf{W}(O(1)))\ar[d]\\
		Ker(Sq^2_{O(1)})_j\ar[r]		&E^j(Gr(k,n),O(1))
	}.
\]
Thus we obtain an injection
\[v:M_{k,n}^j\otimes_{\mathbb{Z}}\textbf{GW}(F)\longrightarrow \widetilde{CH}^j(Gr(k,n),O(1)).\]
\end{proposition}
\begin{proof}
\begin{enumerate}
\item If \(k\) is even, \(n\neq 0\ (mod\ 4)\) and \(j=4i+k\), we set \(\psi_j=\lambda_1\circ\gamma_1^{-1}\). Then the statement follows from Proposition \ref{can} because we have
\[Sq^2(x)\cdot\sigma_{1^k}=Sq^2_{O(1)}(x\cdot\sigma_{1^k}).\]
\item If \(n-k\) is even, \(n\neq 0\ (mod\ 4)\) and \(j=4i+n-k\), we set \(\psi_j=\lambda_2\circ\gamma_2^{-1}\). Then the statement follows from Proposition \ref{can} because we have
\[Sq^2(x)\cdot\sigma_{n-k}=Sq^2_{O(1)}(x\cdot\sigma_{n-k}).\]
\item If \(k(n-k)\) is even, \(n=0\ (mod\ 4)\) and \(j=4i+k\), we set \(\psi_j=\lambda_3\circ\gamma_3^{-1}\). Then the statement follows from (1) and (2).
\item Otherwise \(M_{k,n}^j\) vanishes.
\end{enumerate}
\end{proof}
The Proposition \ref{can} and \ref{can1} were partially announced in \cite[Proposition 3.6]{W2}. We gave the complete version for clarity.

\begin{definition}\label{special}
Suppose \(\Lambda=(a_1,\cdots,a_l)\) is an irredundant Young diagram, which is untwisted and \(\mathscr{L}=O_{Gr(k,n)}\) (resp. twisted and \(\mathscr{L}=O(1)\)).

If \(\Lambda\) is even, define
\[p(\Lambda)=u(\Lambda)\in\widetilde{CH}^{|\Lambda|}(Gr(k,n),\mathscr{L})\]
by Proposition \ref{can}.

Otherwise for every \(1<i_1<\cdots<i_l\leq |A(\Lambda)|\), define
\[c(\Lambda_{i_1,\cdots,i_l})\in\eta^{|\Lambda_{i_1,\cdots,i_l}|}_{MW}(Gr(k,n),\mathscr{L})\]
to be the unique morphism corresponding to
\[(\Lambda_{i_1,\cdots,i_l},Sq^2_{\mathscr{L}}(\Lambda_{i_1,\cdots,i_l}))\]
by Theorem \ref{eta}, where the \(Sq^2\) here is the integral lift (see remarks before Proposition \ref{basis}).
\end{definition}
\begin{theorem}\label{Grass1}
In the context above, we have:
\begin{enumerate}
\item The morphism
\[\mathbb{Z}(Gr(k,n))\xrightarrow{(p(\Lambda),c(\Lambda_{i_1,\cdots,i_l}))}\bigoplus_{\Lambda\textrm{ even}}\mathbb{Z}((|\Lambda|))\oplus\bigoplus_{\Lambda\textrm{ irred. not full}, i_1>1}\mathbb{Z}/{\eta}((|\Lambda_{i_1,\cdots,i_l}|))\]
is an isomorphism in \(\widetilde{DM}(pt,\mathbb{Z})\).
\item The morphism
\[\tiny Th(O_{Gr(k,n)}(1))\xrightarrow{(p(\Lambda),c(\Lambda_{i_1,\cdots,i_l}))}\bigoplus_{\Lambda\textrm{ even}}\mathbb{Z}((|\Lambda|+1))\oplus\bigoplus_{\Lambda\textrm{ irred. not full}, i_1>1}\mathbb{Z}/{\eta}((|\Lambda_{i_1,\cdots,i_l}|+1))\]
is an isomorphism in \(\widetilde{DM}(pt,\mathbb{Z})\).
\end{enumerate}
\end{theorem}
\begin{proof}
We adopt the notations in Proposition \ref{gen}.
\begin{enumerate}
\item Suppose \(\Lambda\) is even. We have
\[\gamma(u(\Lambda))=\Lambda\]
by Proposition \ref{can}. Suppose \((x,y)\in\eta_{MW}^n(Gr(k,n))\). We showed in Proposition \ref{gen} that
\[\begin{array}{cc}\gamma(h(x,y))=2x&\gamma(\partial(x,y))=y\end{array}.\]
So the \(\{c(\Lambda_{i_1,\cdots,i_l})|i_1>1\}\) is just the basis of \(Ker_t(Sq^2\circ\pi)\) given in Proposition \ref{basis}. By Theorem \ref{Grass}, we suppose
\[\mathbb{Z}(Gr(k,n))=\bigoplus_i\mathbb{Z}(i)[2i]^{\oplus w_i}\oplus\bigoplus_j\mathbb{Z}/\eta(j)[2j]^{\oplus t_j}.\]
By \cite[Proposition 3.14]{W2} and Proposition \ref{basis}, we have
\[w_i=|\{\Lambda|\Lambda\textrm{ is even},|\Lambda|=i\}|\]
\[w_i+t_i+t_{i-1}=rank(Ker(Sq^2\circ\pi)_i).\]
Then the statement follows.
\item Same proof as above.
\end{enumerate}
\end{proof}
As an application of Theorem \ref{Grass1}, we could discuss the invariance of Chow-Witt cycles of trivial Grassmannians under linear automorphisms.
\begin{proposition}\label{oriinvb}
Let \(X\in Sm/F\), \(\mathscr{E}\) be a vector bundle of rank \(n\) on \(X\) and \(\varphi\in Aut_{O_X}(\mathscr{E})\). The morphism \(Gr(k,\varphi)=Id\) in \(DM(pt,\mathbb{Z})\).
\end{proposition}
\begin{proof}
It suffices to prove that \(Gr(k,\varphi)\) keeps Chern classes of the tautological bundle \(\mathscr{U}\) invariant by the Grassmannian bundle theorem in Voevodsky's motives (see \cite[Proposition 2.4]{S} and the construction of \(c_{\mathscr{E}}\) in Proposition \ref{global}), which follows from construction.
\end{proof}
\begin{proposition}\label{oriinv1}
Suppose that \(X\in Sm/F\), \(\varphi\in GL_n(O_X^{\oplus n})\) and that \(p:X\times Gr(k,n)\longrightarrow Gr(k,n)\) is the projection. We have
\[Gr(k,\varphi)=Id\]
in \(\widetilde{DM}(pt,\mathbb{Z})\) if one of the following conditions hold:
\begin{enumerate}
\item The \(\varphi\) is an elementary matrix (see \cite[Definition 3]{An});
\item The \(k(n-k)\) is even and \(_2CH^*(X)=0\).
\item The \(k(n-k)\) is odd, \(_2CH^*(X)=0\) and \(\varphi=\left(\begin{array}{cc}g^2&\\&Id\end{array}\right)\) where \(g:X\longrightarrow\mathbb{G}_m\) is an invertible function.
\end{enumerate}
\end{proposition}
\begin{proof}
The Theorem \ref{Grass1} gives us the formulas
\[\mathbb{Z}(X\times Gr(k,n))\cong\bigoplus_{\Lambda\textrm{ even}}\mathbb{Z}(X)((|\Lambda|))\oplus\bigoplus_{\Lambda\textrm{ irred. not full}, i_1>1}\mathbb{Z}(X)/{\eta}((|\Lambda_{i_1,\cdots,i_l}|))\]
and
\[\widetilde{CH}^i(X\times Gr(k,n))\cong\bigoplus_{\Lambda\textrm{ even}}\widetilde{CH}^{i-|\Lambda|}(X)\oplus\bigoplus_{\Lambda\textrm{ irred. not full}, i_1>1}\eta_{MW}^{i-|\Lambda_{i_1,\cdots,i_l}|-1}(X).\]
\begin{enumerate}
\item Essentially the same as in \cite[Lemma 1]{An}.
\item Suppose \(s\in\widetilde{CH}^i(X\times Gr(k,n))\) is equal to the pullback of \(t\in\widetilde{CH}^i(Gr(k,n))\) along \(p\). Let us consider its behaviour under \(Gr(k,\varphi)^*\). By Proposition \ref{decomp1} and Proposition \ref{etaIm}, \(t=t_1+t_2\) where \(t_1\in\textbf{GW}(F)[p_j,p_j^{\perp}]\) and \(t_2\in Ker_t(Sq^2\circ\pi)_i\subseteq Im(\partial)\). But \(\eta_{MW}^{i-1}(X\times Gr(k,\varphi))=Id\) by Proposition \ref{oriinvb} and \cite[Theorem 4.13]{Y1}. Hence \(Gr(k,\varphi)^*(p^*(t_2))=p^*(t_2)\) by the commutative diagram
\[
	\xymatrix
	{
		\eta_{MW}^{i-1}(X\times Gr(k,n))\ar[r]^{\partial}\ar[d]_{Gr(k,\varphi)^*}	&\widetilde{CH}^{i}(X\times Gr(k,n))\ar[d]_{Gr(k,\varphi)^*}\\
		\eta_{MW}^{i-1}(X\times Gr(k,n))\ar[r]^{\partial}										&\widetilde{CH}^{i}(X\times Gr(k,n))
	},
\]
while the corresponding result for \(t_1\) is clear. So \(Gr(k,\varphi)^*(s)=s\). Then the statement follows from \cite[Theorem 4.13]{Y1}.
\item The \(Gr(k,\varphi)\) is equal to the composite
\[X\times Gr(k,n)\xrightarrow{\Gamma_{g}\times Id_{Gr(k,n)}}X\times\mathbb{G}_m\times Gr(k,n)\xrightarrow{Id_X\times u}X\times Gr(k,n)\]
where \(\Gamma_{g}\) is the graph of \(g\) and \(u\) is the morphism
\[\begin{array}{ccccc}\mathbb{G}_m&\times&Gr(k,n)&\longrightarrow&Gr(k,n)\\t&,&(x_1,\cdots,x_n)\in V&\longmapsto&(t^2x_1,x_2,\cdots,x_{n})\in t\cdot V\end{array}.\]
But the squaring morphism \(-^2:\mathbb{G}_m\longrightarrow\mathbb{G}_m\) is the multiplication by \(h\) (see the proof of \cite[Proposition 5.16]{Y1}) and \([\mathbb{Z}(Gr(k,n))(1)[1],\mathbb{Z}(Gr(k,n))]_{MW}\) is a direct sum of \(\textbf{W}(F)\) by \cite[Proposition 5.4]{Y1}. Hence \(u\) coincides with the projection map in \(\widetilde{DM}(pt,\mathbb{Z})\). Hence \(Gr(k,\varphi)=Id\) in \(\widetilde{DM}(pt,\mathbb{Z})\) by \textit{loc. cit.}.
\end{enumerate}
\end{proof}
\section{Theorems of Fiber Bundles}
\begin{definition}
For every Young diagram \(\Lambda=(a_1,\cdots,a_l)\), we define its transpose \(\Lambda^T=(b_1,\cdots,b_s)\) by
\[b_i=\#\textrm{ of boxes in i-th column of \(\Lambda\)}\]
where \(s=\#\textrm{ of columns in \(\Lambda\)}\). If \(\Lambda\) is \((k,n)\)-truncated, \(\Lambda^T\) is understood as \((n-k,n)\)-truncated.
\end{definition}
Suppose \(x^{(n)}\) is a set of \(n\) indeterminants. Denote by \(e_i(x^{(n)})\) (resp. \(h_i(x^{(n)})\)) the elementary (resp. complete) symmetric polynomial of degree \(i\). Define the Schur function
\[s_{\Lambda}(x^{(n)})=det(h_{\Lambda_i-i+j}(x^{(n)}))=det(e_{\Lambda^T_i-i+j}(x^{(n)})).\]

Let us state a result between Schur functions and the Steenrod square, which has been essentially proved in \cite[Theorem 4.2]{W2}.
\begin{proposition}\label{Schur}
Let \(R\) be a commutative ring with identity and \(x_1,\cdots,x_n\) are indeterminants. Define
\[x_{a_1,\cdots,a_l}=\left|\begin{array}{ccc}x_{a_1}&x_{a_1+1}&\cdots\\x_{a_2-1}&x_{a_2}&\cdots\\\cdots&&\end{array}\right|_{l\times l}\]
where \(a_i\in\mathbb{N}\) for every \(1\leq i\leq l\), \(x_0:=1\) and \(x_i:=0\) if \(i\notin[0,n]\).

If \(R=\mathbb{Z}/2\), define \(Sq^2:R[x_1,\cdots,x_n]\longrightarrow R[x_1,\cdots,x_n]\) to be the \(R\)-derivation satisfying
\[Sq^2(x_i)=(i+1)x_{i+1}+x_1x_i.\]

We have
\begin{enumerate}
\item The \(\{x_{a_1,\cdots,a_l}\}_{n\geq a_1\geq\cdots\geq a_l\geq 0}\) form an \(R\)-basis of \(R[x_1,\cdots,x_n]\).
\item Suppose \(a_1\geq\cdots\geq a_l\) and \(1\leq t\leq n\), we have the Pieri formula
\[x_tx_{a_1,\cdots,a_l}=\sum_{b=(b_1,\cdots,b_{l+1})}x_b\]
where \(b_1\geq\cdots\geq b_{l+1}\), \(\sum b_i=t+\sum a_i\) and \(a_i\leq b_i\leq a_{i-1}\) for every \(i\).
\item\[Sq^2(x_{a_1,\cdots,a_l})=\sum_{i=1}^{l}(a_i-i+1)x_{\cdots,a_{i-1},a_i+1,a_{i+1},\cdots}+l\cdot x_{a_1,\cdots,a_l,1}.\]
\item We have
\[Ker(Sq^2)=Im(Sq^2)\oplus N\]
where \(N\) is the vector space generated by \(x_{a_1,\cdots,a_l}\) where \((a_1,\cdots,a_l)\) is completely even (see \cite[Definition 3.8]{W2}).
\item Suppose \(a_1\geq\cdots\geq a_l\). We have
\[\left|\begin{array}{ccc}x_{2a_1}&x_{2a_1+2}&\cdots\\x_{2a_2-2}&x_{2a_2}&\cdots\\\cdots&&\end{array}\right|_{l\times l}^2=x_{2a_1,\cdots,2a_l}+Im(Sq^2).\]
\end{enumerate}
\end{proposition}
\begin{proof}
Set
\[\begin{array}{cccc}s_i:&\mathbb{N}^{\times l}&\longrightarrow&\mathbb{N}^{\times l}\\&(a_1,\cdots,a_l)&\longmapsto&(\cdots,a_{i-1},a_i+1,a_{i+1},\cdots)\end{array}.\]
Recall we have the Cramer's rule
\[x_{a_1,\cdots,a_l}=\sum_{i=1}^{l}(-1)^{i-1}x_{a_i-i+1}x_{(s_1\cdots s_{i-1})(\cdots\widehat{a_i}\cdots)}.\]
It implies that
\[\sum_{i=1}^{l}(-1)^{i-1}x_{a_i-i+2}x_{(s_1\cdots s_{i-1})(\cdots\widehat{a_i}\cdots)}=0.\]
The \(e_i(y_1,\cdots,y_n), i=1,\cdots,n\) are algebraically independent so we may as well set \(x_i=e_i\). Then the \(R[x_1,\cdots,x_n]\) is identified with symmetric polynomials in \(R[y_1,\cdots,y_n]\) and \(\{x_{a_1,\cdots,a_l}\}\) are Schur polynomials.
\begin{enumerate}
\item This is because the Schur polynomials form a basis of the symmetric polynomials.
\item The statement follows from the Pieri formula of Schur polynomials.
\item It's easy to verify the statement when \(l=1\). Let us prove by induction on \(l\) and suppose \(l\geq 2\). We have
{\footnotesize\begin{align*}
	&Sq^2(x_{a_1,\cdots,a_l})\\
=	&\sum_{i=1}^{l}Sq^2(x_{a_i-i+1}x_{(s_1\cdots s_{i-1})(\cdots\widehat{a_i}\cdots)})\\
=	&\sum_{i=1}^{l}((a_i-i+2)x_{a_i-i+2}+x_1x_{a_i-i+1})x_{(s_1\cdots s_{i-1})(\cdots\widehat{a_i}\cdots)}+\\
	&\sum_{i=1}^{l}(l-1)x_{a_i-i+1}x_{(s_1\cdots s_{i-1})(\cdots\widehat{a_i}\cdots,1)}+\\
	&\sum_{i=1}^{l}x_{a_i-i+1}(\sum_{j=1}^{i-1}(a_j-j+2)x_{(s_js_1\cdots s_{i-1})(\cdots\widehat{a_i}\cdots)}+\sum_{j=i+1}^{l}(a_j-j+2)x_{(s_{j-1}s_1\cdots s_{i-1})(\cdots\widehat{a_i}\cdots)})
\end{align*}
\begin{align*}
=	&\sum_{i=1}^{l}((a_i-i+1)x_{a_i-i+2}+x_1x_{a_i-i+1})x_{(s_1\cdots s_{i-1})(\cdots\widehat{a_i}\cdots)}+(l-1)x_{a_1,\cdots,a_l,1}\\
	&+\sum_{i=1}^{l}x_{a_i-i+1}(\sum_{j=1}^{i-1}(a_j-j+2)x_{(s_js_1\cdots s_{i-1})(\cdots\widehat{a_i}\cdots)}+\sum_{j=i+1}^{l}(a_j-j+2)x_{(s_{j-1}s_1\cdots s_{i-1})(\cdots\widehat{a_i}\cdots)})\\
=	&\sum_{i=1}^{l}(a_i-i+1)x_{a_i-i+2}x_{(s_1\cdots s_{i-1})(\cdots\widehat{a_i}\cdots)}+l\cdot x_{a_1,\cdots,a_l,1}\\
	&+\sum_{i=1}^{l}x_{a_i-i+1}(\sum_{j=1}^{i-1}(a_j-j+1)x_{(s_js_1\cdots s_{i-1})(\cdots\widehat{a_i}\cdots)}+\sum_{j=i+1}^{l}(a_j-j+1)x_{(s_{j-1}s_1\cdots s_{i-1})(\cdots\widehat{a_i}\cdots)})\\
=	&\sum_{i=1}^{l}(a_i-i+1)x_{\cdots,a_{i-1},a_i+1,a_{i+1},\cdots}+l\cdot x_{a_1,\cdots,a_l,1}.
\end{align*}}
\item It is easy to show that \(Sq^2(Sq^2(x_i))=0\) for every \(i\). Then
\[Sq^2(Sq^2(\prod_jx_{i_j}))=\sum_kSq^2(Sq^2(x_{i_k}))\prod_{j\neq k}x_{i_j}=0.\]
Hence \(Sq^2\circ Sq^2=0\). Then the proof is essentially the same as Lemma \ref{decomp} and Proposition \ref{imsq}.
\item The proof is essentially the same as in Proposition \ref{can}.
\end{enumerate}
\end{proof}
Given a vector bundle \(\mathscr{E}\) on \(X\), we have a tautological exact sequence on \(Gr(k,\mathscr{E})\) which relates \(\mathscr{E}\), \(\mathscr{U}_{\mathscr{E}}\) and \(\mathscr{U}_{\mathscr{E}}^{\perp}\). It is interesting to express Schur functions of Chern roots of one of them in terms of the other two.
\begin{proposition}\label{inversion}
Let \(f(t)=1-a_1t+\cdots+(-1)^ka_kt^k\), \(g(t)=1-b_1t+\cdots+(-1)^{n-k}b_{n-k}t^{n-k}\) and \(h(t)=1-c_1t+\cdots+(-1)^nc_nt^n\), where \(t\) is an indeterminant and all coefficients live in some commutative ring with identity. We could define \(a_{\Lambda}, b_{\Lambda}, c_{\Lambda}\) for every Young diagram \(\Lambda\) as in Proposition \ref{Schur}. Suppose that
\[f(t)g(t)=h(t).\]
\begin{enumerate}
\item We have
\[\begin{array}{cc}\sum_{i=0}^m(-1)^ib_{1^i}c_{m-i}=a_m&\sum_{i=0}^m(-1)^ia_{1^i}c_{m-i}=b_m\\\sum_{i=0}^m(-1)^ib_ic_{1^{m-i}}=a_{1^m}&\sum_{i=0}^m(-1)^ia_ic_{1^{m-i}}=b_{1^m}\end{array}\]
for every \(m\in\mathbb{N}\) (\(a_0:=1\) and \(a_i:=0\) if \(i\neq [0,k]\), similar convention for \(b\) and \(c\)).
\item For every untruncated Young diagram \(\Lambda\), we have
\[c_{\Lambda}=\sum_{S_1,S_2\leq\Lambda}c_{S_1,S_2}^{\Lambda}a_{S_1}b_{S_2}\]
\[b_{\Lambda}=\sum_{S\leq\Lambda,S_k=0}b_S\cdot\mathbb{Z}[c_i]+a_k\cdot\mathbb{Z}[a_i,c_j]=\sum_{S\leq\Lambda,S_{k+1}=0}b_S\cdot\mathbb{Z}[c_i]\]
\[a_{\Lambda}=\sum_{S\leq\Lambda,S_{n-k}=0}a_S\cdot\mathbb{Z}[c_i]+b_{n-k}\cdot\mathbb{Z}[b_i,c_j]=\sum_{S\leq\Lambda,S_{n-k+1}=0}a_S\cdot\mathbb{Z}[c_i]\]
where \(c_{S,T}^{\Lambda}\) is the Littlewood-Richardson coefficient.
\end{enumerate}
\end{proposition}
\begin{proof}
We may assume \(a_i=e_i(x^{(k)})\), \(b_i=e_i(y^{(n-k)})\) and \(c_i=e_i(z^{(n)}),z^{(n)}=(x^{(k)},y^{(n-k)})\).
\begin{enumerate}
\item The statements come from the fact that
\[\begin{array}{cc}\frac{1}{f(t)}=\sum_{i=0}^{\infty}a_{1^i}t^i&\frac{1}{g(t)}=\sum_{i=0}^{\infty}b_{1^i}t^i\end{array}.\]
\item The \(c_{\Lambda}\) is the coproduct of Schur function \(s_{\Lambda^T}(z^{(n)})\). So the first equation follows from \cite[4.18]{Z}.

The \(b_{\Lambda}\) is the super Schur polynomial \(s_{\Lambda^T}(z^{(n)}/(-x^{(k)}))\) (see \cite[Definition 1.3]{PT}). By \cite[Theorem 3.1]{PT}, we have
\[s_{\Lambda^T}(z^{(n)}/(-x^{(k)}))=\sum_{S_1,S_2\leq\Lambda}\alpha_{S_1S_2}s_{S_1}(-x^{(k)})s_{S_2^T}(z^{(n)})\]
where \(\alpha_{S_1S_2}\in\mathbb{Z}\). Hence we have
\[b_{\Lambda}=\sum_{S_1,S_2\leq\Lambda}\alpha_{S_1S_2}a_{S_1^T}c_{S_2}=\sum_{S_1,S_2\leq\Lambda,(S_1)_{k+1}=0}\alpha_{S_1S_2}a_{S_1^T}c_{S_2}\]
and
\[b_{\Lambda}=\sum_{S_1,S_2\leq\Lambda}\alpha_{S_1S_2}a_{S_1^T}c_{S_2}=\sum_{S_1,S_2\leq\Lambda,(S_1)_k=0}\alpha_{S_1S_2}a_{S_1^T}c_{S_2}+\sum_{S_1,S_2\leq\Lambda,(S_1)_k>0}\alpha_{S_1S_2}a_{S_1^T}c_{S_2}.\]
But \(a_S\) is a multiple of \(a_k\) if \(S_1\geq k\) (which is equivalent to \((S_1^T)_k>0\)) and
\[a_{\Lambda}=\sum_{S_1,S_2\leq\Lambda}\beta_{S_1S_2}b_{S_1^T}c_{S_2}\]
where \(\beta_{S_1S_2}\in\mathbb{Z}\) by the same arguments before. So we have proved the second equation. The third one follows symmetrically.
\end{enumerate}
\end{proof}
\begin{proposition}\label{global}
Suppose that \(X\in Sm/F\), \(_2CH^*(X)=0\) and that \({\mathscr{E}}\) is a vector bundle of rank \(n\) on \(X\) trivialized by an acyclic covering \(\{U_a\}\) (see \cite[Definition 5.12]{Y1}). Denote by \(p:Gr(k,\mathscr{E})\longrightarrow X\) the structure map and \(w:\widetilde{CH}^{{*}}\longrightarrow H^{{*}}(-,\textbf{W}(-))\) the canonical map.
\begin{enumerate}
\item For every element \(s\in\eta_{MW}^i(Gr(k,n))\) there is a canonical element \(\varphi_{\mathscr{E}}(s)\in\eta_{MW}^i(Gr(k,{\mathscr{E}}))\) such that \(\varphi_{\mathscr{E}}(s)|_{U_a}\) is the pullback of \(s\) along \(Gr(k,n)\times U_a\longrightarrow Gr(k,n)\).
\item Suppose both \(k\) and \(n\) are even, \(\mathscr{L}\in Pic(X)\) and \(\varphi_{\mathscr{E}}(s)\in\eta_{MW}^i(Gr(k,n),O(1))\). There is a canonical element \(\varphi_{\mathscr{E}}(s)\in\eta_{MW}^i(Gr(k,{\mathscr{E}}),p^*\mathscr{L}\otimes O(1))\) such that \(\varphi_{\mathscr{E}}(s)|_{U_a}\) is the pullback of \(s\) along \(Gr(k,n)\times U_a\longrightarrow Gr(k,n)\).
\item Suppose \(k(n-k)\) is even. For every element \(s\in\widetilde{CH}^i(Gr(k,n))\) there is an element \(\psi_{\mathscr{E}}(s)\in\widetilde{CH}^i(Gr(k,{\mathscr{E}}))\) such that \(\psi_{\mathscr{E}}(s)|_{U_a}\) is the pullback of \(s\) along \(Gr(k,n)\times U_a\longrightarrow Gr(k,n)\).
\item Suppose both \(k\) and \(n\) are even and \(s\in\widetilde{CH}^i(Gr(k,n),O(1))\).
\begin{enumerate}
\item If \(w(s)\in e_k\cdot H^*(Gr(k,n),\textbf{W})\), there is an element \(\psi_{\mathscr{E}}(s)\in\widetilde{CH}^i(Gr(k,{\mathscr{E}}),O(1))\) such that \(\psi_{\mathscr{E}}(s)|_{U_a}\) is the pullback of \(s\) along \(Gr(k,n)\times U_a\longrightarrow Gr(k,n)\).
\item If \(w(s)\in e_{n-k}^{\perp}\cdot H^*(Gr(k,n),\textbf{W})\), there is an element \(\psi_{\mathscr{E}}(s)\in\widetilde{CH}^i(Gr(k,{\mathscr{E}}),p^*det(\mathscr{E})\otimes O(1))\) such that \(\psi_{\mathscr{E}}(s)|_{U_a}\) is the pullback of \(s\) along \(Gr(k,n)\times U_a\longrightarrow Gr(k,n)\).
\end{enumerate}
\end{enumerate}
\end{proposition}
\begin{proof}
\begin{enumerate}
\item By \cite[Theorem 4.13]{Y1} and Proposition \ref{basis}, \(\eta_{MW}^{i}(Gr(k,n))\subseteq CH^i(Gr(k,n))\oplus CH^{i+1}(Gr(k,n))\) is freely generated by elements like
\[\begin{array}{ccc}u_{a_1,\cdots,a_l}=(\sigma_{a_1,\cdots,a_l},Sq^2(\sigma_{a_1,\cdots,a_l}))&and&v_{a_1,\cdots,a_l}=(0,2\sigma_{a_1,\cdots,a_l})\end{array}.\]
We define a morphism
\[\begin{array}{cccc}c_{\mathscr{E}}:&CH^*(Gr(k,n))&\longrightarrow&CH^*(Gr(k,{\mathscr{E}}))\\&\sigma_{a_1,\cdots,a_l}&\longmapsto&\left|\begin{array}{ccc}c_{a_1}(\mathscr{U}_{\mathscr{E}}^{\perp})&c_{a_1+1}(\mathscr{U}_{\mathscr{E}}^{\perp})&\cdots\\c_{a_2-1}(\mathscr{U}_{\mathscr{E}}^{\perp})&c_{a_2}(\mathscr{U}_{\mathscr{E}}^{\perp})&\cdots\\\cdots&&\end{array}\right|_{l\times l}\end{array}\]
Define
\[\varphi_{\mathscr{E}}(v_{a_1,\cdots,a_l})=(0,2c_{\mathscr{E}}(\sigma_{a_1,\cdots,a_l})).\]
We have
\[Sq^2(c_i(\mathscr{U}_{\mathscr{E}}^{\perp}))=(i+1)c_{i+1}(\mathscr{U}_{\mathscr{E}}^{\perp})+c_1(\mathscr{U}_{\mathscr{E}}^{\perp})c_i(\mathscr{U}_{\mathscr{E}}^{\perp})\]
in \(Ch^{{*}}(Gr(k,{\mathscr{E}}))\) by \cite[Proposition 5.5.2]{R}. If \((a_1,\cdots,a_l)=T\cdot\sigma_{1^l}\) where \(T\) is completely even, \(l\) is even and \(T_{l+1}=0\), we define
\[\varphi_{\mathscr{E}}(u_{a_1,\cdots,a_l})=(c_{\mathscr{E}}(T)c_l(\mathscr{U}_{\mathscr{E}}),c_{\mathscr{E}}(T)(c_{l+1}(\mathscr{U}_{\mathscr{E}})+c_1(\mathscr{U}_{\mathscr{E}})c_l(\mathscr{U}_{\mathscr{E}})));\]
Otherwise define
\[\varphi_{\mathscr{E}}(u_{a_1,\cdots,a_l})=(c_{\mathscr{E}}(\sigma_{a_1,\cdots,a_l}),c_{\mathscr{E}}(\sum_{i=1}^{l}\mu(a_i-i+1)\sigma_{\cdots,a_{i-1},a_i+1,a_{i+1},\cdots}+\mu(l)\sigma_{a_1,\cdots,a_l,1}))\]
where
\[\begin{array}{cccc}\mu:&\mathbb{Z}&\longrightarrow&\{0,1\}\\&\{\textrm{odd numbers}\}&\longmapsto&1\\&\{\textrm{even numbers}\}&\longmapsto&0\end{array}.\]
We have \(\varphi_{\mathscr{E}}(u_{a_1,\cdots,a_l})\in\eta_{MW}^{\sum a_i}(Gr(k,{\mathscr{E}}))\) by Proposition \ref{Schur}. Then we expand the definition linearly to obtain a map
\[\varphi_{\mathscr{E}}:\eta_{MW}^i(Gr(k,n))\longrightarrow\eta_{MW}^i(Gr(k,{\mathscr{E}})).\]
\item The proof is completely the same as (1). By Proposition \ref{eta} and \ref{basis}, \(\eta_{MW}^{i}(Gr(k,n),O(1))\subseteq CH^{i}(Gr(k,n))\oplus CH^{i+1}(Gr(k,n))\) is freely generated by elements like
\[\begin{array}{ccc}u_{a_1,\cdots,a_l}=(\sigma_{a_1,\cdots,a_l},Sq^2_{O(1)}(\sigma_{a_1,\cdots,a_l}))&and&v_{a_1,\cdots,a_l}=(0,2\sigma_{a_1,\cdots,a_l})\end{array}.\]
We define
\[\varphi_{\mathscr{E}}(v_{a_1,\cdots,a_l})=(0,2c_{\mathscr{E}}(\sigma_{a_1,\cdots,a_l})).\]
If \((a_1,\cdots,a_l)=T\cdot\sigma_{1^l}\) where \(T\) is completely even, \(l\) is even and \(T_{l+1}=0\), we define 
\[\varphi_{\mathscr{E}}(u_{a_1,\cdots,a_l})=(c_{\mathscr{E}}(T)\cdot c_l(\mathscr{U}_{\mathscr{E}}),c_{\mathscr{E}}(T)(c_{l+1}(\mathscr{U}_{\mathscr{E}})+c_1(p^*\mathscr{L})c_l(\mathscr{U}_{\mathscr{E}}));\]
Otherwise we define
\[\varphi_{\mathscr{E}}(u_{a_1,\cdots,a_l})=(c_{\mathscr{E}}(\sigma_{a_1,\cdots,a_l}),c_{\mathscr{E}}(\sum_{i=1}^{l}\mu(a_i-i+1)\sigma_{\cdots,a_{i-1},a_i+1,a_{i+1},\cdots}+\mu(l)\sigma_{a_1,\cdots,a_l,1})+\delta(\mathscr{E}))\]
where \(\delta(\mathscr{E})=c_1(p^*\mathscr{L}\otimes O(1))c_{\mathscr{E}}(\sigma_{a_1,\cdots,a_l})\).
\item By Proposition \ref{decomp1}, there is a \(u_1'=\psi(p_j,p_j^{\perp})\) such that \(\psi\) is a polynomial with coefficients in \(\textbf{GW}(F)\) and \(s-u_1'\in Ker_t(Sq^2\circ\pi)_i\). Suppose
\[t:Ker_t(Sq^2\circ\pi)_i\longrightarrow\eta_{MW}^{i-1}(Gr(k,n))\]
is a section of \(\partial\) (see Proposition \ref{etaIm}). Then \((\pi\circ\varphi_{\mathscr{E}}\circ t)(s-u_1')\) is a global version of \(s-u_1'\). The global version of \(u_1'\) exists when \(k(n-k)\) is even, which is
\[x=\psi(p_j(\mathscr{U}_{\mathscr{E}}),p_j(\mathscr{U}_{\mathscr{E}}^{\perp})).\]
Then we set \(\psi_{\mathscr{E}}(s)=x+(\pi\circ\varphi_{\mathscr{E}}\circ t)(s-u_1')\).
\item
\begin{enumerate}
\item By Proposition \ref{decomp1}, there is a \(u_1'=\psi(p_j,p_j^{\perp},e_k)\) such that \(\psi\) is a polynomial with coefficients in \(\textbf{GW}(F)\) and \(s-u_1'\in Ker_t(Sq^2_{O(1)}\circ\pi)_i\). Then proceed the same proof as (3).
\item By Proposition \ref{decomp1}, there is a \(u_1'=\psi(p_j,p_j^{\perp},e_{n-k}^{\perp})\) such that \(\psi\) is a polynomial with coefficients in \(\textbf{GW}(F)\) and \(s-u_1'\in Ker_t(Sq^2_{O(1)}\circ\pi)_i\). Then proceed the same proof as (3).
\end{enumerate}
\end{enumerate}
\end{proof}

Now in order to perform more precise computation, we have to globalize the elements \(p(\Lambda)\) defined in Definition \ref{special} in a canonical way, in addition to the \(\varphi_{\mathscr{E}}\) defined above. Note that the \(\psi_{\mathscr{E}}\) above is not canonical.
\begin{proposition}\label{precise}
Suppose that \(X\in Sm/F\) and that \(\mathscr{E}\) is a vector bundle of rank \(n\) over \(X\). Recall the morphisms \(\gamma:\widetilde{CH}^{{*}}\longrightarrow CH\) and \(w:\widetilde{CH}^{{*}}\longrightarrow H^{{*}}(-,\textbf{W}(-))\) defined in Proposition \ref{gen}.
\begin{enumerate}
\item If \(\mathbb{Z}(Gr(k,\mathscr{E}))\) splits as an MW-motive, for every completely even Young diagram
\[\Lambda=(\cdots,2a_i,2a_i,\cdots),\]
there is a unique element \(p_{\mathscr{E}}(\Lambda)\in\widetilde{CH}^{|\Lambda|}(Gr(k,\mathscr{E}))\) such that
\[\gamma(p_{\mathscr{E}}(\Lambda))=c_{\mathscr{E}}(\Lambda)\]
\[w(p_{\mathscr{E}}(\Lambda))=\left|\begin{array}{ccc}p_{2a_1}(\mathscr{U}_{\mathscr{E}}^{\perp})&p_{2a_1+2}(\mathscr{U}_{\mathscr{E}}^{\perp})&\cdots\\p_{2a_2-2}(\mathscr{U}_{\mathscr{E}}^{\perp})&p_{2a_2}(\mathscr{U}_{\mathscr{E}}^{\perp})&\cdots\\\cdots&&\end{array}\right|.\]
\item If \(Th(O_{Gr(k,\mathscr{E})}(1))\) splits as an MW-motive, for every even Young diagram
\[\Lambda=(\cdots,2a_i,2a_i,\cdots)\cdot\sigma_{1^k}\]
(when it makes sense), there is a unique element \(p_{\mathscr{E}}(\Lambda)\in\widetilde{CH}^{|\Lambda|}(Gr(k,\mathscr{E}),O(1))\) such that
\[\gamma(p_{\mathscr{E}}(\Lambda))=c_{\mathscr{E}}(\sigma_{\cdots,2a_i,2a_i,\cdots})\cdot c_k(\mathscr{U}_{\mathscr{E}})\]
\[w(p_{\mathscr{E}}(\Lambda))=\left|\begin{array}{ccc}p_{2a_1}(\mathscr{U}_{\mathscr{E}}^{\perp})&p_{2a_1+2}(\mathscr{U}_{\mathscr{E}}^{\perp})&\cdots\\p_{2a_2-2}(\mathscr{U}_{\mathscr{E}}^{\perp})&p_{2a_2}(\mathscr{U}_{\mathscr{E}}^{\perp})&\cdots\\\cdots&&\end{array}\right|\cdot e(\mathscr{U}_{\mathscr{E}}).\]
\item If \(Th(O_{Gr(k,\mathscr{E})}(1)\otimes p^*det(\mathscr{E}))\) splits as an MW-motive, for every even Young diagram
\[\Lambda=(\cdots,2a_i,2a_i,\cdots)\cdot\sigma_{n-k}\]
(when it makes sense), there is a unique element \(p_{\mathscr{E}}(\Lambda)\in\widetilde{CH}^{|\Lambda|}(Gr(k,\mathscr{E}),p^*det(\mathscr{E})\otimes O(1))\) such that
\[\gamma(p_{\mathscr{E}}(\Lambda))=c_{\mathscr{E}}(\sigma_{\cdots,2a_i,2a_i,\cdots})\cdot c_{n-k}(\mathscr{U}_{\mathscr{E}}^{\perp})\]
\[w(p_{\mathscr{E}}(\Lambda))=\left|\begin{array}{ccc}p_{2a_1}(\mathscr{U}_{\mathscr{E}}^{\perp})&p_{2a_1+2}(\mathscr{U}_{\mathscr{E}}^{\perp})&\cdots\\p_{2a_2-2}(\mathscr{U}_{\mathscr{E}}^{\perp})&p_{2a_2}(\mathscr{U}_{\mathscr{E}}^{\perp})&\cdots\\\cdots&&\end{array}\right|\cdot e(\mathscr{U}_{\mathscr{E}}^{\perp})\]
where \(p:Gr(k,\mathscr{E})\longrightarrow X\) is the structure map.
\end{enumerate}
\end{proposition}
\begin{proof}
\begin{enumerate}
\item Suppose \(|\Lambda|=n\). We have a commutative diagram with squares being Cartesian
\[
	\xymatrix
	{
		\widetilde{CH}^n(Gr(k,\mathscr{E}))\ar[r]\ar[d]	&Ker(Sq^2\circ\pi)_n\ar[d]\\
		H^n(Gr(k,\mathscr{E}),\textbf{W})\ar[r]				&E^n(Gr(k,\mathscr{E}))
	}
\]
by Theorem \ref{sq}. Set \(c_i=c_i(\mathscr{U}_{\mathscr{E}}^{\perp})\) and \(p_i=p_i(\mathscr{U}_{\mathscr{E}}^{\perp})\). Then \(p_i=c_i^2\) in \(Ch^*(Gr(k,\mathscr{E}))\). Hence we are to prove that
\[\left|\begin{array}{ccc}c_{2a_1}&c_{2a_1+2}&\cdots\\c_{2a_2-2}&c_{2a_2}&\cdots\\\cdots&&\end{array}\right|^2=\left|\begin{array}{ccc}c_{2a_1}&c_{2a_1+1}&\cdots\\c_{2a_2-1}&c_{2a_2}&\cdots\\\cdots&&\end{array}\right|+Im(Sq^2),\]
which follows by Proposition \ref{Schur}.
\item Applying Theorem \ref{sq}, the statement follows from same method as in Proposition \ref{can1}.
\item Applying Theorem \ref{sq}, the statement follows from same method as in Proposition \ref{can1}.
\end{enumerate}
\end{proof}
Suppose \(f,g\in Hom_{\mathscr{C}}(A,B)\) are morphisms in some triangulated category \(\mathscr{C}\). We say that \(f\) could be simplified to \(g\) if \(g=h\circ f\) where \(h\in Aut(B)\). If \(B=C\oplus D\), any morphism \(C\longrightarrow D\) gives an element of \(Aut(C\oplus D)\) by an elementary matrix. Simplification of morphisms does not change their mapping cones.

In the sequel, let \(R\) be a commutative ring with identity.
\begin{theorem}\label{Grassbdl}
Let \(S\in Sm/F\), \(X\in Sm/S\) be quasi-projective, \(\mathscr{L}\in Pic(X)\) and \({\mathscr{E}}\) be a vector bundle of rank \(n\) over \(X\). Denote by \(p:Gr(k,\mathscr{E})\longrightarrow X\) the structure map.
\begin{enumerate}
\item We have
\[R(Gr(k,\mathscr{E}))/\eta\cong\bigoplus_{\Lambda\textrm{ \((k,n)\)-truncated}}R(X)/\eta((|\Lambda|))\]
in \(\widetilde{DM}(S,R)\).
\item If \(k(n-k)\) is even, we have
\[Th(p^*\mathscr{L})\cong\bigoplus_{\Lambda\textrm{ even}}Th(\mathscr{L})((|\Lambda|))\oplus\bigoplus_{\Lambda\textrm{ irred. not full}, i_1>1}R(X)/{\eta}((|\Lambda_{i_1,\cdots,i_l}|+1))\]
in \(\widetilde{DM}(S,R)\).
\item If both \(k\) and \(n\) are even, we have
\[\begin{array}{c}Th(p^*\mathscr{L}\otimes O(1))\cong\bigoplus_{\Lambda=\sigma_{n-k}T}Th(det(\mathscr{E})^{\vee}\otimes\mathscr{L})((|\Lambda|))\oplus\bigoplus_{\Lambda=\sigma_{1^k}T}Th(\mathscr{L})((|\Lambda|))\\\oplus\bigoplus_{i_1>1}R(X)/{\eta}((|\Lambda_{i_1,\cdots,i_l}|+1))\end{array}\]
in \(\widetilde{DM}(S,R)\), where \(T\) is completely even.
\item If \(n-k\) is odd, we have
\[Th(p^*\mathscr{L}\otimes O(1))\cong\bigoplus_{\Lambda\textrm{ even}}Th(\mathscr{L})((|\Lambda|))\oplus\bigoplus_{\Lambda\textrm{ irred. not full}, i_1>1}R(X)/{\eta}((|\Lambda_{i_1,\cdots,i_l}|+1))\]
in \(\widetilde{DM}(S,R)\).
\item If \(k\) and \(n\) are odd, we have
\[Th(p^*\mathscr{L}\otimes O(1))\cong\bigoplus_{\Lambda\textrm{ even}}Th(\mathscr{L}\otimes det(\mathscr{E})^{\vee})((|\Lambda|))\oplus\bigoplus_{\Lambda\textrm{ irred. not full}, i_1>1}R(X)/{\eta}((|\Lambda_{i_1,\cdots,i_l}|+1))\]
in \(\widetilde{DM}(S,R)\).
\item If \(k\) is odd, \(n\) is even and \(e(\mathscr{E})=0\in\widetilde{CH}^n(X,det(\mathscr{E})^{\vee})\), there is an isomorphism
\[\begin{array}{c}Th(p^*\mathscr{L})\cong\bigoplus_{\Lambda=\mathcal{R}\cdot T}Th(\mathscr{L}\otimes det(\mathscr{E})^{\vee})((|\Lambda|))\oplus\bigoplus_{\Lambda=T}Th(\mathscr{L})((|\Lambda|))\\\oplus\bigoplus_{\Lambda\textrm{ irred. not full}, i_1>1}R(X)/{\eta}((|\Lambda_{i_1,\cdots,i_l}|+1))\end{array}\]
in \(\widetilde{DM}(S,R)\), where \(T\) is completely even.
\end{enumerate}
\end{theorem}
\begin{proof}
It suffices to prove the case \(R=\mathbb{Z}\) and \(\mathscr{L}=O_X\). Since \(X\) is quasi-projective, by Jouanolou's trick we may assume \(X\) is affine hence there is an epimorphism \(O_X^{\oplus m}\longrightarrow\mathscr{E}\). So there is a map \(f:X\longrightarrow Gr(n,m)\) such that \(f^*\mathscr{U}_{n,m}=\mathscr{E}\). So we may suppose \(X=S=Gr(n,m)\), in particular, \(\mathbb{Z}(X)\) splits.
\begin{enumerate}
\item By tensoring the formulas in Theorem \ref{Grass1} with \(\mathbb{Z}/\eta\) we obtain an isomorphism
\[\mathbb{Z}(Gr(k,n))/\eta\xrightarrow{t_{\Lambda}}\bigoplus_{\Lambda\textrm{ \((k,n)\)-truncated}}\mathbb{Z}/\eta((|\Lambda|))\]
in \(\widetilde{DM}(pt,\mathbb{Z})\) by \cite[Proposition 5.4]{Y1}. We have
\[[\mathbb{Z}(Gr(k,n))/\eta,\mathbb{Z}/\eta((j))]_{MW}=\eta_{MW}^{j-1}(Gr(k,n))\oplus\eta_{MW}^j(Gr(k,n))\]
by strong duality of \(\mathbb{Z}/\eta\) (see \cite[Proposition 5.8]{Y1}). So by Proposition \ref{global}, we could find
\[(t_{\Lambda})_{\mathscr{E}}\in[\mathbb{Z}(Gr(k,\mathscr{E}))/\eta,\mathbb{Z}/\eta((|\Lambda|))]_{MW}\]
for every \(\Lambda\) such that
\[(t_{\Lambda})_{\mathscr{E}}|_{U_i}=q^*(t_{\Lambda})\]
under the chosen trivialization of \(E|_{U_i}\) where \(q:Gr(k,\mathscr{E}|_{U_i})\longrightarrow Gr(k,n)\) is the projection. Then the map
\[\mathbb{Z}(Gr(k,\mathscr{E}))/\eta\xrightarrow{(t_{\Lambda})_{\mathscr{E}}}\mathbb{Z}/\eta((|\Lambda|))\]
is an isomorphism in \(\widetilde{DM}(X,\mathbb{Z})\) by \cite[Proposition 2.4]{Y1}.
\item We have morphisms
\[\begin{array}{cc}\psi_{\mathscr{E}}(T):\mathbb{Z}(Gr(k,\mathscr{E}))\longrightarrow\mathbb{Z}(X)((|T|))&\varphi_{\mathscr{E}}(\Lambda):\mathbb{Z}(Gr(k,\mathscr{E}))\longrightarrow\mathbb{Z}(X)/\eta((|\Lambda|))\end{array}\]
by Proposition \ref{global}, which gives the ismorphism desired using \cite[Proposition 2.4]{Y1}, where \(T\) is completely even and \(\Lambda\) is not full as in Theorem \ref{Grass1}.
\item We have morphisms
\[\begin{array}{c}\psi_{\mathscr{E}}(\sigma_{1^k}T):Th(O_{Gr(k,\mathscr{E})}(1))\longrightarrow\mathbb{Z}(X)((|\sigma_{1^k}T|+1))\\\varphi_{\mathscr{E}}(\Lambda):Th(O_{Gr(k,\mathscr{E})}(1))\longrightarrow\mathbb{Z}(X)/\eta((|\Lambda|+1))\\\psi_{\mathscr{E}}(\sigma_{n-k}T):Th(O_{Gr(k,\mathscr{E})}(1))\longrightarrow Th(det(\mathscr{E})^{\vee})((|\sigma_{n-k}T|))\end{array}\]
by Proposition \ref{global}, where \(T\) and \(\Lambda\) are same as above. Then proceeds in the same way as (2).
\item We have morphisms
\[\begin{array}{c}\psi_{\mathscr{E}}(\sigma_{1^k}T):Th(O_{Gr(k,\mathscr{E})}(1))\longrightarrow\mathbb{Z}(X)((|\sigma_{1^k}T|+1))\\\varphi_{\mathscr{E}}(\Lambda):Th(O_{Gr(k,\mathscr{E})}(1))\longrightarrow\mathbb{Z}(X)/\eta((|\Lambda|+1))\end{array}\]
by Proposition \ref{global}, where \(T\) and \(\Lambda\) are same as above. Then proceeds in the same way as (2).
\item We have morphisms
\[\begin{array}{c}\psi_{\mathscr{E}}(\sigma_{n-k}T):Th(O_{Gr(k,\mathscr{E})}(1))\longrightarrow Th(det(\mathscr{E})^{\vee})((|\sigma_{n-k}T|))\\\varphi_{\mathscr{E}}(\Lambda):Th(O_{Gr(k,\mathscr{E})}(1))\longrightarrow\mathbb{Z}(X)/\eta((|\Lambda|+1))\end{array}\]
by Proposition \ref{global}, where \(T\) and \(\Lambda\) are same as above. Then proceeds in the same way as (2).
\item We will adopt the notation in Proposition \ref{gen}. We have a distinguished triangle
\[Th(O_{Gr(k-1,\mathscr{E})}(1))\xrightarrow{i^*}Th(O_{Gr(k,\mathscr{E}\oplus O_X)}(1))\longrightarrow\mathbb{Z}(Gr(k,\mathscr{E}))((k+1))\longrightarrow\cdots[1]\]
where \(i:Gr(k-1,\mathscr{E})\longrightarrow Gr(k,\mathscr{E}\oplus O_X)\) is the canonical embedding.

By the discussion above, we have already had the splitting formula for the first two terms, which split by \(X=Gr(n,m)\). Then we could apply the precise classes defined in Proposition \ref{precise}. The idea is to compute \(i^*\) with respect to their splitting and simplify \(i^*\). We will denote by \(T\) a completely even Young diagram.

Suppose \(\Lambda=T\cdot\sigma_{n+1-k}\) is a \((k,n+1)\)-truncated twisted even Young diagram. If \(\Lambda_k=0\), we have
\[i^*(p_{\mathscr{E}}(\Lambda))=p_{\mathscr{E}}(\Lambda)\]
by Proposition \ref{precise}. Let us suppose \(\Lambda_k>0\), so \(T_{k-1}>0\). By Proposition \ref{odd}, we have
\[\sum_ap_{2a}(\mathscr{U}_{\mathscr{E}})t^{2a}\sum_bp_{2b}(\mathscr{U}_{\mathscr{E}}^{\perp})t^{2b}=\sum_cp_{2c}(q^*\mathscr{E})t^{2c}\]
in \(H^{{*}}(Gr(k-1,\mathscr{E}),\textbf{W})[t]\), where \(q\) is the structure map of \(Gr(k-1,\mathscr{E})\). Let \(s=-t^2\). We obtain the following equation by Proposition \ref{inversion}
\[w(p_{\mathscr{E}}(T))\equiv\sum_{C\leq T,C_{k-1}=0,C\textrm{ comp. even}}w(p_{\mathscr{E}}(C))\cdot\mathbb{Z}[p_{2c}(q^*\mathscr{E})]\textrm{ (mod \(e(\mathscr{U}_{\mathscr{E}})\))}.\]
Hence
\[w(p_{\mathscr{E}}(\Lambda))\equiv\sum_{C\leq T,C_{k-1}=0,C\textrm{ comp. even}}w(p_{\mathscr{E}}(C\cdot\sigma_{n-k+1}))\cdot f_C(p_{2c}(q^*\mathscr{E}))\textrm{ (mod \(e(q^*\mathscr{E})\))}.\]
where \(f_C\) is a polynomial with integer coefficients. Hence we could simplify \(i^*\) to some \(\varphi\) via the morphisms
\[Th(det(\mathscr{E}))((|C|+n-k+1))\xrightarrow{f_C}Th(det(\mathscr{E}))((|\Lambda|))\]
for every \(C\) above, so that \(w(p_{\mathscr{E}}(\Lambda)\circ\varphi)\equiv0\textrm{ (mod \(e(q^*\mathscr{E})\))}\). Thus
\[(\pi\circ\gamma)(p_{\mathscr{E}}(\Lambda)\circ\varphi)\in Im(Sq^2_{q^*\mathscr{E}\otimes O(1)})\textrm{ (mod \(e(q^*\mathscr{E})\))}\]
by Theorem \ref{sq}. So \(p_{\mathscr{E}}(\Lambda)\circ\varphi\) is equal, up to a multiple of \(e(q^*\mathscr{E})\), to a composite
\[Th(O_{Gr(k-1,\mathscr{E})}(1))\longrightarrow Th(det(\mathscr{E}))/\eta((|\Lambda|-1))\xrightarrow{\partial}Th(det(\mathscr{E}))((|\Lambda|)),\]
where the first arrow comes from a sum composites of three kinds, according to the decomposition given by (3). Let us discuss case by case:
\begin{enumerate}
\item A composite induced by \(C=T\cdot\sigma_{n-k+1}\).

In this case, the composite factors through \(i^*\) and \(C\neq\Lambda\).
\item A composite induced by \(C=T\cdot\sigma_{1^{k-1}}\).

In this case, \(C\) is an irredundant non-even \((k,n+1)\)-truncated twisted diagram. There is a commutative diagram
\[
	\xymatrix
	{
		Th(O_{Gr(k-1,\mathscr{E})}(1))\ar[r]^{i^*}\ar[d]_{p_{\mathscr{E}}(C)}	&Th(O_{Gr(k,\mathscr{E}\oplus O_X)}(1))\ar[d]_{\varphi_{\mathscr{E}}(C)}\\
		\mathbb{Z}(X)((|C|+1))\ar[r]																&\mathbb{Z}(X)((|C|+1))/\eta
	}
\]
where the lower horizontal arrow is given by the quotient map. Hence the composite factors through \(i^*\) and \(C\neq\Lambda\).
\item A composite induced by \(C\) which is irredundant and not even.

In this case, the composite factors through \(i^*\) and \(C\neq\Lambda\).
\end{enumerate}
Hence we see that \(\varphi\) could be simplified so that
\[p_{\mathscr{E}}(\Lambda)\circ\varphi=0\textrm{ (mod \(e(q^*\mathscr{E})\))}.\]

Suppose \(\Lambda=C_{i_1,\cdots,i_l}, i_1>1\) where \(C\) is a \((k,n+1)\)-truncated irredundant twisted Young diagram. If \(\Lambda_k=0\) and \(|A(\Lambda)|>1\), there is an obvious naturality. If \(\Lambda_k=0\) and \(|A(\Lambda)|=1\), \(\Lambda=T\cdot\sigma_{1^{k-1}}\) is even, regarded as a \((k-1,n)\)-truncated twisted Young diagram. Then \(\varphi_{\mathscr{E}}(\Lambda)\circ\varphi\) is the composite
\[Th(O_{Gr(k-1,\mathscr{E})}(1))\xrightarrow{p_{\mathscr{E}}(\Lambda)}\mathbb{Z}(X)((|\Lambda|+1))\longrightarrow\mathbb{Z}(X)/\eta((|\Lambda|+1))\]
by Proposition \ref{global}. If \(\Lambda_k>0\), we could simplify \(\varphi\) according to the decomposition of \(Th(O_{Gr(k-1,\mathscr{E})}(1))\) as above so that \(\varphi_{\mathscr{E}}(\Lambda)\circ\varphi=0\).

Now pullback everything along \(f\), the \(e(q^*\mathscr{E})\) all vanish. Summarizing the computation above, we see that
\begin{footnotesize}
\[\mathbb{Z}(Gr(k,\mathscr{E}))=\bigoplus_{\Lambda\in S_1}\mathbb{Z}(X)((|\Lambda|-k+1))\oplus\bigoplus_{\Lambda\in S_2}Th(det(\mathscr{E}^{\vee}))((|\Lambda|-k-1))\oplus\bigoplus_{\Lambda\in S_3}\mathbb{Z}(X)/\eta((|\Lambda|-k))\]
\end{footnotesize}
where
\[S_1=\{\sigma_{1^{k-1}}\cdot T|T\textrm{ completely even}\}\]
\[S_2=\{\sigma_{n-k+1}\cdot T|T_{k-1}>0,T\textrm{ completely even}\}\]
\[S_3=\{C_{i_1,\cdots,i_l}|i_1>1,C_k>0,C\textrm{ irredundant and not full}\}.\]
Then we obtain the statement by erasing the first column of every diagram in \(S_1,S_2,S_3\), in order to get bijections between indexes.
\end{enumerate}
Finally we tensor the equations with \(Th(\mathscr{L})\) and apply \(g_{\#}\) where \(g:X\longrightarrow S\) is the structure map.
\end{proof}
We note that the (6) above is the most difficult part of the theorem, which is not a consequence of Proposition \ref{oriinv1}. Because the \(e(\mathscr{E})=0\) is a global condition.

Suppose \(A\in\widetilde{DM}(pt,\mathbb{Z})\) splits. In order to compute \(A\), it suffices to compute its image in \(DM\) and \(\widetilde{DM}_{\eta}\) by Lemma \ref{uniqueness}. For flag varieties \(Gr(d_1,\cdots,d_t)\), their decompositions in \(DM\) are well-known (see \cite[pp. 22]{S}).
\begin{proposition}\label{flag2}
For any odd \(n\) and \(\mathscr{M}\in Pic(Gr(n-2,n-1,n))/2\), \(Th(\mathscr{M})\) splits as an MW-motive and we have
\[Th(\mathscr{M})=\begin{cases}\mathbb{Z}[1]\oplus\mathbb{Z}[2n-2]&\textrm{if }\mathscr{M}=0\\0&\textrm{else}\end{cases}\]
in \(\widetilde{DM}_{\eta}\). So \(Th(\mathscr{M})\) are mutually isomorphic in \(\widetilde{DM}(pt,R)\) if \(\mathscr{M}\neq 0\). Denote by \(G\) this common object.

Suppose that \(S\in Sm/F\), \(X\in Sm/S\) is quasi-projective, \(\mathscr{M}\in Pic(\mathbb{P}(\Omega_p(1)))/2\) and that \(\mathscr{E}\) is a vector bundle of odd rank \(n\) on \(X\). Denote by \(p:\mathbb{P}(\mathscr{E})\longrightarrow X\), \(q:\mathbb{P}(\Omega_p(1))\longrightarrow\mathbb{P}(\mathscr{E})\) the structure maps. We have
\[Th(\mathscr{M})\cong\begin{cases}Th(\mathscr{L})\otimes R(Gr(n-2,n-1,n))&\begin{array}{c}p_{n-1}(\mathscr{E})=0\in H^{2n-2}(X,\textbf{W})\\\mathscr{M}=q^*p^*\mathscr{L},\mathscr{L}\in Pic(X)/2\end{array}\\R(X)\otimes G&\mathscr{M}\notin Pic(X)/2\end{cases}\]
in \(\widetilde{DM}(S,R)\).
\end{proposition}
\begin{proof}
If \(\mathscr{M}\) does not come from \(X\), Proposition \ref{flagthom} shows that we may suppose \(\mathscr{M}=(q^*\mathscr{L})(1)\). Then the statements follow from applying Theorem \ref{Grassbdl} on both \(\mathscr{E}\) and \(\Omega_p(1)\). Otherwise we may suppose \(S=X=Gr(n,m)\) and \(R=\mathbb{Z}\) as before. The Whitney sum formula for even Pontryagin classes tells us that
\[p^*(p_{n-1}(\mathscr{E}))=p_{n-1}(\Omega_p(1))=e(\Omega_p(1))^2\in H^{2n-2}(\mathbb{P}(\mathscr{E}),\textbf{W})\]
hence they all vanish. The composite
\[\mathbb{Z}(\mathbb{P}(\Omega_p(1)))\xrightarrow{\cdot e(\Omega_p(1))}Th(p^*det(\mathscr{E})(1))(n-3)[2n-5]\xrightarrow[\cong]{e(\Omega_p(1))\cdot p^*}\mathbb{Z}(X)(2n-1)[4n-3]\]
is zero by \(e(\Omega_p(1))^2=0\) and the second arrow is an isomorphism after applying \(L\) by Theorem \ref{Grassbdl}, so the first arrow is zero. Hence we have
\[\mathbb{Z}(\mathbb{P}(\Omega_p(1)))\cong Th(p^*det(\mathscr{E})(1))((n-3))\oplus\mathbb{Z}(\mathbb{P}(\mathscr{E}))\oplus\bigoplus_{i=1}^{\frac{n-3}{2}}\mathbb{Z}(\mathbb{P}(\mathscr{E}))/\eta((2i-1))\]
by \cite[Corollary 5.15]{Y1}. Then the statement follows from Theorem \ref{Grassbdl}.
\end{proof}
\begin{theorem}\label{flag}
Suppose \(\mathscr{L}\in Pic(Gr(1,\cdots,n))/2\). The \(Th(\mathscr{L})\) splits as an MW-motive. Moreover, we have
\begin{enumerate}
\item If \(n\) is odd, define \(deg(a)=4a-1\). We have
\[Th(\mathscr{L})=\begin{cases}\bigoplus_{1\leq t\leq\frac{n-1}{2}}\bigoplus_{1\leq a_1<\cdots<a_t\leq\frac{n-1}{2}}\mathbb{Z}[1+\sum_sdeg(a_s)]&\textrm{if }\mathscr{L}=0\\0&\textrm{else}\end{cases}.\]
in \(\widetilde{DM}_{\eta}\).
\item If \(n\) is even, define
\[deg(a)=\begin{cases}4a-1&1\leq a\leq\frac{n}{2}-1\\n-1&a=\frac{n}{2}\end{cases}.\]
We have
\[Th(\mathscr{L})=\begin{cases}\bigoplus_{1\leq t\leq\frac{n}{2}}\bigoplus_{1\leq a_1<\cdots<a_t\leq\frac{n}{2}}\mathbb{Z}[1+\sum_sdeg(a_s)]&\textrm{if }\mathscr{L}=0\\0&\textrm{else}\end{cases}.\]
in \(\widetilde{DM}_{\eta}\).
\end{enumerate}
So \(Th(\mathscr{L})\) are mutually isomorphic in \(\widetilde{DM}(pt,R)\) if \(\mathscr{L}\neq 0\). Denote by \(G\) this common object.

Suppose \(X\in Sm/S\) is quasi-projective, \(\mathscr{M}\in Pic(Fl(\mathscr{E}))/2\) and \(\mathscr{E}\) is a vector bundle of rank \(n\) on \(X\). Denote by \(p:Fl(\mathscr{E})\longrightarrow X\) the structure map. We have
\[Th(\mathscr{M})\cong\begin{cases}Th(\mathscr{L})\otimes R(Gr(1,\cdots,n))&\begin{array}{c}p_i(\mathscr{E}),e(\mathscr{E})=0\in H^{{*}}(X,\textbf{W}(-))\\\textrm{for any }i>0\\\mathscr{M}=p^*\mathscr{L},\mathscr{L}\in Pic(X)/2\end{array}\\R(X)\otimes G&\mathscr{M}\notin Pic(X)/2\end{cases}\]
in \(\widetilde{DM}(S,R)\).
\end{theorem}
\begin{proof}
In the context of Definition \ref{flagdef}, we have
\[Y_i=Gr(n-i,n-i+1,\Omega_{p_{i-2}}).\]
Then the statements follow from inductively using Proposition \ref{flag2} and \cite[Corollary 5.15]{Y1}.
\end{proof}
{}
\end{document}